\documentclass[11pt]{article}

\usepackage[T1]{fontenc}
\usepackage[utf8]{inputenc}
\usepackage{newpxtext}      
\usepackage{iftex}          
\ifPDFTeX                   
  \usepackage[activate={true,nocompatibility},final]{microtype}
\else                       
  \usepackage[protrusion=true,final]{microtype}
\fi

\usepackage[margin=1in]{geometry}

\usepackage{authblk}                    

\usepackage{amsmath}
\usepackage{amssymb,amsfonts}
\usepackage{newpxmath}      
\usepackage{mathtools}      
\usepackage{bm}             
\usepackage{nicefrac}
\allowdisplaybreaks         

\usepackage{amsthm}

\usepackage{graphicx}
\graphicspath{{figures/}{./}}                       
\usepackage{booktabs}
\usepackage{longtable}                              
\usepackage{multirow,makecell,array}
\usepackage[font=small,labelfont=bf]{caption}       
\usepackage{subcaption}                             
\usepackage{algorithm}
\usepackage{algpseudocode}                          

\usepackage{enumitem}
\setlist[itemize]{leftmargin=2.2em,itemsep=2pt,topsep=2pt}
\setlist[enumerate]{leftmargin=2.2em,itemsep=2pt,topsep=2pt}

\usepackage{xcolor}
\definecolor{LinkColor}{rgb}{0.10,0.40,0.75}        
\definecolor{CiteColor}{rgb}{0.70,0.25,0.20}        
\definecolor{UrlColor} {rgb}{0.22,0.52,0.34}        
\definecolor{TodoColor}{rgb}{0.80,0.30,0.10}        
\usepackage{colortbl}                               
\definecolor{OurRow}{gray}{0.92}                    

\usepackage{tikz}
\usetikzlibrary{positioning,calc,arrows.meta,fit}

\usepackage[round,sort&compress]{natbib}

\usepackage{url}
\usepackage{hyperref}
\hypersetup{
  colorlinks=true,
  linkcolor=LinkColor,
  citecolor=CiteColor,
  urlcolor=UrlColor,
  breaklinks=true,
  bookmarksnumbered=true,
}
\usepackage{bookmark}                               

\numberwithin{equation}{section}

\newif\ifdraft \draftfalse
\ifdraft
  \usepackage{lineno}\linenumbers
  \newcommand{\todo}[1]{\textcolor{TodoColor}{\textbf{[TODO:}~#1\textbf{]}}}
\else
  \newcommand{\todo}[1]{}                            
\fi

\newcommand{\R}{\mathbb{R}}

\newcommand{\Q}{\mathbb{Q}}
\newcommand{\E}{\mathbb{E}}                       
\let\C\relax \newcommand{\C}{\mathbb{C}}          
\renewcommand{\P}{\mathbb{P}}                     

\DeclareMathOperator{\Var}{Var}

\DeclarePairedDelimiterX{\inner}[2]{\langle}{\rangle}{#1,#2}   

\newcommand{\LemmaForge}{\textsc{LemmaForge}}     
\newcommand{\Ckap}{\mathcal{C}(\kappa)}           
\newcommand{\CkapOf}[1]{\mathcal{C}(#1)}
\newcommand{\Vone}{V_{1}}                         
\newcommand{\Vtwo}{V_{2}}                         
\newcommand{\kc}{\kappa_{\mathrm{c}}}             
\newcommand{\kss}{\kappa^{**}}                    
\newcommand{\ks}{\kappa^{*}}                      
\newcommand{\PDeg}{\mathrm{PD}}                   

\newcommand{\tauofk}{\tau(\kappa)}

\usepackage{aliascnt}
\newcommand{\newshared}[2]{\newaliascnt{#1}{theorem}\newtheorem{#1}[#1]{#2}\aliascntresetthe{#1}}

\theoremstyle{plain}
\newtheorem{theorem}{Theorem}[section]
\newshared{proposition}{Proposition}
\newshared{lemma}{Lemma}
\newshared{corollary}{Corollary}
\newshared{conjecture}{Conjecture}
\newshared{fact}{Fact}
\newshared{mvfact}{Machine-Verified Fact}

\theoremstyle{definition}
\newshared{definition}{Definition}
\newshared{assumption}{Assumption}
\newshared{example}{Example}
\newshared{problem}{Problem}

\theoremstyle{remark}
\newshared{remark}{Remark}
\newshared{claim}{Claim}

\usepackage[capitalise,nameinlink,noabbrev,sort&compress]{cleveref}  

\crefname{theorem}{Theorem}{Theorems}          \Crefname{theorem}{Theorem}{Theorems}
\crefname{proposition}{Proposition}{Propositions}
\Crefname{proposition}{Proposition}{Propositions}
\crefname{lemma}{Lemma}{Lemmas}                \Crefname{lemma}{Lemma}{Lemmas}
\crefname{corollary}{Corollary}{Corollaries}   \Crefname{corollary}{Corollary}{Corollaries}
\crefname{conjecture}{Conjecture}{Conjectures} \Crefname{conjecture}{Conjecture}{Conjectures}
\crefname{fact}{Fact}{Facts}                   \Crefname{fact}{Fact}{Facts}
\crefname{mvfact}{Machine-Verified Fact}{Machine-Verified Facts}
\Crefname{mvfact}{Machine-Verified Fact}{Machine-Verified Facts}
\crefname{definition}{Definition}{Definitions} \Crefname{definition}{Definition}{Definitions}
\crefname{assumption}{Assumption}{Assumptions} \Crefname{assumption}{Assumption}{Assumptions}
\crefname{example}{Example}{Examples}          \Crefname{example}{Example}{Examples}
\crefname{problem}{Problem}{Problems}          \Crefname{problem}{Problem}{Problems}
\crefname{remark}{Remark}{Remarks}             \Crefname{remark}{Remark}{Remarks}
\crefname{claim}{Claim}{Claims}                \Crefname{claim}{Claim}{Claims}
\crefname{algorithm}{Algorithm}{Algorithms}    \Crefname{algorithm}{Algorithm}{Algorithms}

\newcommand{\restatetitle}{}
\newtheorem*{theoremrestated}{\restatetitle}
\newcommand{\RestateTheorem}[3]{%
  \renewcommand{\restatetitle}{Theorem~\ref{#1} (#2, restated)}%
  \begin{theoremrestated}#3\end{theoremrestated}}
\newcommand{\RestateProposition}[3]{%
  \renewcommand{\restatetitle}{Proposition~\ref{#1} (#2, restated)}%
  \begin{theoremrestated}#3\end{theoremrestated}}
\newcommand{\RestateCorollary}[3]{%
  \renewcommand{\restatetitle}{Corollary~\ref{#1} (#2, restated)}%
  \begin{theoremrestated}#3\end{theoremrestated}}
\newcommand{\RestateLemma}[3]{%
  \renewcommand{\restatetitle}{Lemma~\ref{#1} (#2, restated)}%
  \begin{theoremrestated}#3\end{theoremrestated}}

\title{The Exact Worst-Case Tail Probability under Bounded Kurtosis}

\author[1]{Xiaoyu Li\thanks{\texttt{xiaoyu.li2@unsw.edu.au}}}
\author[2]{Andi Han\thanks{\texttt{andi.han@sydney.edu.au}}}
\author[1]{Jiaojiao Jiang\thanks{\texttt{jiaojiao.jiang@unsw.edu.au}}}
\author[2]{Junbin Gao\thanks{\texttt{junbin.gao@sydney.edu.au}}}
\affil[1]{University of New South Wales\qquad\textsuperscript{2}University of Sydney}
\date{\today}

\hypersetup{pdftitle={The Exact Worst-Case Tail Probability under Bounded Kurtosis},
            pdfauthor={Xiaoyu Li, Andi Han, Jiaojiao Jiang, Junbin Gao}}

\begin{document}
\begingroup
\renewcommand\thefootnote{\fnsymbol{footnote}}
\maketitle
\endgroup

\begin{abstract}
We determine exactly what a kurtosis bound buys for one-sided tail control. For the class
$\Ckap$ of real random variables with mean $0$, variance $1$, and fourth moment at most
$\kappa$, the skewness left free (the form in which fourth-moment information typically
arrives in learning theory and robust statistics), we compute the worst-case tail
probability $\Vone(t,\kappa)=\sup_{X\in\Ckap}\P(X\geq t)$ for every threshold $t>0$ and
every $\kappa\geq 1$. The answer is a four-regime map: a \emph{Cantelli tongue}
$b(\kappa)\le t\le c(\kappa)$ on which the two-moment bound $1/(1+t^2)$ remains tight and
the kurtosis constraint is worthless; a tail regime $t\geq c(\kappa)$ with the closed form
$\Vone=(\kappa-1)/((t^2-1)^2+\kappa-1)$; a plateau regime, present only for
$\kappa\le 3/2$, on which the worst case freezes and the value does not depend on $t$; and
a central regime described exactly by an explicit algebraic system, provably admitting no
closed form in nested square roots. Beyond $c(\kappa)$ the one-sided and two-sided worst
cases \emph{coincide}: Cantelli's improvement over Chebyshev is annihilated by
fourth-moment information. The minimal degree of a sum-of-squares proof of the tight bound
is $2$ on the closed tongue and $4$ everywhere else, an exact phase diagram of proof
degree. Every closed-form regime carries an explicit dual certificate and an explicit
extremal distribution, re-verified on parameter grids by an independent checker in exact
arithmetic. The closed forms invert to exact worst-case quantiles, sharpen a
median-of-means constant, and give the exact per-direction tail available to degree-$4$
reasoning under certifiable kurtosis. We found the map through an AI-guided search around the certifying
pipeline, \LemmaForge, which is validated on classical benchmarks, independently
reproduces the symmetric-slice bound of \citet{zelen1954}, and recovers the
$2\sqrt3-3$ constant of \citet{hezhangzhang2010} at $t=0$. Code, certificates, and the
full instance tables are available at
\href{https://github.com/xiaoyulics/lemmaforge}{\nolinkurl{https://github.com/xiaoyulics/lemmaforge}}.

\end{abstract}

\newpage
{\hypersetup{linkcolor=black}%
\makeatletter
\renewcommand*\l@section[2]{%
  \ifnum \c@tocdepth >\z@
    \addpenalty\@secpenalty
    \addvspace{0.25em \@plus\p@}%
    \setlength\@tempdima{1.5em}%
    \begingroup
      \parindent \z@ \rightskip \@pnumwidth
      \parfillskip -\@pnumwidth
      \leavevmode \bfseries
      \advance\leftskip\@tempdima
      \hskip -\leftskip
      #1\nobreak\hfil\nobreak\hb@xt@\@pnumwidth{\hss #2}\par
    \endgroup
  \fi}
\makeatother
\tableofcontents}
\newpage

\section{Introduction}
\label{sec:intro}

Two moments buy the Chebyshev--Cantelli inequalities. If $X$ has mean $0$ and variance $1$,
then $\P(X\geq t)\leq 1/(1+t^2)$ for every $t>0$, this is tight, and the worst case is a
two-point distribution \citep{cantelli1928}. The natural next unit of information, and the
one that actually appears in the assumptions of modern learning theory and robust statistics,
is the fourth moment. Bounded-kurtosis classes underlie outlier-robust moment estimation
and the certifiable-subgaussianity program, where fourth-moment control at sum-of-squares
degree $4$ is a central currency of algorithmic estimators
\citep{kotharisteurer2017,hopkinsli2018}; the present paper works in one variable, where the
question can be answered \emph{exactly}: the base case of that program, not yet its
algorithmic payoff, and we say so plainly. We therefore consider the bounded-kurtosis class
\[
\Ckap \;=\; \bigl\{\,X \text{ real-valued}\;:\; \E X = 0,\; \E X^2 = 1,\; \E X^4 \leq \kappa \,\bigr\},
\qquad \kappa \geq 1,
\]
and we ask the sharp question:
\begin{center}
\emph{what exactly is the worst-case tail probability
$\Vone(t,\kappa) = \sup\limits_{X\in\Ckap}\,\P(X\geq t)$?}
\end{center}
Note what the class does \emph{not}
constrain: the third moment $m_3\coloneqq\E X^3$ is left free. This is deliberate, and it is the form in which the
assumption arrives in applications: a kurtosis bound is routinely available or imposed,
while the skewness is unknown.

The corresponding question with \emph{all} moments up to order four pinned is classical. The
reduced moment problem of Chebyshev, Markov, and Posse was made explicit for four moments by
\citet{royden1953} and \citet{zelen1954}, with the general theory of principal representations
in \citet{karlinstudden1966} and the classical inequality line surveyed by \citet{savage1961}
(\citealp{selberg1940} is the two-moment two-sided refinement; \citealp{guttman1948} an early
kurtosis inequality); the modern semidefinite-optimization treatment of
moment-constrained probability bounds is due to \citet{bertsimaspopescu2005}, whose explicit
closed forms stop at three moments; and \citet{hezhangzhang2010} solved a fourth-moment
one-sided problem in the small-deviation regime by conic duality. But these works either pin
the skewness (Zelen's bounds are multi-branch case analyses over all four moment values)
or cover one corner of the parameter space (\cref{tab:closest}). For the kurtosis-only
class $\Ckap$, to our
knowledge, no complete description of $\Vone$ has appeared. This paper closes the question,
and does so in a form where every constant is not merely stated but \emph{certified}: each
regime of the answer ships with an explicit polynomial certificate --- a sum-of-squares proof
a reader can check line by line, which an independent verifier additionally re-checks in exact
rational arithmetic --- together with the extremal distribution that attains it.

\begin{figure}[t]
\centering
\includegraphics[width=0.86\linewidth]{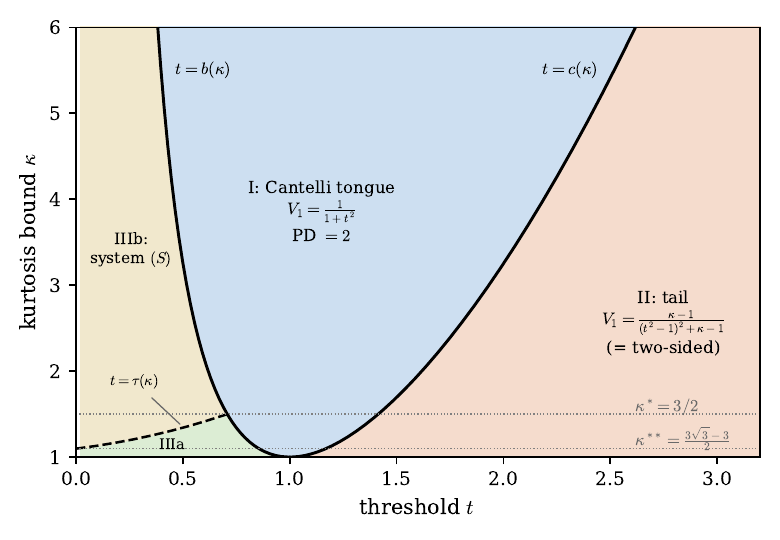}
\caption{The one-sided map (\cref{thm:main-map}). Inside the tongue
$b(\kappa)\le t\le c(\kappa)$ (equivalently $\kappa\ge\kc(t)$), Cantelli's two-moment bound
is tight and the proof degree is $2$. To the right of $t=c(\kappa)$ the tail formula of
\eqref{eq:regimeII} is tight, one-sided and two-sided worst cases coincide
(\cref{cor:collapse}), and the proof degree is $4$. Left of the tongue: the plateau IIIa
(value independent of $t$; only for $\kappa\le\ks=3/2$) and the central regime IIIb
(explicit algebraic system; \cref{sec:endpoint}). The dashed curve is $t=\tauofk$; the two
marked constants are $\kss=(3\sqrt3-3)/2$, where $\tau(\kss)=0$, and $\ks=3/2$, where
$\tau=b=1/\sqrt2$.}
\label{fig:phase}
\end{figure}

\begin{table}[t]
\centering\small
\renewcommand{\arraystretch}{1.2}%
\begin{tabular}{>{\raggedright\arraybackslash}p{0.24\dimexpr\linewidth-8\tabcolsep\relax}%
                >{\raggedright\arraybackslash}p{0.26\dimexpr\linewidth-8\tabcolsep\relax}%
                >{\raggedright\arraybackslash}p{0.32\dimexpr\linewidth-8\tabcolsep\relax}%
                >{\raggedright\arraybackslash}p{0.18\dimexpr\linewidth-8\tabcolsep\relax}}
\toprule
work & class & answer & at $(t,\kappa)=(2,3)$ \\
\midrule
\citeauthor{bienayme1853}~(\citeyear{bienayme1853})--\citeauthor{chebyshev1867}~(\citeyear{chebyshev1867})
  & $\E X{=}0$, $\E X^2{=}1$
  & $\P(|X|\ge t)\le\min(1,1/t^2)$; sharp & $1/4$ \\
\citet{cantelli1928}
  & $\E X{=}0$, $\E X^2{=}1$
  & $\P(X\ge t)\le 1/(1+t^2)$; sharp & $1/5$ \\
\mbox{Markov applied to $X^4$}
  & $\E X{=}0$, $\E X^2{=}1$, $\E X^4\le\kappa$
  & $\P(|X|\ge t)\le\kappa/t^4$; loose by the factor $\kappa/(\kappa-1)$ in the tail
  & $3/16$ \\
\citet{zelen1954}
  & all moments to order $4$ pinned ($m_3$ included)
  & sharp multi-branch closed forms where each branch applies & $2/19$ \\
\citet{hezhangzhang2010}
  & $\E X$, $\E X^2$, $\E X^4$ pinned ($m_3$ free)
  & sharp near the mean; the $t{=}0$ constant $2\sqrt3-3$ & --- \\
\rowcolor{OurRow}
this paper
  & $\E X{=}0$, $\E X^2{=}1$, $\E X^4\le\kappa$ ($m_3$ free)
  & $\Vone(t,\kappa)$ exactly, for every $(t,\kappa)$: the map of \cref{thm:main-map}
  & $\mathbf{2/11}$ \\
\bottomrule
\end{tabular}
\caption{The nearest classes at a glance; \cref{app:priorwork} gives the full landscape
with sources. The last column is each row's bound on $\P(X\ge2)$ at
$(t,\kappa)=(2,3)$, with Zelen evaluated on his $m_3=0$ branch. Values are not
comparable down the column: each is sharp (or not) for its own class. Zelen's $2/19$
undercuts our $2/11$ only because his class pins the skewness, and the gap is exactly
the price of not knowing it (\cref{app:priorwork}).}
\label{tab:closest}
\end{table}

\paragraph{The map.} The answer is a four-regime phase diagram over $(t,\kappa)$
(\cref{fig:phase}), with explicit boundary curves
\[
b(\kappa)=\tfrac{1}{2}\bigl(\sqrt{\kappa+3}-\sqrt{\kappa-1}\bigr),\qquad
c(\kappa)=\tfrac{1}{2}\bigl(\sqrt{\kappa+3}+\sqrt{\kappa-1}\bigr)=\frac{1}{b(\kappa)},\qquad
\kc(t)=t^2-1+\frac{1}{t^2},
\]
and two structural constants $\kss=\tfrac{3\sqrt3-3}{2}\approx 1.0981$ and $\ks=\tfrac32$.
\begin{itemize}
\item \textbf{Regime I (Cantelli tongue), $b(\kappa)\le t\le c(\kappa)$, equivalently
  $\kappa\ge\kc(t)$:} the kurtosis constraint is \emph{worthless};
  \[
  \Vone(t,\kappa)\;=\;\frac{1}{1+t^2},
  \]
  and the Cantelli two-point worst case survives because its own fourth moment, $\kc(t)$,
  happens to satisfy the constraint. The tongue closes at the single point $(t,\kappa)=(1,1)$.
\item \textbf{Regime II (tail), $t\ge c(\kappa)$:} the closed form
  \[
  \Vone(t,\kappa)\;=\;\frac{\kappa-1}{(t^2-1)^2+\kappa-1},
  \]
  attained by a skewed three-point distribution $\{t, a, -a\}$. As $t\to\infty$ this improves
  the trivial fourth-moment Markov bound $\kappa/t^4$ by exactly the factor
  $(\kappa-1)/\kappa$.
\item \textbf{Regime IIIa (plateau), $\kappa\le\ks$ and $\max(\tauofk,0)\le t\le b(\kappa)$ (with $t>0$):}
  \[
  \Vone(t,\kappa)\;=\;\frac12\biggl(1+\sqrt{\frac{\kappa-1}{\kappa+3}}\,\biggr),
  \]
  \emph{independent of $t$}: the worst case freezes at a two-point distribution that no
  longer needs an atom at the threshold.
\item \textbf{Regime IIIb (central), the remaining wedge:} the value is characterized exactly
  by an explicit algebraic system (system (S), \cref{sec:endpoint}); no closed form of the
  preceding kind exists: at $(t,\kappa)=(\tfrac12,2)$ its minimal polynomial over $\Q$
  already has degree six. At the endpoint $t=0$ (for $\kappa\ge\kss$) the value is
  $1-(2\sqrt3-3)/\kappa$, recovering the constant of \citet{hezhangzhang2010}.
\end{itemize}

\paragraph{Three consequences.} First, a collapse: on all of Regime II the one-sided and
two-sided worst cases coincide, $\Vone(t,\kappa)=\Vtwo(t,\kappa)$
(\cref{thm:twosided,cor:collapse}). At the two-moment level, one-sidedness buys the strict
Cantelli improvement $1/(1+t^2)$ versus $1/t^2$; once a kurtosis bound is added and
$t\ge c(\kappa)$, that gain is annihilated: the skewed three-point worst case for the
one-sided event is matched exactly by a symmetric four-point worst case for the two-sided
event. We are not aware of a prior observation of this phenomenon. Practically: any
analysis that pays for one-sided control (a Cantelli-type step) while also assuming a
kurtosis bound is, in the regime $t\ge c(\kappa)$, paying for nothing: the two-sided
bound already \emph{is} the one-sided truth, and symmetrization arguments lose nothing
there. Second, a proof-complexity
phase transition (\cref{thm:proofdegree}): the tight bound admits a degree-$2$ sum-of-squares
proof exactly on the Cantelli tongue, and requires (and admits) degree $4$ everywhere else;
the phase boundary of proof degree is the explicit algebraic curve $\kappa=\kc(t)$. We view
this as the simplest nontrivial instance of a general program, \emph{proof degree as a
computable coordinate of probability inequalities}, in the spirit of the
proofs-to-algorithms correspondence \citep{hopkinsli2018}. Third, everything is constructive
and machine-checked: the extremal distributions form a small gallery of two- and three-point
laws, the certificates are explicit polynomial identities with rational or quadratic-surd
data, and an independent verifier (exact arithmetic, no solver code) re-checks every
certified grid instance in a fresh process.

\paragraph{What exactness buys downstream.} Each closed form transfers
(\cref{sec:consequences}): inverted, the tail regime yields worst-case quantiles in
closed form (at $\kappa=3$ the one-sided $99\%$ quantile falls from Cantelli's
$9.95\sigma$ to $3.88\sigma$); composed with the fourth-moment identity for averages, it
sharpens the median-of-means per-block constant from Chebyshev's $1/4$ to $2/11$; read
on a classification margin, it gives an exact margin-to-error bound, quartic beyond
$c(\kappa)$; and under the certifiable-kurtosis hypothesis of sum-of-squares estimation,
$\Vone$ is the exact per-direction tail constant, unimprovable from that hypothesis
alone.

\paragraph{The instrument.} We found the results through an AI-guided search organized
around \LemmaForge, a pipeline that compiles moment problems into semidefinite programs via
the exact univariate sum-of-squares representations (Markov--Luk\'acs;
\citealp{nesterov2000}), solves them numerically, recognizes exact values from the floats
(PSLQ and sequence fitting), rounds numerical dual solutions to exact rational certificates in
the manner of \citet{peyrlparrilo2008}, and verifies every certificate with an independent
exact-arithmetic checker. The pipeline is validated against a benchmark battery whose answers
are classical: Markov, Cantelli, three-moment bounds on the half-line
\citep{bertsimaspopescu2005}, and the Khintchine $p=4$ values $3-2/n$, which the recognition
stage outputs autonomously from the numerics for $n=2,\dots,8$ \citep[cf.][]{haagerup1981}.
One further benchmark comes out sharper than the textbook: the pipeline computes the
\emph{tight} Paley--Zygmund constant $(1-\theta)^2/\bigl((1-\theta)^2+m_2-1\bigr)$ for the
class $\{X\ge 0,\ \E X=1,\ \E X^2=m_2\}$, strictly better than the textbook
$(1-\theta)^2/m_2$, an instructive reminder that classical inequalities are not always
stated in their extremal form. On the symmetric ($\E X^3=0$-pinned) slice, the pipeline
independently reproduces Zelen's 1954 bound wherever his formula applies, and locates its
regime boundary.

\paragraph{Related work.} Our Regime II formula is algebraically different from Zelen's
symmetric-case bound $(\kappa-1)/\bigl((1+t^2)(\kappa-1)+(t^2-1)^2\bigr)$
\citep{zelen1954,maher2019}, and must be: Zelen pins $\E X^3=0$ while we leave it free, and
our extremal three-point law is genuinely skewed. Neither bound is comparable to the other as
a theorem; quantitatively, at $(t,\kappa)=(2,3)$ the skewness-free worst case is $2/11$
against $2/19$ for the symmetric class: the exact price of not knowing the skewness.
\citet{bhattacharyya1987} proposed a one-sided four-moment inequality; \citet{maher2019}
showed it is not sharp; our map supplies the sharp values it should be compared against.
\citet{hezhangzhang2010} treat, by the same primal--dual method we use, the fourth-moment
one-sided problem in the small-deviation regime (skewness likewise free), with tight bounds
near the mean; our $t=0$ endpoint recovers their constant, and our map embeds their regime
as the small-$t$ edge of a global picture. In the same line, \citet{garnett2020} proves the
sharp $t=0$ bound $1-1/(2\kappa)$ under the additional sign constraint $\E X^3\ge 0$; the
gap against our free-skewness endpoint $1-(2\sqrt3-3)/\kappa$ is exactly the price of not
knowing the skewness sign at the mean. The rational-certificate toolchain builds
on \citet{peyrlparrilo2008} (see also \citealp{magron2018,cifuentes2020,harrison2007}); the moment--SOS
duality framework is that of \citet{bertsimaspopescu2005} and \citet{lasserre2010}; univariate
exactness of the degree-$4$ relaxation is the Markov--Luk\'acs theorem in the form of
\citet{nesterov2000}. Proof-degree lower bounds in the SOS literature concern combinatorial
problems (e.g., \citealp{grigoriev2001}); we are not aware of a prior exact proof-degree
phase diagram for a probability inequality.

The optimization view of Chebyshev-type inequalities also has a longer line of its own:
the duality theory of generalized moment problems goes back to \citet{isii1962};
\citet{smith1995} developed generalized Chebyshev inequalities for decision analysis;
\citet{popescu2005} and \citet{vandenberghe2007} compute sharp moment bounds by
semidefinite programming in richer, often multivariate settings; and moment ambiguity
sets are the backbone of distributionally robust optimization
\citep{delageye2010,wiesemann2014}. Our map can be read as the complete, closed-form
solution of the simplest genuinely fourth-moment ambiguity set. On the statistical side,
finite-kurtosis conditions of exactly our form are the standing assumption of
heavy-tailed mean estimation \citep{catoni2012,lugosimendelson2019}. For the
classical moment-problem monographs, see \citet{karlinstudden1966} and
\citet{kreinnudelman1977}; for exact rational certification in polynomial
optimization beyond \citet{peyrlparrilo2008}, see \citet{kaltofen2012}.

For the sum-of-squares method at large, see \citet{barak2014} and \citet{fleming2019};
within learning theory and statistics it has become a central algorithmic tool, whose
estimation applications typically run on certifiable moment hypotheses, fourth-moment
conditions chief among them. A partial list of the problems attacked this way: dictionary
learning and tensor decomposition
\citep{barakkelnersteurer2015,hopkinsshisteurer2015,mashisteurer2016}; robust moment
estimation, clustering, and mixture learning
\citep{kotharisteurer2017,hopkinsli2018,kss2018moments,bakshi2020clustering,bakshi2022mixtures};
outlier-robust regression \citep{klivanskotharimeka2018}; heavy-tailed mean estimation
at sub-Gaussian rates \citep{hopkins2020mean}; robust sparse mean estimation
\citep{dkkpp2022sparse}; certification of sparse singular values
\citep{dhpt2025sparse}; and the certifiability of subgaussianity itself
\citep{dhpt2025subgaussian}. See \citet{rss2018survey} for a survey of the method as an
estimation tool and \citet{diakonikolaskane2023} for the robust-statistics monograph;
on the flip side, SOS-degree lower bounds delimit what low-degree certificates can
detect or remove \citep{hkprss2017,li2026sos}. The map of this paper sits at the base of
this edifice: one variable and degree four, where the relaxation is exact and every
certificate is explicit.

\paragraph{Contributions.}
\begin{itemize}
  \item \textbf{(Complete map.)} \cref{thm:main-map}: $\Vone(t,\kappa)$ in closed form on
    three regimes and as an explicit algebraic system on the fourth, with extremal
    distributions and self-contained certificate proofs (\cref{app:proofs}).
  \item \textbf{(Collapse.)} \cref{thm:twosided,cor:collapse}: the exact two-sided map and
    the one-sided $=$ two-sided collapse on $t\ge c(\kappa)$.
  \item \textbf{(Proof-degree phase diagram.)} \cref{thm:proofdegree}: minimal SOS proof
    degree $2$ on the tongue, $4$ off it; the transition curve is exactly $\kappa=\kc(t)$.
  \item \textbf{(Certified instrument.)} \LemmaForge: a benchmark battery, exact certificates
    verified by an independent checker in fresh processes, an extremal-distribution gallery,
    and a machine-readable table of certified constants (\cref{sec:pipeline},
    \cref{app:pipeline}).
  \item \textbf{(Applications.)} \cref{sec:consequences}: exact worst-case quantiles
    and confidence intervals (closed-form fourth-moment value-at-risk), a sharpened
    median-of-means constant, an exact margin-to-error bound, and the exact directional
    tail constant available to degree-$4$ reasoning under certifiable kurtosis, with a
    matching witness.
\end{itemize}

\paragraph{Organization.} \cref{sec:prelim} sets up the moment problem, certificates, and
proof degree, and records the five-line weak-duality lemma all upper bounds rest on.
\cref{sec:results} states the map, the two-sided results, the proof-degree theorem, and the
endpoint proposition, each followed by its regime discussion. \cref{sec:pipeline} describes
the pipeline and the certification methodology. \cref{sec:consequences} records four
applications: exact worst-case quantiles and confidence intervals, a sharpened
median-of-means constant, an exact margin-to-error bound, and exact directional tails
under certifiable kurtosis.
\cref{sec:discussion} discusses limitations and the program this opens. The appendix opens with a symbol glossary and the dependency
graph of the paper's results (\cref{app:glossary}), then a self-contained primer on the
moment problem, sums of squares, and semidefinite programming (\cref{app:background}), and
the exhaustive register of every external and auxiliary fact used (\cref{app:register}).
Full proofs are in \cref{app:proofs}; \cref{app:worked} works three instances
of the map in explicit numbers; \cref{app:priorwork} tabulates the prior-work classes and
formulas; \cref{app:pipeline} documents the certificate format, the verifier's checks, a
verbatim certificate with its verification transcript, computational provenance, and
reproduction instructions; and \cref{app:tables} lists every certified instance.

\section{Setup: moment classes, certificates, and proof degree}
\label{sec:prelim}

\paragraph{The class and the two values.} Throughout, $X$ is a real-valued random variable.
We standardize the first two moments; the general case follows by the affine substitution
$X\mapsto (X-\mu)/\sigma$, under which $\kappa$ below is the kurtosis
$\E(X-\mu)^4/\sigma^4$ and the threshold is measured in standard deviations.

\begin{assumption}[The bounded-kurtosis class]\label{asm:class}
Fix $\kappa\geq 1$. The class $\Ckap$ consists of all laws of real random variables $X$ with
\textup{(A1)} $\E X = 0$, \textup{(A2)} $\E X^2 = 1$, and \textup{(A3)} $\E X^4 \leq \kappa$.
\end{assumption}

Each hypothesis is interpretable: (A1)--(A2) fix location and scale; (A3) is the only
substantive constraint, an upper bound on the fourth moment. By Jensen's inequality
$\E X^4\ge(\E X^2)^2=1$, so $\kappa\ge 1$ is the full range, and $\CkapOf{1}$ contains only
the symmetric two-point law on $\{\pm1\}$ (equality in Jensen forces $X^2$ a.s.\ constant
$=1$; mean zero then forces equal masses). The third moment is deliberately unconstrained.
For $t>0$ we study
\[
\Vone(t,\kappa)\;=\;\sup_{X\in\Ckap}\,\P(X\geq t),
\qquad
\Vtwo(t,\kappa)\;=\;\sup_{X\in\Ckap}\,\P(|X|\geq t).
\]
Both events are closed; atoms placed exactly at $t$ count toward the probability, and several
extremal distributions below exploit this.

\paragraph{Certificates.} Every upper bound we state is proved by a polynomial majorant of an
indicator, evaluated through the moment constraints: the mechanism behind Chebyshev's and
Cantelli's classical proofs, organized so that it can be checked mechanically. (Readers new
to the moment-problem/SOS/SDP dictionary will find a self-contained primer, run end to end
on Cantelli's inequality, in \cref{app:background}.)

\begin{definition}[SOS tail certificate]\label{def:cert}
Let $S=[t,\infty)$ (one-sided) or $S=\{|x|\ge t\}$ (two-sided). A \emph{certificate}
for the bound $\P(X\in S)\le B$ over $\Ckap$ is a polynomial in the moment monomials of
$\Ckap$,
\[
q(x)\;=\;\lambda_0 + \lambda_1 x + \lambda_2 x^2 + \lambda_4 x^4,
\qquad \lambda_4\ge 0,
\]
together with sum-of-squares data witnessing the two pointwise inequalities
\begin{align}
q(x) &\;\ge\; 0 \quad\text{for all } x\in\R, \label{eq:cert-glob}\\
q(x) &\;\ge\; 1 \quad\text{for all } x\in S, \label{eq:cert-event}
\end{align}
such that $B=\lambda_0+\lambda_2+\lambda_4\kappa$. We call the certificate
\emph{degree-$4$} when $\lambda_4>0$ and \emph{degree-$2$} when $\lambda_4=0$; these are
the only two kinds. Only \emph{constrained} moments carry
multipliers: the coefficient of $x^3$ is zero because $\E X^3$ is not constrained in $\Ckap$,
and $\lambda_4\geq 0$ because (A3) is an inequality in the direction $\le\kappa$.
For \eqref{eq:cert-event} on $[t,\infty)$ we use the exact Markov--Luk\'acs form
$q-1=\sigma_0+(x-t)\,\sigma_1$ with $\sigma_0,\sigma_1$ sums of squares; \eqref{eq:cert-glob}
is an explicit sum of squares. Such data make the two inequalities checkable by expanding
polynomial identities and verifying positive semidefiniteness of small Gram matrices;
no analysis is required.
\end{definition}

\begin{lemma}[Weak duality]\label{lem:weak-duality}
If $(q,B)$ is a certificate as in \cref{def:cert}, then $\P(X\in S)\le B$ for every
$X\in\Ckap$.
\end{lemma}

\begin{proof}
By \eqref{eq:cert-glob}--\eqref{eq:cert-event}, $\mathbf{1}_S(x)\le q(x)$ pointwise on $\R$.
Taking expectations under any $X\in\Ckap$ (all moments through degree $4$ exist by (A3)),
\[
\P(X\in S)\;\le\;\E q(X)
\;=\;\lambda_0+\lambda_1\,\E X+\lambda_2\,\E X^2+\lambda_4\,\E X^4
\;\le\;\lambda_0+\lambda_2+\lambda_4\kappa\;=\;B,
\]
using (A1), (A2), and, for the last inequality, (A3) together with $\lambda_4\ge0$.
\end{proof}

A matching \emph{lower} bound is always exhibited by a distribution: a discrete law
$\mu=\sum_i w_i\delta_{x_i}\in\Ckap$ with $\mu(S)=B$. When both are present, $\Vone=B$
exactly; every tightness claim in this paper is of this matched-pair form, and both halves are
finite objects that a referee (or a program) can verify by rational arithmetic.

\begin{definition}[Proof degree]\label{def:pd}
For $(t,\kappa)$ with $\Vone(t,\kappa)=B$, the \emph{proof degree} $\PDeg(t,\kappa)$ is
the least degree ($2$ or $4$, by \cref{def:cert}) of a certificate for
$\P(X\ge t)\le B$ over $\Ckap$. Degree-$2$ certificates ($\lambda_4=0$) are exactly the
classical Cantelli-type proofs, blind to (A3).
\end{definition}

\begin{remark}[Why degree $4$ always suffices]\label{rem:exactness}
In one variable there is no relaxation gap: a univariate polynomial nonnegative on $\R$ is a
sum of two squares, and one nonnegative on $[t,\infty)$ of degree $\le 4$ has an exact
representation $\sigma_0+(x-t)\sigma_1$ with $\deg\sigma_0\le4$, $\deg\sigma_1\le2$
(Markov--Luk\'acs; see \citealp{nesterov2000}). Consequently the degree-$4$ moment
relaxation of $\Vone$ is exact, and the certificates of \cref{sec:results} realize the
optimum. This exactness is what makes a \emph{complete} map possible and distinguishes the
univariate landscape from the multivariate one, where hierarchies converge under an
Archimedean condition but may need unbounded degree
\citep{putinar1993,nie2007putinar,nie2014}.
\end{remark}

\paragraph{Extremal supports.} Classical Tchebycheff-system theory
\citep{karlinstudden1966} says worst-case measures for problems with four moment constraints
concentrate on few points; we will exhibit two- and three-point extremal laws throughout and
never need the general theory: each regime's proof is a self-contained matched pair.

\paragraph{Notation.} $b(\kappa)=\tfrac12(\sqrt{\kappa+3}-\sqrt{\kappa-1})$ and
$c(\kappa)=\tfrac12(\sqrt{\kappa+3}+\sqrt{\kappa-1})$, so that $b(\kappa)c(\kappa)=1$,
$b\le 1\le c$, and $b^2+c^2=\kappa+1$; $\kc(t)=t^2-1+1/t^2$ (the fourth moment of the
Cantelli extremal law at threshold $t$); $\kss=(3\sqrt3-3)/2$; $\ks=3/2$; and, for
$1\le\kappa\le\ks$, with $u=\sqrt{\kappa-1}$, $s=\sqrt{\kappa+3}$,
\[
\tauofk\;=\;\sqrt{u\,(s+u)}\;-\;b(\kappa).
\]
The curve $\kappa=\kc(t)$ and the curves $t=b(\kappa)$, $t=c(\kappa)$ are the same set: for
$\kappa\geq 1$, $\kappa\ge\kc(t)\iff b(\kappa)\le t\le c(\kappa)$
(\cref{lem:tongue-equiv}).

\section{Main results}
\label{sec:results}

This section states the results. \cref{sec:map} gives the one-sided map
(\cref{thm:main-map}) and reads off its regime structure; \cref{sec:twosided} gives the
two-sided map and the one-/two-sided collapse; \cref{sec:pd} the proof-degree phase
transition; and \cref{sec:endpoint} the central regime, its $t=0$ endpoint, and the
impossibility of a nested-square-root closed form. Throughout, every upper bound is a
certificate in the sense of \cref{def:cert} and every matching lower bound an explicit
two- or three-point law; the paired proofs are deferred to \cref{app:proofs}.

\subsection{The one-sided map}
\label{sec:map}

\newcommand{\ThmMainMapBody}[1][]{%
Let $\kappa\ge1$, $t>0$, and adopt \cref{asm:class} and the notation of \cref{sec:prelim}.
\begin{enumerate}
\item[\textup{(I)}] \textbf{\textup{(Cantelli tongue.)}} If $b(\kappa)\le t\le c(\kappa)$,
equivalently $\kappa\ge\kc(t)$, then
\[
\Vone(t,\kappa)\;=\;\frac{1}{1+t^2},
\]
attained by $\mu^\star=\frac{1}{1+t^2}\,\delta_t+\frac{t^2}{1+t^2}\,\delta_{-1/t}$, whose
fourth moment is exactly $\kc(t)$.
\item[\textup{(II)}] \textbf{\textup{(Tail regime.)}} If $t\ge c(\kappa)$ and
$(t,\kappa)\neq(1,1)$, then
\begin{equation}#1
\Vone(t,\kappa)\;=\;p\;:=\;\frac{\kappa-1}{(t^2-1)^2+\kappa-1},
\end{equation}
attained by the three-point law $\mu^\star=p\,\delta_t+w_+\delta_a+w_-\delta_{-a}$ with
\[
a^2=u:=\frac{1-p\,t^2}{1-p},
\qquad
w_\pm=\frac{(1-p)\mp p\,t/a}{2}\;\ge 0 .
\]
\item[\textup{(IIIa)}] \textbf{\textup{(Plateau.)}} If $1<\kappa\le\ks=\tfrac32$ and
$\max(\tauofk,0)\le t\le b(\kappa)$ (with $t>0$), then, independently of $t$,
\[
\Vone(t,\kappa)\;=\;\frac12\biggl(1+\sqrt{\frac{\kappa-1}{\kappa+3}}\,\biggr)
\;=\;\frac{c(\kappa)}{\sqrt{\kappa+3}},
\]
attained by the two-point law
$\mu^\star=\frac{c}{\sqrt{\kappa+3}}\,\delta_{b}+\frac{b}{\sqrt{\kappa+3}}\,\delta_{-c}$,
which does not depend on $t$. The plateau window is nonempty exactly for $\kappa\le\ks$, and
$\tauofk\le0$ exactly for $\kappa\le\kss$, so for $\kappa\le\kss$ the plateau covers the whole
range $0<t\le b(\kappa)$.
\item[\textup{(IIIb)}] \textbf{\textup{(Central regime.)}} For the remaining parameters
($0<t<b(\kappa)$ with $\kappa>\ks$, or $0<t<\tauofk$ with $\kss<\kappa\le\ks$), the value
$\Vone(t,\kappa)=w_1+w_2$ at any solution of the explicit system \textup{(S)} of
\cref{sec:endpoint}, describing a three-point extremal law $\{-c,\,t,\,b\}$ and a matching
degree-$4$ certificate; solvability holds at every parameter point we evaluated
\textup{(}\cref{sec:pipeline}\textup{)}.
\end{enumerate}
The values glue continuously across all regime boundaries, and for every $(t,\kappa)$ the
upper bound is witnessed by an explicit certificate in the sense of \cref{def:cert}.
}
\begin{theorem}[The exact tail map of the bounded-kurtosis class]\label{thm:main-map}\ThmMainMapBody[\label{eq:regimeII}]\end{theorem}

\begin{proof}[Proof sketch]
Each regime is a matched pair: a certificate $q$ (\cref{lem:weak-duality}) and a distribution
attaining its value. The certificates are, respectively,
\[
\text{(I)}\;\; q=\frac{(x+1/t)^2}{(t+1/t)^2},
\qquad
\text{(II)}\;\; q=\frac{(x^2-u)^2}{(t^2-u)^2},
\qquad
\text{(IIIa)}\;\; q=\lambda_4\bigl[(x^2-c^2)^2+k_2(x+c)^2\bigr]
\]
with $k_2=2b(c-b)$ and $\lambda_4^{-1}=(c-b)(b+c)^3$; Regime (IIIb) uses the same shape as
(IIIa) with parameters determined by (S). What varies is which pointwise inequalities are
active where, and the entire content of the proof is the verification, by explicit
factorization, that $q\ge\mathbf{1}_{[t,\infty)}$ holds exactly on the stated parameter
region, plus rational moment arithmetic for the witness. Full proofs: \cref{app:proofs};
fully worked numerical instances of each closed-form regime: \cref{app:worked};
the equality-constrained class $\E X^4=\kappa$ has the same map (\cref{prop:eqclass}).
\end{proof}

Two remarks on how the statement is organized. First, the Regime I
\emph{upper} bound $1/(1+t^2)$ holds for every $(t,\kappa)$: its certificate never uses
\textup{(A3)}; the tongue hypothesis $\kappa\ge\kc(t)$ is exactly the condition for the bound
to be \emph{attained} (the Cantelli witness must satisfy the fourth-moment budget), and off
the tongue the value is strictly smaller (\cref{lem:rigidity}). Second, the closed regimes
overlap along the boundary curves $t=b(\kappa)$, $t=c(\kappa)$, $t=\tauofk$; the map is
nevertheless single-valued because the adjacent formulas agree there (gluing identities,
verified symbolically), so boundary points may be read from either side.

\paragraph{Reading the map.} Three features deserve comment.
\begin{itemize}
\item \emph{The tongue is where kurtosis is worthless.} The two-moment worst case (Cantelli's
two-point law) has fourth moment $\kc(t)=t^2-1+1/t^2$; whenever the budget $\kappa$ covers
this, no fourth-moment information can help, and the map says this is the \emph{only}
obstruction. The tongue equivalence $\kappa\ge\kc(t)\iff b(\kappa)\le t\le c(\kappa)$ is the
factorization $t^4-(\kappa+1)t^2+1=(t^2-b^2)(t^2-c^2)$. A pointed instance: $t=1$ lies
in the tongue for \emph{every} $\kappa$ (as $\kc(1)=1$), so the classical median--mean
bound $|m-\mu|\le\sigma$ (Cantelli at $t=1$) cannot be improved by any kurtosis
assumption; the witness there is the symmetric law on $\{\pm1\}$, whose fourth moment
is the \emph{minimal} value $1$.
\item \emph{Regime II quantifies the tail gain.} For $t\ge c(\kappa)$,
\eqref{eq:regimeII} decays like $(\kappa-1)/t^4$: fourth-moment information converts the
$1/t^2$ Cantelli decay into $1/t^4$ decay, and the constant $\kappa-1$ --- the variance of
$X^2$, not the raw fourth moment $\kappa$ --- is optimal; the trivial Markov bound
$\E X^4/t^4\le\kappa/t^4$ is off by exactly the factor $\kappa/(\kappa-1)$ (and by
infinitely much at $\kappa=1$). \cref{fig:comparison} draws the comparison at $\kappa=3$.
\item \emph{The plateau is a freezing phenomenon.} For small $\kappa$ and small $t$, the
optimal adversary ignores the threshold entirely: the same two-point law
$\{b(\kappa),-c(\kappa)\}$ is worst for every $t$ in the window, and its tail mass
$c/\sqrt{\kappa+3}$ is the value. The window degenerates at two computable constants: at
$\kappa=\ks=3/2$ its endpoints collide at $\tau=b=1/\sqrt2$, and at $\kappa=\kss$ the left
endpoint hits $t=0$. Both constants seem to be new.
\end{itemize}

\begin{figure}[t]
\centering
\includegraphics[width=0.8\linewidth]{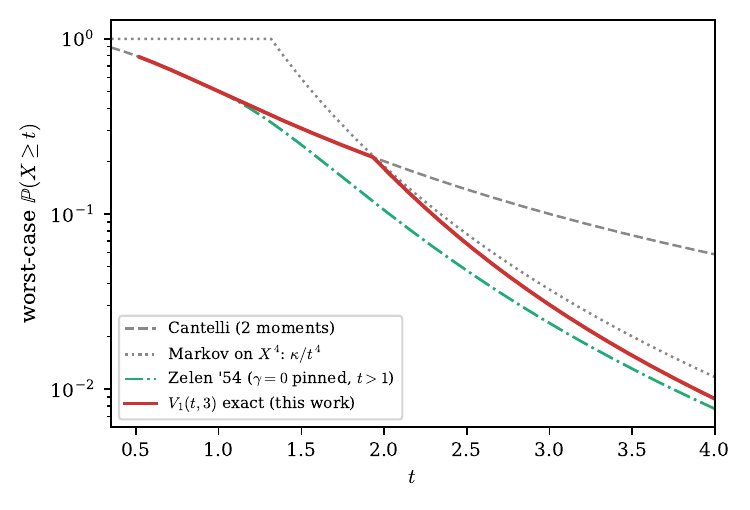}
\caption{The exact map against the classical bounds at $\kappa=3$ (log scale). The gap to
Cantelli is the value of kurtosis information; the gap to the fourth-moment Markov bound
$\kappa/t^4$ is the factor $\kappa/(\kappa-1)$; the gap to Zelen's symmetric bound (valid
for the smaller, $\gamma$-pinned class, $t>1$) is the exact price of not knowing the
skewness.}
\label{fig:comparison}
\end{figure}

\paragraph{Boundary cases.} At $\kappa=1$, $\CkapOf{1}=\{\text{uniform on }\{\pm1\}\}$ and
$\Vone=\tfrac12$ for $t\le1$, $0$ for $t>1$; the formulas' limits agree. The corner
$(t,\kappa)=(1,1)$, where \eqref{eq:regimeII} reads $0/0$, belongs to Regime I (the tongue
degenerates to $\{1\}$ there), with value $\tfrac12$. Monotonicity --- $\Vone$ nonincreasing
in $t$, nondecreasing in $\kappa$ --- holds throughout with kinks, never jumps, across the
regime boundaries.

\subsection{The two-sided map and the collapse}
\label{sec:twosided}

\newcommand{\ThmTwoSidedBody}{%
Under \cref{asm:class}, for $t\ge1$:
\[
\Vtwo(t,\kappa)\;=\;
\begin{cases}
\dfrac{1}{t^2}, & \kappa\ge t^2 \quad(\text{Chebyshev regime; extremal }\{0,\pm t\}),\\[2ex]
\dfrac{\kappa-1}{(t^2-1)^2+\kappa-1}, & 1\le\kappa\le t^2\quad(\text{extremal }\{\pm t,\pm a\}),
\end{cases}
\]
and $\Vtwo(t,\kappa)=1$ for $0<t\le1$ (extremal uniform on $\{\pm1\}$). In the second case
the certificate is the \emph{same} even polynomial $q=(x^2-u)^2/(t^2-u)^2$ as in Regime II of
\cref{thm:main-map}.
}
\begin{theorem}[Two-sided map]\label{thm:twosided}\ThmTwoSidedBody\end{theorem}

\newcommand{\CorCollapseBody}{%
For $t\ge 1$, $\ \Vone(t,\kappa)=\Vtwo(t,\kappa)$ if and only if $t\ge c(\kappa)$
\textup{(}Regime II; at $\kappa=1$, for all $t>1$, both vanish\textup{)}, where both equal
\eqref{eq:regimeII}. In particular, beyond $c(\kappa)$ the ability to place all deviation
mass on one side is worthless to the adversary: Cantelli's strict one-sided gain over
Chebyshev, present at the two-moment level for every $t$, does not survive the addition of a
binding kurtosis constraint.
}
\begin{corollary}[One-sided/two-sided collapse]\label{cor:collapse}\CorCollapseBody\end{corollary}

The mechanism is visible in the extremal laws: the one-sided worst case $\{t,a,-a\}$ and the
two-sided worst case $\{\pm t,\pm a\}$ are different distributions with the \emph{same}
objective value, and the certificate $q$ is even, so it cannot tell the two events apart.
When $\kappa>\kc(t)$ (inside the tongue) the collapse fails: $\Vone=1/(1+t^2)$ strictly below
$\Vtwo=\min\{1/t^2,\,(\kappa-1)/((t^2-1)^2+\kappa-1)\}$.

\subsection{The proof-degree phase transition}
\label{sec:pd}

\newcommand{\ThmProofDegreeBody}{%
Under \cref{asm:class}, for every $t>0$, $\kappa\ge1$ \textup{(}with
$\Vone$ from \cref{thm:main-map}\textup{)}:
\[
\PDeg(t,\kappa)\;=\;
\begin{cases}
2, & \kappa\ge\kc(t)\quad(\text{the closed tongue}),\\
4, & \kappa<\kc(t)\quad(\text{everywhere else}),
\end{cases}
\]
and degree $4$ always suffices. Concretely: a degree-$2$ certificate is blind to
\textup{(A3)}, so evaluating any degree-$2$-feasible $q$ against the Cantelli law
$\mu_C=\frac{1}{1+t^2}\delta_t+\frac{t^2}{1+t^2}\delta_{-1/t}$ (feasible for the two-moment
class) forces its bound $\ge 1/(1+t^2)$; off the tongue $\Vone<1/(1+t^2)$ strictly, so no
degree-$2$ certificate proves the tight constant, while the explicit degree-$4$ certificates
of \cref{thm:main-map} do.
}
\begin{theorem}[Proof-degree map]\label{thm:proofdegree}\ThmProofDegreeBody\end{theorem}

The mechanism is elementary (a degree-$2$ certificate simply cannot see the fourth
moment), and we do not claim otherwise; the content of the theorem is that the transition
curve is \emph{exact and explicit}, and that the obstruction is witnessed by an honest
probability law rather than a pseudo-distribution. The transition curve is exactly the tongue
boundary, an explicit algebraic curve along which the least sufficient proof degree of a
true inequality jumps. The obstruction witness
$\mu_C$ is not a pathological pseudo-distribution but a bona fide probability law that the
weaker proof system cannot exclude; in sum-of-squares terms, the degree-$2$ pseudo-moment
vector $(1,0,1)$ extends $\mu_C$ and certifies that the degree-$2$ relaxation value equals
$1/(1+t^2)$ for every $\kappa$. We record all such obstruction witnesses alongside the
certificates in the artifact repository.

\subsection{The central regime and the \texorpdfstring{$t=0$}{t=0} endpoint}
\label{sec:endpoint}

On IIIb the extremal law has three atoms $\{-c,\,t,\,b\}$ with $c>b>t$: an atom \emph{at the
threshold} reappears, but, unlike Regime II, the far atoms are asymmetric and the
certificate is no longer a perfect square. The matched pair is described by:
\begin{equation}\tag{S}\label{eq:systemS}
\begin{aligned}
&w_0+w_1+w_2=1,\qquad -c\,w_0+t\,w_1+b\,w_2=0,\qquad c^2w_0+t^2w_1+b^2w_2=1,\\
&c^4w_0+t^4w_1+b^4w_2=\kappa,\qquad q(t)=q(b)=1,\qquad q'(b)=0,\\
&\text{where } q(x)=\lambda_4\bigl[(x^2-c^2)^2+k_2\,(x+c)^2\bigr],\\
&\text{subject to}\quad w_i\ge0,\quad 0<t<b<c,\quad k_2\ge0,\quad \lambda_4>0,\quad
q-1\ge0\ \text{on }[t,\infty).
\end{aligned}
\end{equation}
Any solution of \eqref{eq:systemS} yields, exactly,
\[
\Vone(t,\kappa)\;=\;w_1+w_2\;=\;\lambda_0+\lambda_2+\lambda_4\,\kappa
\qquad\text{(see \cref{prop:IIIb})}.
\]
Two exact statements anchor the regime:

\newcommand{\PropEndpointBody}{%
For $\kappa\ge\kss=(3\sqrt3-3)/2$,
\[
\Vone(0,\kappa)\;=\;\sup_{X\in\Ckap}\P(X\geq0)\;=\;1-\frac{2\sqrt3-3}{\kappa},
\]
attained by the three-point law on $\{-c_0,\,0,\,b_0\}$ with $b_0=\sqrt{2\kappa/3}$,
$c_0=b_0(1+\sqrt3)/2$; for $1\le\kappa\le\kss$, $\Vone(0^+,\kappa)$ equals the plateau value
of \cref{thm:main-map}\textup{(IIIa)}. The two expressions agree at $\kappa=\kss$
\textup{(}both equal $1/\sqrt3$\textup{)}.
}
\begin{proposition}[The $t=0$ endpoint]\label{prop:endpoint}\PropEndpointBody\end{proposition}

The constant $2\sqrt3-3$ is that of \citet{hezhangzhang2010}; our map exhibits it as the
$t=0$ boundary value of the central regime. In the interior of IIIb no comparable closed form
exists: at the rational point $(t,\kappa)=(\tfrac12,2)$, exact elimination in \eqref{eq:systemS}
shows the minimal polynomial of $\Vone$ over $\Q$ is
\begin{align*}
98875\,V^6-73824030\,V^5+345341355\,V^4-642400008\,V^3&\\
{}+592968512\,V^2-271538688\,V+49352704&\;=\;0
\end{align*}
(degree six, irreducible; coefficients exact, obtained twice by independent exact
eliminations). Since any expression in nested \emph{square} roots of rational functions of
$(t,\kappa)$ takes, at a rational point, values inside an iterated quadratic extension of
$\Q$, whose elements have algebraic degree a power of $2$ (\cref{item:tower} in
\cref{app:register}), a degree-$6$ value rules out
every formula of the Regime I--IIIa shape. (Higher-order radicals are not ruled out; that
would require the Galois group.) On grids, we certify IIIb values by \emph{sandwiches}: the witness value from
\eqref{eq:systemS} below, and a rational $\varepsilon$-retreat certificate above, matching to
the solver's precision (\cref{sec:pipeline}).

\section{\texorpdfstring{\LemmaForge}{LemmaForge}: a certified discovery pipeline}
\label{sec:pipeline}

We discovered the theorems of \cref{sec:results} through an AI-guided search organized
around the pipeline of this section, \LemmaForge, and re-validated them independently of
how they were found. The division of labor is strict: AI agents proposed candidate regime
boundaries, extremal configurations, and certificate shapes; the pipeline's numerical and
symbolic layers tested every proposal and discarded what failed; certification of the
survivors was granted by the independent verifier alone. Crucially, this automates a
discovery process previously carried out by hand, one sharp inequality at a time, and it
adds nothing to the trusted base: the design principle throughout is that no claimed
constant rests on floating-point evidence, so every claim ships with a hand-checkable
proof and an independently re-verified certificate. We describe the pipeline briefly;
formats, checks, and reproduction instructions are in \cref{app:pipeline}, a certificate
is reproduced verbatim with the verifier's output in \cref{app:walkthrough}, fully worked
instances of the map (hand-checkable in minutes) are in \cref{app:worked}, and the
complete instance tables are in \cref{app:tables}.

\begin{algorithm}[b]
\caption{\LemmaForge{} (driver). The producing subroutines
(\cref{alg:solve,alg:round}) and the certifying one (\cref{alg:verify}) share no code;
only \textsc{Verify} grants \textsc{certified} status.}
\label{alg:lemmaforge}
\begin{algorithmic}[1]
\Require instance $\mathcal L=\bigl(K,\{(g_i,\bowtie_i,c_i)\}_{i=0}^m,S=[t,\infty),\text{sense}\bigr)$, rational data
\Ensure exact certificate (and witness where attained); verdict
\State $(\tilde\lambda,\{\tilde G_j\},\tilde V)\gets\Call{SolveCrossCheck}{\mathcal L}$
\State conjecture exact value from $\tilde V$ by PSLQ / sequence fitting
  \Comment{status \textsc{recognized}, never more}
\If{a proved closed form covers $\mathcal L$ (Regimes I/II/IIIa; benchmark cells)}
  \State build the certificate and witness \emph{symbolically} from \cref{thm:main-map}
\Else
  \State certificate $\gets\Call{Round}{\mathcal L,\tilde\lambda,\{\tilde G_j\},\tilde V}$
\EndIf
\State attach the extremal witness when the optimum is attained and exact
\State \Return \Call{Verify}{certificate}
  \Comment{fresh process; the only path to \textsc{certified}}
\end{algorithmic}
\end{algorithm}

\suppressfloats[t]
\begin{algorithm}[t]
\caption{\textsc{SolveCrossCheck}: compile, solve, and get a second opinion.}
\label{alg:solve}
\begin{algorithmic}[1]
\Require instance $\mathcal L$, rational data
\Ensure numerical dual $\tilde\lambda,\{\tilde G_j\}$, value $\tilde V$ corroborated by an
  independent LP, and the split pseudo-moments; or the instance flagged and aborted
\State compile the pieces of \cref{def:cert} (event: $q-1$ on $K\cap S$; support: $q$ on
  $K$) with their Markov--Luk\'acs blocks (\cref{lem:sos-line,lem:sos-ray});
  impose coefficient matching and multiplier signs
\State solve the dual SDP (Clarabel; SCS fallback) $\to$
  $\tilde\lambda,\{\tilde G_j\},\tilde V$; read the split pseudo-moments off the
  equality duals
\State solve the independent discretized LP (geometric far-tail grid)
\State \textbf{if} $|\tilde V-\mathrm{LP}|>10^{-4}$ \textbf{then} flag the instance and
  abort \textbf{else} \Return $(\tilde\lambda,\{\tilde G_j\},\tilde V)$
\end{algorithmic}
\end{algorithm}

\begin{algorithm}[t]
\caption{\textsc{Round}: Peyrl--Parrilo rationalization with $\varepsilon$-retreat.}
\label{alg:round}
\begin{algorithmic}[1]
\Require instance $\mathcal L$; numerical dual $\tilde\lambda,\{\tilde G_j\}$ and value
  $\tilde V$ from \cref{alg:solve}
\Ensure an exact certificate claiming $\tilde V$, or one claiming the explicitly weaker
  $\tilde V+\varepsilon$ ($\varepsilon$-retreat, labeled as such), or \textsc{fail}
\For{denominator bound $N\in\{10^{3},10^{6},10^{9}\}$}
  \State rationalize $(\tilde\lambda,\{\tilde G_j\})$ at denominator $N$; project exactly
    (rational normal equations) onto the coefficient-identity subspace
  \State \textbf{if} every projected Gram passes exact $LDL^{\top}$ (\cref{lem:ldlt})
    \textbf{then} \Return certificate for $\tilde V$
\EndFor
\For{$\varepsilon\in\{10^{-9},10^{-6}\}$}\Comment{$\varepsilon$-retreat: restore a strict interior}
  \State re-solve with the objective pinned to $\tilde V+\varepsilon$; rerun
    lines~1--4, the certificate now claiming the explicitly weaker bound
    $\tilde V+\varepsilon$
\EndFor
\State \Return \textsc{fail}\Comment{recorded, never promoted}
\end{algorithmic}
\end{algorithm}

\begin{algorithm}[t]
\caption{\textsc{Verify} (independent checker; fresh process, exact arithmetic,
no producer code).}
\label{alg:verify}
\begin{algorithmic}[1]
\Require a certificate file: constraints, multipliers $\lambda$, pieces with blocks and
  Gram matrices, claimed bound, optional witness --- all exact (rational or quadratic-surd)
\Ensure verdict \textsc{verified-tight} / \textsc{verified-bound} / \textsc{fail};
  \textsc{certified} status is granted exactly on pass
\State re-parse the certificate; recompute $q=\sum_i\lambda_ig_i$; reject any float
\State check each piece identity coefficient-by-coefficient; each Gram PSD by exact
  $LDL^{\top}$ \emph{and} the coefficient criterion
  (\cref{lem:ldlt,lem:charpoly}); multiplier signs; claimed bound $=\sum_i\lambda_ic_i$
\State \textbf{if} a witness is attached: check its moments and tail mass exactly
\State \Return \textsc{verified-tight} / \textsc{verified-bound} / \textsc{fail}
  (checks V1--V8, \cref{app:pipeline})
\end{algorithmic}
\end{algorithm}

\begin{table}[t]
\centering\small
\begin{tabular}{llll}
\toprule
cell & class / objective & truth & verdicts \\
\midrule
B1 Markov & $[0,\infty)$, $\E X{=}1$; $\P(X\ge t)$ & $1/t$ & 4 tight \\
B2 Cantelli & $\R$, $(0,\sigma^2)$; $\P(X\ge t)$ & $\sigma^2\!/(\sigma^2{+}t^2)$ & 4 tight \\
B3 3-moment & $[0,\infty)$, $(1,m_2,m_3)$; $\P(X\ge t)$ & LP cross-check & 8 bound, 1 fail \\
B4 tight PZ & $[0,\infty)$, $(1,m_2)$; $\inf\P(X\ge\theta)$ & $\frac{(1-\theta)^2}{(1-\theta)^2+m_2-1}$ & 6 bound (inf unattained) \\
B5 Khintchine $p{=}4$ & sphere, $n{=}2..8$ (POP) & $3-2/n$ & 7 numeric-validated$^{\dagger}$ \\
Zelen slice & $\R$, $(0,1,\gamma{=}0,\kappa)$; $t>1$ & \citet{zelen1954} & 5 numeric-validated$^{\dagger}$ \\
\midrule
F1 kurtosis map & $\Ckap$; $\P(X\ge t)$ & \cref{thm:main-map} & 45 tight, 2 bound($\varepsilon$), 1 fail \\
F2 skewness-pinned & $(0,1,\gamma,\kappa{\le})$ & --- & 47 bound, 13 fail \\
\bottomrule
\end{tabular}
\caption{Validation battery and certified frontier instances, decomposed by verdict.
\emph{tight} = exact certificate + exact extremal witness, both re-verified by the
independent checker in a fresh process (\textsc{verified-tight}); \emph{bound} = exact
certificate for the bound only (\textsc{verified-bound}; for the two IIIb instances the
certified bound is value${}+\varepsilon$, $\varepsilon\le10^{-6}$); \emph{numeric-validated}
($\dagger$) = value matches the classical truth to $10^{-10}$ but no certificate is emitted
(the POP mode and the literature cross-check are validation demos, not certificate-bearing
cells); \emph{fail} = rounding failure on a numerically degenerate instance, recorded and
never promoted. Certification and LP cross-checking are separate axes: all solved values,
including the failed-certification instances, agree with the independent discretized LP to
its $10^{-7}$--$10^{-8}$ resolution.}
\label{tab:benchmarks}
\end{table}

\paragraph{Architecture.} \cref{alg:lemmaforge} is the driver; \cref{alg:solve,alg:round,alg:verify} are its subroutines. A lemma instance (support, moment constraints, objective event,
parameter values) is compiled into the dual sum-of-squares program of \cref{def:cert} using
the exact univariate representations of \cref{rem:exactness} (never a generic hierarchy
template), so the semidefinite value \emph{is} the moment-problem value. The SDP is solved
numerically (Clarabel, with an SCS fallback); the primal atoms are recovered from the duals of
the coefficient-matching equalities (the split pseudo-moment sequences; cf.\ the flat-data
viewpoint of \citealp{curtofialkow1996}), and every value is
cross-checked against an \emph{independent} discretized linear program over $10^4$--$2\cdot
10^4$ support atoms with a geometric far-tail grid: a deliberately low-tech second opinion
that has no code in common with the SDP path. Exact values are then produced along two
routes:
\begin{itemize}
\item \textbf{Symbolic route} (Regimes I, II, IIIa; Markov; Cantelli; tight Paley--Zygmund):
  the certificate is constructed in exact arithmetic directly from the closed forms of
  \cref{thm:main-map}: Gram matrices with rational or quadratic-surd entries.
\item \textbf{Rounding route} (three-moment benchmarks; the skewness-pinned surface; Regime
  IIIb): the numerical dual is rounded to rationals and projected, in exact rational linear
  algebra, onto the affine subspace of the polynomial identity, in the manner of
  \citet{peyrlparrilo2008}; if the projected Gram matrices fail exact positive
  semidefiniteness, an $\varepsilon$-retreat re-solves with the objective pinned to
  $\text{value}+\varepsilon$ (restoring a strict interior) and certifies the weaker,
  explicitly-stated bound.
\end{itemize}

\paragraph{Independent verification.} Certification is granted by a separate program
(\texttt{forge/verify.py}) that shares no code with the compiler, solvers, or rounding: it
re-parses the certificate, recomputes $q$ from the multipliers, checks the polynomial
identities coefficient-by-coefficient over $\Q$ (or exactly over quadratic extensions),
checks every Gram matrix positive semidefinite by exact $LDL^\top$ \emph{and} by a
characteristic-polynomial criterion, checks multiplier signs against constraint directions,
and re-evaluates the claimed bound. If an extremal witness is attached it checks feasibility
and tightness exactly, upgrading the verdict to \textsc{verified-tight}. A certificate is
counted only when this checker passes \emph{in a fresh process}. The trusted base is thus
the checker's $\sim$370 audited lines plus sympy's exact rational and algebraic
arithmetic; the latter is
shared, we note honestly, with the producing side as a library, which is why the two PSD criteria are
computed by different code paths. The verifier was itself
validated by mutation testing: eight classes of corrupted certificates (wrong bound, indefinite
Gram, broken identity, illegal multiplier, infeasible witness, smuggled floats, wrong signs)
are all rejected.

\paragraph{Benchmarks.} \cref{tab:benchmarks} summarizes the validation battery. Every
certificate-bearing benchmark with a classical closed form is reproduced exactly and
certified (B5 and the Zelen slice are validation demos: values match to $10^{-10}$ with no
certificate emitted); the recognition
stage (integer-relation detection and exact sequence fitting) is validated by recovering
$3-2/n$ from the seven Khintchine $p{=}4$ values with no hints. One benchmark corrects the
textbook: for the class $\{X\ge0,\ \E X=1,\ \E X^2=m_2\}$ the tight Paley--Zygmund constant
is $(1-\theta)^2/((1-\theta)^2+m_2-1)$, not $(1-\theta)^2/m_2$; the pipeline computes the
tight value, our certificate proves it, and the infimum is approached but not attained (the
extremal atom sits at $\theta^-$); the discretized LP converges to it from above, as it
must.

\paragraph{Certified instances of the map.} For the kurtosis cell, 47 of 48 grid instances
across all four regimes are \textsc{certified}: for Regimes I/II with fully rational
certificates and witnesses (verdict \textsc{verified-tight}), for IIIa with exact
quadratic-surd data (e.g., at $(t,\kappa)=(\tfrac12,\tfrac54)$ the verified bound is
$\tfrac12+\tfrac{\sqrt{17}}{34}$), and for IIIb by rational $\varepsilon$-retreat sandwiches.
The single failure is the boundary-degenerate instance $(t,\kappa)=(\tfrac12,\tfrac32)$
(exactly on $\kappa=\ks$), where rounding could not restore interiority; it remains labeled
numerical, with the failure itself recorded: under the project's discipline, a failed
certification is data, not an embarrassment to hide.

\paragraph{Extremal gallery.} Every certified instance carries its worst-case distribution
(\cref{fig:gallery}): the machine-extracted atoms confirm the structural story of
\cref{thm:main-map}: two atoms on the tongue and plateau, three atoms with an atom
\emph{at} the threshold in Regimes II and IIIb, the freezing of the plateau law, and the
reappearance of the threshold atom in the central regime.

\begin{figure}[t]
\centering
\includegraphics[width=\linewidth]{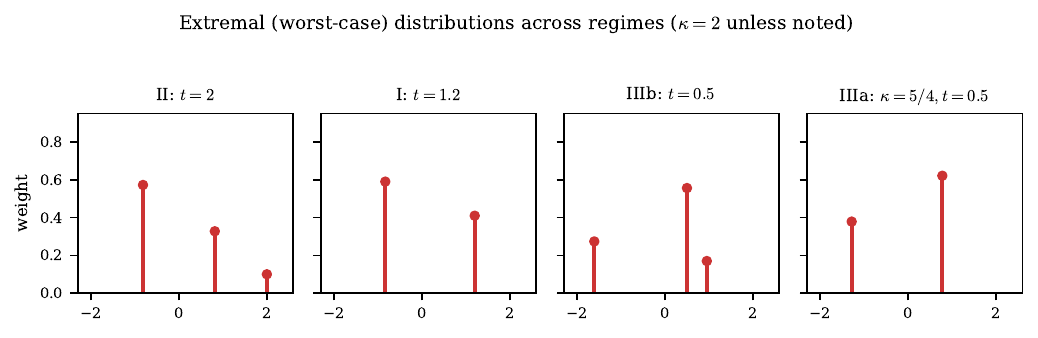}
\caption{Worst-case distributions across the regimes (atoms and weights, exact where the
regime has a closed form). Note the structural transitions: two-point laws (I: the
Cantelli pair; IIIa: frozen at $\{b,-c\}$) against three-point laws (II, IIIb), and an
atom at the threshold everywhere except IIIa.}
\label{fig:gallery}
\end{figure}

\section{Applications}
\label{sec:consequences}

Exactness makes the map usable: closed forms can be inverted, composed with averaging,
and transferred into proof systems. This section records four applications. Each
follows from \cref{thm:main-map,thm:twosided,cor:collapse} by direct substitution and
inherits their exactness; no new machinery enters. Full proofs are in
\cref{app:consproofs}.

\paragraph{Exact worst-case quantiles.} Inverting the map in $t$ answers exactly the
practical question of how many standard deviations buy a prescribed exceedance
probability, and locates the point where the kurtosis budget starts paying.

\newcommand{\CorQuantileBody}{%
Let $\kappa>1$ and let $X$ range over all laws with mean $\mu$, variance $\sigma^2$, and
$\E(X-\mu)^4\le\kappa\sigma^4$. For $0<p<1/(1+b(\kappa)^2)$ set
\[
t_p(\kappa)\;=\;
\begin{cases}
\sqrt{\dfrac{1-p}{p}}, & \dfrac{1}{1+c(\kappa)^2}\;\le\;p
  \qquad\text{\textup{(}Cantelli branch: kurtosis-independent\textup{)}},\\[2ex]
\Bigl(1+\sqrt{(\kappa-1)\,\dfrac{1-p}{p}}\;\Bigr)^{\!1/2}, & p\;\le\;\dfrac{1}{1+c(\kappa)^2}
  \qquad\text{\textup{(}kurtosis branch\textup{)}}.
\end{cases}
\]
\begin{enumerate}
\item[\textup{(i)}] \textbf{\textup{(One-sided.)}} $\P(X\ge\mu+t\sigma)\le p$ holds for
every such $X$ if and only if $t\ge t_p(\kappa)$.
\item[\textup{(ii)}] \textbf{\textup{(Two-sided.)}} For every $0<p<1$,
$\ \P(|X-\mu|\ge t\sigma)\le p$ holds for every such $X$ if and only if
$t\ge t^{(2)}_p(\kappa)$, where $t^{(2)}_p(\kappa)=1/\sqrt{p}$ when $p\ge1/\kappa$, and
$t^{(2)}_p(\kappa)$ is the kurtosis branch of $t_p(\kappa)$ when $p\le1/\kappa$ (that
formula remains valid on all of $p\le1/\kappa$).
\end{enumerate}%
}
\begin{corollary}[Worst-case quantiles]\label{cor:quantile}\CorQuantileBody\end{corollary}

At $\kappa=3$: $t_{0.05}=2.68$, $t_{0.01}=3.88$, $t_{0.001}=6.76$ standard deviations,
against Cantelli's $4.36$, $9.95$, $31.6$: the kurtosis budget compresses the $99.9\%$
one-sided worst-case quantile by a factor $4.7$. \cref{cor:quantile} is an exact answer
to the stress-shock question of \citet{maher2019}, and hands the inner problem of a
fourth-moment distributionally robust chance constraint \citep{delageye2010,wiesemann2014}
a closed form.

\paragraph{Sample means and median-of-means.} Averages enter the map through a one-line
identity.

\newcommand{\LemAvgKurtBody}{%
If $X_1,\dots,X_m$ are i.i.d.\ with mean $\mu$, variance $\sigma^2$, and kurtosis
$\kappa$, then $\sqrt m\,(\bar X_m-\mu)/\sigma$ has mean $0$, variance $1$, and fourth
moment exactly $3+(\kappa-3)/m$. No third-moment assumption is made.%
}
\begin{lemma}[Kurtosis of an average]\label{lem:avgkurt}\LemAvgKurtBody\end{lemma}

Hence $\P\bigl(\bar X_m-\mu\ge t\sigma/\sqrt m\bigr)\le\Vone\bigl(t,\,3+(\kappa-3)/m\bigr)$,
exact for the moment information used; at $m=100$, $\kappa=9$, $t=3$ the bound is
$0.0312$ against Cantelli's $0.1$. As $m$ grows, every underlying $\kappa$ is drawn to
the same slice: the $\kappa=3$ curve of the map is the universal tail law of averages.
Composing with \cref{cor:quantile}(ii) gives an exact-constant confidence interval:
writing $\kappa_m=3+(\kappa-3)/m$, for every $0<p<1$ and every $(\kappa,m)\ne(1,1)$
(the one case where $\kappa_m=1$),
\[
\P\Bigl(\,|\bar X_m-\mu|\;\ge\;t^{(2)}_p(\kappa_m)\,\frac{\sigma}{\sqrt m}\Bigr)\;\le\;p,
\]
valid because the law of the standardized mean lies in $\CkapOf{\kappa_m}$ with
$\kappa_m>1$ (\cref{lem:avgkurt}). At $p=0.01$, $\kappa=9$, $m=100$: the $99\%$ interval has
half-width $3.91\,\sigma/\sqrt m$ against Chebyshev's $10\,\sigma/\sqrt m$, and the
constant is exact for the moment information used.

For median-of-means \citep{lugosimendelson2019} with $k$ blocks of $m=n/k$ samples, the
textbook analysis bounds the per-block deviation at twice the block deviation
$\sigma_B=\sigma\sqrt{k/n}$ by Chebyshev: probability at most $1/4$.
\cref{thm:twosided} with \cref{lem:avgkurt} replaces this, at the same threshold, by the
exact worst case
\[
\Vtwo(2,\kappa_B)\;=\;\min\biggl\{\frac14,\;\frac{\kappa_B-1}{8+\kappa_B}\biggr\},
\qquad \kappa_B=3+\frac{\kappa-3}{m}
\]
(the two branches of \cref{thm:twosided} at $t=2$, written as one $\min$:
$(\kappa_B-1)/(8+\kappa_B)\le\tfrac14$ iff $\kappa_B\le4$), strictly below $1/4$
whenever $\kappa_B<4$, i.e.\ $m>\kappa-3$, and tending to $\Vtwo(2,3)=2/11$: the worked
instance of \cref{app:worked-II}, which by \cref{cor:collapse} is simultaneously the
two-sided value. In the standard binomial--Chernoff step $\exp(-k\,d(\tfrac12\,\|\,p))$,
with $d(a\,\|\,p)=a\ln\tfrac ap+(1-a)\ln\tfrac{1-a}{1-p}$, the exponent improves by the
factor $d(\tfrac12\|\tfrac2{11})/d(\tfrac12\|\tfrac14)\approx1.80$, giving the same
deviation and confidence from $45\%$ fewer blocks on no assumption beyond finite
kurtosis: averaging itself drives $\kappa_B$ to $3$.

\paragraph{Margins.} A classifier's error is a tail probability of its margin, so the
map converts moment information about the margin into an exact error bound.

\newcommand{\CorMarginBody}{%
Let $M$ be a real random variable with mean $\gamma>0$, variance $\sigma^2$, and
$\E(M-\gamma)^4\le\kappa\sigma^4$ ($\kappa\ge1$). Then
\[
\P(M\le 0)\;\le\;\Vone\!\bigl(\gamma/\sigma,\,\kappa\bigr),
\]
and the bound is exact: the supremum of $\P(M\le0)$ over all such $M$ equals
$\Vone(\gamma/\sigma,\kappa)$, and is attained.%
}
\begin{corollary}[Exact margin bound]\label{cor:margin}\CorMarginBody\end{corollary}

Reading $M$ as a classification margin ($M>0$ means correct), \cref{cor:margin} is the
exact fourth-moment correction of the classical Chebyshev--Cantelli margin step
$\P(M\le0)\le1/(1+\gamma^2/\sigma^2)$: by the tongue, that step is unimprovable by any
kurtosis assumption while $b(\kappa)\,\sigma\le\gamma\le c(\kappa)\,\sigma$, and beyond
$c(\kappa)\,\sigma$ the error decays quartically in the margin-to-noise ratio,
$\sim(\kappa-1)\,\sigma^4/\gamma^4$, rather than quadratically.

\paragraph{Directional tails under certifiable kurtosis.} In sum-of-squares estimation
\citep{kotharisteurer2017,kss2018moments,hopkinsli2018}, fourth-moment information
arrives as \emph{certifiable} kurtosis. The map supplies the exact per-direction tail
constant available at that level of the hierarchy.

\newcommand{\PropCertifiableBody}{%
Let $X$ take values in $\R^d$ with mean $\mu$ and covariance $\Sigma$, and suppose its
kurtosis is \emph{$\kappa$-certifiable}: the degree-$4$ polynomial
$u\mapsto\kappa\,\langle u,\Sigma u\rangle^2-\E\langle u,X-\mu\rangle^4$ is a sum of
squares. Then for every $u\ne0$ with $\sigma_u^2=\langle u,\Sigma u\rangle>0$ and every
$t>0$,
\[
\P\bigl(\langle u,X-\mu\rangle\;\ge\;t\,\sigma_u\bigr)\;\le\;\Vone(t,\kappa).
\]
Moreover:
\begin{enumerate}
\item[\textup{(i)}] \textbf{\textup{(Degree-$4$ witnesses.)}} Whenever $(t,\kappa)$ lies
outside the central regime (in particular, whenever $t\ge b(\kappa)$), every inequality
in the derivation carries an explicit degree-$4$ sum-of-squares witness.
\item[\textup{(ii)}] \textbf{\textup{(Unimprovability.)}} For $\kappa\ge3$ and
$t\ge b(\kappa)$, some $X$ with $\kappa$-certifiable kurtosis and $\Sigma=I_d$ attains
equality in a prescribed direction.
\item[\textup{(iii)}] \textbf{\textup{(Two-sided.)}} For $t\ge c(\kappa)$, by
\cref{cor:collapse}, the same constant is exact for the two-sided event
$\{|\langle u,X-\mu\rangle|\ge t\,\sigma_u\}$.
\end{enumerate}%
}
\begin{proposition}[Exact directional tails at degree $4$]\label{prop:certifiable}\PropCertifiableBody\end{proposition}

\cref{prop:certifiable} instantiates the proofs-to-algorithms reading of
\cref{thm:proofdegree}: from certifiable kurtosis alone, $\Vone$ is the exact
directional tail constant: no analysis resting on that hypothesis can use a smaller one
(the witness lies in the hypothesis class), and none is needed (the certificate is
explicit). On the tongue, the degree-$2$ Chebyshev--Cantelli argument already extracts
all of it; beyond the tongue, degree $4$ is necessary and exactly sufficient.

\section{Discussion}
\label{sec:discussion}

\paragraph{What the map says.} The complete solution of $\Vone$ over $\Ckap$ turns three
folklore intuitions into exact statements. Kurtosis information is worthless precisely when
the budget covers the Cantelli law's own fourth moment ($\kappa\ge\kc(t)$): worthless not
approximately but exactly, and the boundary is an explicit curve. In the far tail the right
constant is $\kappa-1=\Var(X^2)$, not $\kappa=\E X^4$: the tight bound
$(\kappa-1)/((t^2-1)^2+\kappa-1)$ improves the fourth-moment Markov bound by the factor
$\kappa/(\kappa-1)$, unboundedly as $\kappa\downarrow1$. And one-sidedness --- the entire
content of Cantelli's improvement over Chebyshev --- is a low-moment phenomenon: on the whole
tail regime the one-sided and two-sided worst cases coincide.

\paragraph{Proof degree as a coordinate.} \cref{thm:proofdegree} is, to our knowledge, the
first exact proof-degree phase diagram for a probability inequality: for each $(t,\kappa)$ we
know the least SOS degree that proves the tight constant, and the transition happens along
the explicit curve $\kappa=\kc(t)$. By the proofs-to-algorithms correspondence, degree here
is not a syntactic curiosity: degree-$4$ SOS proofs of tail bounds are exactly the
concentration facts available to degree-$4$ estimation algorithms
\citep{kotharisteurer2017,hopkinsli2018}. The map quantifies, in the simplest nontrivial
setting, where the cheaper proof system is already optimal and where paying for degree $4$
is mathematically necessary. Extending the map (to higher moments, to support-constrained
classes, to multivariate events at fixed degree) is a program we believe is now clearly
posed and mechanically executable.

\paragraph{Limitations.} Three are worth stating plainly. First, Regime IIIb has no closed
form of the Regime I--IIIa kind: those formulas are nested \emph{square} roots of rational
functions, any value of which lies in an iterated quadratic extension of $\Q$ at a rational
parameter point and hence has algebraic degree a power of $2$; the degree-six irreducible
minimal polynomials at rational points (independently reproduced by Gr\"obner elimination at
two distinct points) therefore \emph{prove} no such formula exists. Expressibility by
higher-order radicals is not ruled out (that would require the Galois group); so the theorem
there is an exact algebraic characterization plus
certified sandwiches on grids, in the same sense in which \citet{bertsimaspopescu2005} treat
their non-closed-form regions. Second, our global claim on IIIb (solvability of
\eqref{eq:systemS} at \emph{every} interior parameter) is verified at every point we
evaluated but not proven in general; the paper's statements are worded accordingly, and no
downstream result depends on global solvability. Third, some suprema/infima are approached
rather than attained (tight Paley--Zygmund; three-moment cells with escaping mass); the
values are still exact, and the phenomenon, extremal mass escaping to a boundary or to
infinity, is itself part of the landscape and is flagged per instance in the artifact
table.

\paragraph{Methodological remark.} Every inequality in this paper can be checked by hand:
each is a polynomial identity plus sign conditions on small matrices. The pipeline's value is
not that it replaces proof but that it \emph{finds} the right identities, including regime
boundaries ($\kss$, $\ks$, $\tauofk$) that we did not anticipate, and that its
verification layer (exact arithmetic, independent code path, fresh process, mutation-tested)
keeps the human's trust anchored to something narrower and more auditable than a numerical
solver's output. The one benchmark where the machine disagreed with the textbook
(Paley--Zygmund) it disagreed correctly.

\paragraph{Open directions.} (i) The skewness-pinned surface $(\gamma,\kappa)$: our grid is
computed and largely certified; the regime geometry above the $\gamma=0$ slice (where we
recover Zelen) awaits the same complete treatment. (ii) The full proof-degree map of the
two-sided problem. (iii) Support-constrained
classes ($X\ge0$ with kurtosis): the same machinery applies verbatim. (iv) Multivariate
cells at degree $4$ (maximal correlations, union-bound slack for $k\le3$ correlated
variables, Hanson--Wright constants at the information available to a degree-$4$ proof),
where exactness is no longer automatic and the proof-degree map becomes the object of
interest itself; certificate frameworks built for exactly this lossy regime, such as SONC
\citep{dressleriliman2017}, may be a natural tool there. (v) Higher even moments: for
the classes $\E X^{2m}\le\kappa_{2m}$ we \emph{conjecture} that the one-/two-sided
collapse of \cref{cor:collapse} persists beyond the corresponding tongue: the
mechanism, an even optimal certificate that cannot distinguish the two events, is not
specific to degree $4$. (vi) Formal export: the certificates are already in a shape
(rational identities plus $LDL^\top$ witnesses) that a proof assistant can consume
\citep{harrison2007}.

\bibliographystyle{plainnat}
\bibliography{refs}

\clearpage
\appendix
\begin{center}
\textbf{\huge Appendix}
\end{center}
\section{Glossary of symbols and dependency graph}
\label{app:glossary}

\begin{center}\small
\vspace{-1.1em}%
\captionsetup{skip=4pt}%
\renewcommand{\arraystretch}{1.1}%
\begin{tabular}{>{\raggedright\arraybackslash}p{0.13\dimexpr\linewidth-6\tabcolsep\relax}%
                >{\raggedright\arraybackslash}p{0.33\dimexpr\linewidth-6\tabcolsep\relax}%
                >{\raggedright\arraybackslash}p{0.54\dimexpr\linewidth-6\tabcolsep\relax}}
\toprule
symbol & definition & role \\
\midrule
$\Ckap$ & $\{\E X{=}0,\ \E X^2{=}1,\ \E X^4\le\kappa\}$ & the bounded-kurtosis class
  (\cref{asm:class}); $\Ckap^{=}$: same with $\E X^4=\kappa$ (\cref{app:eqclass}) \\
$\Vone(t,\kappa)$ & $\sup_{X\in\Ckap}\P(X\ge t)$ & one-sided worst-case tail
  (\cref{thm:main-map}) \\
$\Vtwo(t,\kappa)$ & $\sup_{X\in\Ckap}\P(|X|\ge t)$ & two-sided worst case
  (\cref{thm:twosided}) \\
$b(\kappa)$ & $\tfrac12(\sqrt{\kappa+3}-\sqrt{\kappa-1})$ & left tongue boundary; also the
  plateau witness's positive atom \\
$c(\kappa)$ & $\tfrac12(\sqrt{\kappa+3}+\sqrt{\kappa-1})=1/b(\kappa)$ & right tongue
  boundary; onset of the tail regime and of the collapse \\
$u,\ s$ & $\sqrt{\kappa-1},\ \sqrt{\kappa+3}$ & surd shorthands ($s^2-u^2=4$, $b=\frac{s-u}2$,
  $c=\frac{s+u}2$) \\
$\kc(t)$ & $t^2-1+1/t^2$ & fourth moment of the Cantelli pair at threshold $t$; tongue
  criterion $\kappa\ge\kc(t)$ and the proof-degree phase boundary \\
$\ks$ & $3/2$ & largest $\kappa$ with a nonempty plateau ($\tau(\ks)=b(\ks)=1/\sqrt2$) \\
$\kss$ & $(3\sqrt3-3)/2\approx1.098$ & below it the plateau reaches $t=0$
  ($\tau(\kss)=0$); threshold of the endpoint formula (\cref{prop:endpoint}) \\
$\tauofk$ & $\sqrt{u(s+u)}-b(\kappa)$ & lower edge of the plateau window (for
  $\kappa\le\ks$) \\
$p$ & $(\kappa-1)/((t^2-1)^2+\kappa-1)$ & Regime II value \eqref{eq:regimeII} \\
$u_\ast,\ a$ & $u_\ast=\frac{1-pt^2}{1-p}$, $a=\sqrt{u_\ast}$ & Regime II certificate root
  and witness atom locations $\pm a$ \\
$w_\pm$ & $\frac{(1-p)\mp pt/a}{2}$ & Regime II witness weights at $\pm a$ \\
$q(x)$ & $\lambda_0+\lambda_1x+\lambda_2x^2+\lambda_4x^4$ & certificate polynomial
  (\cref{def:cert}); $\lambda_4\ge0$, no $x^3$ term \\
$\sigma_0,\sigma_1$ & SOS polynomials & Markov--Luk\'acs pieces:
  $q-1=\sigma_0+(x-t)\sigma_1$ on $[t,\infty)$ (\cref{lem:sos-ray}) \\
$\mu^\star$, $\mu_C$ & extremal laws & regime witnesses; $\mu_C$ = Cantelli pair
  $\frac1{1+t^2}\delta_t+\frac{t^2}{1+t^2}\delta_{-1/t}$ (\cref{lem:rigidity}) \\
$\PDeg(t,\kappa)$ & least certificate degree & proof degree (\cref{def:pd,thm:proofdegree}) \\
(S) & system \eqref{eq:systemS} & Regime IIIb characterization: atoms $\{-c,t,b\}$ +
  tangency conditions \\
$k_2,\ \lambda_4$ & $2b(c-b)$, $\bigl((c-b)(b+c)^3\bigr)^{-1}$ & parameters of the
  IIIa/IIIb certificate shape $\lambda_4[(x^2-c^2)^2+k_2(x+c)^2]$ \\
$t_p(\kappa)$ & $\bigl(1+\sqrt{(\kappa-1)\tfrac{1-p}{p}}\bigr)^{1/2}$ (kurtosis branch) &
  exact worst-case quantile / chance-constraint threshold (\cref{cor:quantile});
  $t^{(2)}_p$: its two-sided variant \\
$\kappa_m$, $\kappa_B$ & $3+(\kappa-3)/m$ & fourth moment of a standardized $m$-average
  (\cref{lem:avgkurt}); drives the confidence-interval and median-of-means constants \\
$\sigma_u$ & $\langle u,\Sigma u\rangle^{1/2}$ & directional deviation in the
  certifiable-kurtosis bound (\cref{prop:certifiable}) \\
\bottomrule
\end{tabular}
\captionof{table}{Symbols in order of first substantial use. In Regime II discussions, $u_\ast$ is
written to avoid clashing with $u=\sqrt{\kappa-1}$; the theorem statement uses $u$ for it,
locally.}
\label{tab:glossary}
\end{center}

\paragraph{Dependency graph of the results.} \cref{fig:resultmap} charts the paper's own results --- the
imported facts live in the register (\cref{app:register}) and are not drawn. An arrow
$A\to B$ means the proof of $B$ uses $A$. Two lemmas are global roots: weak duality
(\cref{lem:weak-duality}) underlies every upper bound, and the tongue equivalence
(\cref{lem:tongue-equiv}) every regime boundary; their arrows are drawn once, into the
map, though every result above uses them.

\begin{center}
\resizebox{\linewidth}{!}{%
\begin{tikzpicture}[x=1cm,y=1cm,
  n/.style={draw, rounded corners=2pt, align=center, font=\footnotesize,
            inner sep=3.5pt, fill=white},
  root/.style={n, fill=gray!13},
  regn/.style={n, fill=blue!8},
  topn/.style={n, fill=orange!14},
  sup/.style={n, fill=green!10},
  appn/.style={n, fill=violet!9},
  arr/.style={-{Stealth[length=2mm]}, semithick, black!60}]
\node[regn] (rI)    at (0.6,2.4)  {\ref{thm:main-map}(I)\\ tongue};
\node[regn] (rII)   at (2.9,2.4)  {\ref{thm:main-map}(II)\\ tail};
\node[regn] (rIIIa) at (5.4,2.4)  {\ref{thm:main-map}(IIIa)\\ plateau};
\node[regn] (rIIIb) at (8.0,2.4)  {\ref{thm:main-map}(IIIb)\\ system (S)};
\coordinate (dip) at (4.3,1.62);
\node[draw, dashed, rounded corners=3pt, inner sep=7pt,
      fit=(rI)(rII)(rIIIa)(rIIIb)(dip)] (box) {};
\node[font=\footnotesize\itshape, fill=white, inner sep=2pt, anchor=south]
      at ([yshift=1pt]box.north) {\cref{thm:main-map} (the map)};
\draw[arr] (rII) -- (rIIIa);
\draw[arr] (rIIIa) -- (rIIIb);
\draw[arr] (rI.south) -- ++(0,-0.34) -| (rIIIb.south);
\node[root] (te) at (2.0,0.1)  {\cref{lem:tongue-equiv}\\ tongue equivalence};
\node[root] (wd) at (6.2,0.1)  {\cref{lem:weak-duality}\\ weak duality};
\draw[arr] (te.north) -- ($(box.south west)!0.22!(box.south east)$);
\draw[arr] (wd.north) -- ($(box.south west)!0.64!(box.south east)$);
\node[sup] (rep) at (11.0,0.1) {\ref{lem:sos-line}, \ref{lem:sos-ray}\\ SOS representation};
\node[sup] (ar)  at (12.0,2.4) {\ref{lem:attain}, \ref{lem:rigidity}\\ attainment, rigidity};
\node[sup] (mv)  at (15.2,2.4) {Fact \ref{mv:minpoly}\\ minimal polynomial};
\node[topn] (t32) at (1.4,4.9)  {\cref{thm:twosided}\\ two-sided map};
\node[topn] (t33) at (4.5,4.9)  {\cref{cor:collapse}\\ collapse};
\node[topn] (t34) at (8.3,4.9)  {\cref{thm:proofdegree}\\ proof degree};
\node[topn] (t35) at (12.8,4.9) {\cref{prop:endpoint}, \S\ref{sec:endpoint}\\ endpoint; no closed form};
\draw[arr] (rII.north) -- (t32.south);
\draw[arr] (rII.north) -- (t33.south);
\draw[arr] (t32.east) -- (t33.west);
\draw[arr] ($(box.north west)!0.88!(box.north east)$) --
  node[left=1.5pt, font=\scriptsize, pos=0.3]{I, II, IIIa} (t34.south);
\draw[arr] (ar.north) -- (t34.south east);
\draw[arr] (rep.north) -- (t34.south);
\draw[arr] (rIIIb.north) -- (t35.south west);
\draw[arr] (mv.north) -- (t35.south east);
\node[appn] (aq)  at (2.0,7.35)  {\ref{cor:quantile}\\ quantiles};
\node[appn] (aav) at (4.8,7.35)  {\ref{lem:avgkurt}\\ averages, CI};
\node[appn] (ama) at (7.6,7.35)  {\ref{cor:margin}\\ margin};
\node[appn] (adi) at (10.5,7.35) {\ref{prop:certifiable}\\ directional};
\node[draw, dashed, rounded corners=3pt, inner sep=7pt,
      fit=(aq)(aav)(ama)(adi)] (appsbox) {};
\node[font=\footnotesize\itshape, fill=white, inner sep=2pt, anchor=south]
      at ([yshift=1pt]appsbox.north) {\cref{sec:consequences} (applications)};
\draw[arr] (t32.north) -- (aq.south);
\draw[arr] (t33.north) -- (adi.south);
\draw[arr] ($(box.north west)!0.70!(box.north east)$)
        -- ($(appsbox.south west)!0.48!(appsbox.south east)$);
\end{tikzpicture}}

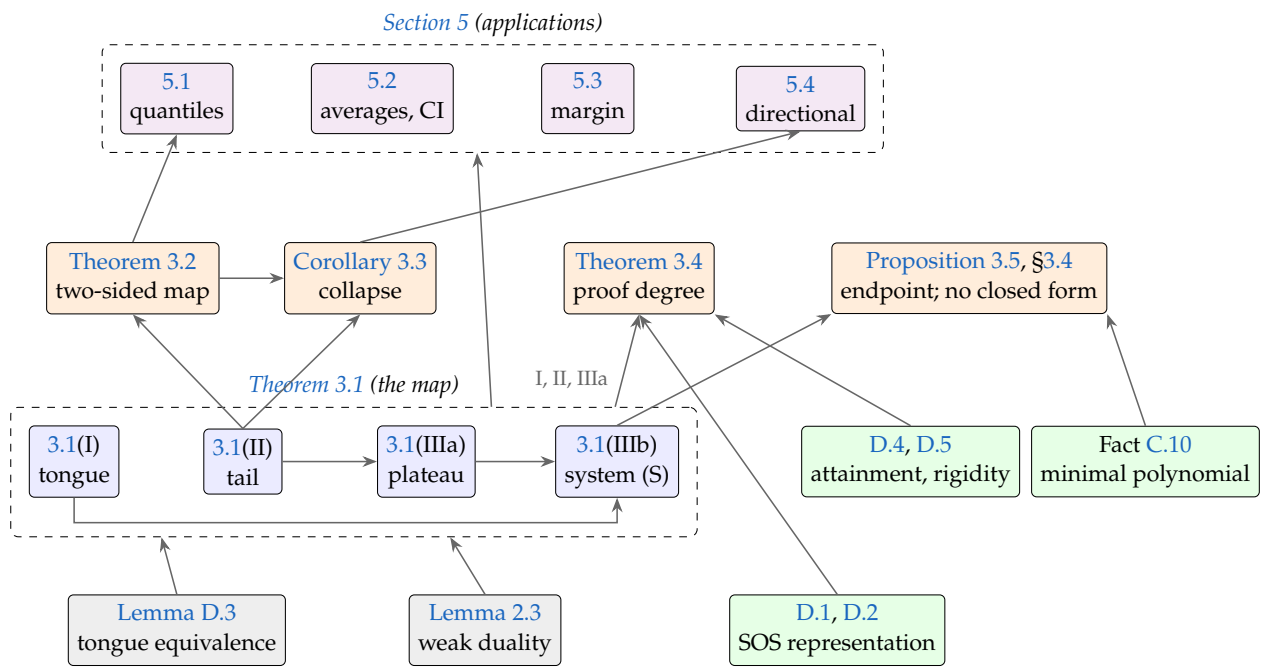
\captionof{figure}{Dependency graph of the paper's results (ours only). An arrow means
``is used in the proof of''. The dashed box is \cref{thm:main-map}, whose four regime
clauses depend on one another as drawn; the two root lemmas at the lower left are used
by everything above them (arrows drawn once, into the map). The endpoint node also
carries the no-closed-form statement of \cref{sec:endpoint}, which rests on the
machine-verified minimal polynomial (\cref{mv:minpoly}); its small-$\kappa$ clause
additionally reads off the IIIa plateau value. Top box: the applications of
\cref{sec:consequences}; every one of them uses the map, whose arrow is drawn once,
into their box. \cref{cor:quantile} additionally inverts the two-sided map,
\cref{prop:certifiable} uses the collapse, and the attainment clause of
\cref{cor:margin} uses \cref{lem:attain} (arrow omitted for legibility).}
\label{fig:resultmap}
\end{center}

\section[Background: moment problems, sums of squares, and SDP]{Background:
the moment problem, sums of squares, and semidefinite programming}
\label{app:background}

This appendix is a self-contained primer on the dictionary the paper translates through:
a sharp probability inequality, an optimization problem over measures (the \emph{moment
problem}), a polynomial-nonnegativity certificate (\emph{sums of squares}), and a
\emph{semidefinite program}. Nothing here is new; the standard sources are
\citet{isii1962}, \citet{karlinstudden1966}, \citet{bertsimaspopescu2005},
\citet{lasserre2001}, and the monographs \citet{lasserre2010,lasserre2015,nie2023moment}, and a reader
fluent in them can skip ahead. Each step is
illustrated on the smallest instance in the paper: Cantelli's inequality, which the map
reproduces as Regime I of \cref{thm:main-map}.

\paragraph{The generalized moment problem.} Every Chebyshev-type inequality answers the
same underlying question: among all laws consistent with a few prescribed moments, how
much probability can a given event carry? The classical move, going back to Chebyshev
and Markov, is to treat the unknown law itself as the optimization variable; the
question then becomes a linear program over probability measures, and the sharp
constant is its optimal value.

\begin{definition}[Moment program]\label{def:bg-gmp}
Given a Borel event $S\subseteq\R$ and finitely many prescribed moments (values or
bounds), the \emph{generalized moment problem} is the linear program over probability
measures that extremizes $\mu(S)$ subject to the moment constraints. The quantity
$\Vone(t,\kappa)$ of \cref{sec:prelim} is the instance
\begin{equation}\label{eq:bg-primal}
\sup_{\mu}\;\;\mu(S)
\qquad\text{s.t.}\quad
\textstyle\int\!d\mu=1,\;\;
\int\!x\,d\mu=0,\;\;
\int\!x^2\,d\mu=1,\;\;
\int\!x^4\,d\mu\le\kappa,
\end{equation}
with $S=[t,\infty)$.
\end{definition}

The unknown is the \emph{law} $\mu$, not a density: the feasible set is convex, and its
extreme points are measures supported on finitely many atoms --- for four moment
conditions, never more than a handful, and for the class $\Ckap$ never more than three
atoms anywhere in this paper. This is the classical theory of principal representations
for Tchebycheff systems \citep{karlinstudden1966,kreinnudelman1977}; it is why every
worst case exhibited in the paper is a two- or three-point law, and why a search for
extremal distributions may be confined to small atomic configurations from the start.
Whether the supremum is \emph{attained} is a separate question, answered here
constructively: every regime of the map exhibits a law achieving its value, and
\cref{lem:attain} records the abstract attainment argument that the proof-degree theorem
additionally needs.

\paragraph{Duality: polynomials above indicators.} The dual of \eqref{eq:bg-primal}
replaces measures by polynomial test functions, and its shape can be read off the
Lagrangian: attach one multiplier to each moment condition of \eqref{eq:bg-primal} ---
$\lambda_0$ to the mass, $\lambda_1$ to the mean, $\lambda_2$ to the second moment,
$\lambda_4$ to the kurtosis budget --- and collect what multiplies $\mu$.

\begin{definition}[Dual polynomial]\label{def:bg-dual}
A \emph{dual feasible point} for \eqref{eq:bg-primal} is a polynomial
$q(x)=\lambda_0+\lambda_1x+\lambda_2x^2+\lambda_4x^4$ with
\[
q(x)\;\ge\;\mathbf 1_S(x)\quad\text{for all }x\in\R,
\qquad\lambda_4\ge 0;
\]
its \emph{value} is the constraint data paired with the multipliers,
$\lambda_0+\lambda_2+\lambda_4\kappa$.
\end{definition}

The sign rules are the standard ones: equality
constraints get free multipliers, the one-sided budget $\E X^4\le\kappa$ gets
$\lambda_4\ge0$, and a moment that appears in no constraint gets no multiplier at all ---
$\Ckap$ says nothing about $\E X^3$, so $q$ has no cubic term. This last, automatic
omission is the entire difference between our skewness-free class and the pinned-moment
classics of \cref{app:priorwork}, where a pinned $\E X^3$ buys the dual an extra cubic
degree of freedom. Taking expectations of the pointwise inequality gives weak duality in
one line (\cref{lem:weak-duality}): every dual feasible $q$ is a proof of an upper bound.
At an optimal pair the measure lives exactly
where the polynomial touches the indicator: atoms of $\mu$ off the event sit at zeros of
$q$, atoms on the event at points where $q=1$ (complementary slackness). This is how
extremal laws are read off certificates throughout the paper, and conversely why optimal
Gram matrices are rank-deficient (the touching forces zeros), a fact that matters for the
rounding step below. Strong duality (no gap) holds under mild conditions
\citep{isii1962,bertsimaspopescu2005}; the paper never needs to invoke it abstractly,
because every stated bound comes as a \emph{matched pair} --- certificate above, attaining
distribution below.

\paragraph{One variable is exact: sums of squares.} Dual feasibility asks for proofs of
two pointwise inequalities: $q\ge 0$ on $\R$ and $q-1\ge 0$ on $S$.

\begin{definition}[Sum of squares]\label{def:bg-sos}
A polynomial $p\in\R[x]$ is a \emph{sum of squares} (SOS) if $p=g_1^2+\dots+g_m^2$ for
finitely many $g_i\in\R[x]$. An SOS polynomial is nonnegative on $\R$ by inspection; the
converse direction is the substance.
\end{definition}

In several variables,
deciding nonnegativity is intractable, and replacing ``nonnegative'' by ``SOS'' gives
the tractable but generally lossy relaxations of the
Lasserre/Parrilo hierarchy \citep{lasserre2010,fleming2019}. In \emph{one} variable the
relaxation is lossless: a univariate polynomial nonnegative on $\R$ is a sum of two
squares, and one nonnegative on the ray $[t,\infty)$ of degree $\le 4$ has the exact
representation
\[
q-1\;=\;\sigma_0+(x-t)\,\sigma_1,
\qquad \sigma_0,\sigma_1\text{ SOS},\;\;
\deg\sigma_0\le4,\;\deg\sigma_1\le2
\]
(the Markov--Luk\'acs theorem; see \citealp{nesterov2000}). The paper does not import
this as a black box: the two forms it uses are proved from the fundamental theorem of
algebra as \cref{lem:sos-line,lem:sos-ray}, and the register of what \emph{is} imported
is \cref{app:register}. Consequently the degree-$4$ program computes $\Vone$
\emph{exactly} --- there is no relaxation gap to apologize for (\cref{rem:exactness}) ---
and ``proof degree'' as studied in \cref{thm:proofdegree} is a property of the inequality
itself, not of a hierarchy level.

\paragraph{SOS search is a semidefinite program.} The bridge to computation is the Gram
representation. A polynomial $p$ of degree $2d$ is a sum of squares if and only if there
is a positive semidefinite matrix $G$, indexed by the monomial vector
$z(x)=(1,x,\dots,x^d)^{\!\top}$, with
\begin{equation}\label{eq:bg-gram}
p(x)\;=\;z(x)^{\!\top} G\,z(x)
\qquad\Longleftrightarrow\qquad
\sum_{i+j=k} G_{ij}\;=\;p_k\quad(0\le k\le 2d),
\end{equation}
where $p_k$ is the coefficient of $x^k$: the identity on the left is, coefficient by
coefficient, the system of affine equations on the right. Writing $G=LL^{\!\top}$
(Cholesky/$LDL^{\!\top}$) recovers the explicit squares as the entries of
$L^{\!\top}z(x)$.

\begin{definition}[Semidefinite program]\label{def:bg-sdp}
A \emph{semidefinite program} (SDP) is the optimization of a linear functional over the
intersection of the cone of positive-semidefinite matrices with an affine subspace.
\end{definition}

SDPs are convex, and interior-point solvers deliver ten-plus digits in well under a
second at our sizes. By \eqref{eq:bg-gram}, the certificate search of
\cref{def:cert} is exactly such a program. Concretely, for the one-sided event the
unknowns are the multipliers $\lambda_0,\lambda_1,\lambda_2,\lambda_4$ and three Gram
matrices: one for the global inequality $q\ge0$ in the basis $z=(1,x,x^2)$ (size
$3\times3$), and one each for $\sigma_0$ ($3\times3$) and $\sigma_1$ (basis $(1,x)$,
$2\times2$) in the ray representation of $q-1$; the two-sided event contributes one such
representation per ray. The constraints are the affine coefficient matches
\eqref{eq:bg-gram} tying every Gram entry to the $\lambda_i$, plus $G\succeq0$ per block;
the objective $\lambda_0+\lambda_2+\lambda_4\kappa$ is linear. This is the program
\textsc{LemmaForge}
compiles and solves (\cref{sec:pipeline}); checking a \emph{given} certificate needs no
optimization at all --- one verifies the affine identities by expansion and the
$G\succeq0$ conditions in exact arithmetic by the pivot or coefficient criteria
(\cref{app:psd}), which is precisely what the independent verifier does.

\paragraph{The dictionary, run on Cantelli.} Take the two-moment class ((A1)--(A2) only)
and $S=[t,\infty)$. The degree-$2$ certificate of \cref{thm:main-map}(I) is
$q(x)=(tx+1)^2/(1+t^2)^2$, and the whole loop is visible by hand. \emph{Dual
feasibility:} $q\ge0$ on $\R$ is the single square \eqref{eq:bg-gram} with rank-one Gram
$G=\bigl(\begin{smallmatrix}1&t\\ t&t^2\end{smallmatrix}\bigr)/(1+t^2)^2\succeq0$ in the
basis $z=(1,x)$, and $q-1\ge0$ on
$[t,\infty)$ is the identity
\[
(tx+1)^2\;=\;t^2(x-t)^2\;+\;2t(1+t^2)\,(x-t)\;+\;(1+t^2)^2,
\]
i.e.\ $q-1=\sigma_0+(x-t)\sigma_1$ with $\sigma_0=\bigl(t(x-t)\bigr)^2/(1+t^2)^2$ a
square and $\sigma_1=2t/(1+t^2)$ a nonnegative constant. \emph{Value:}
$\E q(X)=(\,t^2\E X^2+2t\E X+1\,)/(1+t^2)^2=1/(1+t^2)$. \emph{Extremal law from the
touching set:} $q$ vanishes only at $x=-1/t$ and equals $1$ on $S$ only at $x=t$, so the
worst case must be the two-point law on $\{-1/t,\,t\}$; matching $\E X=0$, $\E X^2=1$
gives masses $t^2/(1+t^2)$ and $1/(1+t^2)$, and the latter is the tail mass: the bound is
attained. \emph{Where the bounded-kurtosis class enters:} this witness has
$\E X^4=t^2-1+1/t^2=\kc(t)$, so the Cantelli pair survives in $\Ckap$ exactly when
$\kappa\ge\kc(t)$ --- the tongue of \cref{fig:phase}. Off the tongue the budget
\textup{(A3)} excludes the witness, the value drops strictly below $1/(1+t^2)$, and
certifying the new value requires the quartic term: proof degree jumps from $2$ to $4$
(\cref{thm:proofdegree}). The same walk at degree $4$, with the kurtosis constraint
active and all arithmetic in rationals, is carried out at $(t,\kappa)=(2,3)$ in
\cref{app:worked-II}.

\paragraph{From floating point to exact certificates.} An interior-point SDP solver
returns $\lambda$ and Gram matrices to ten-ish digits, which proves nothing. The repair,
following \citet{peyrlparrilo2008}, is to rationalize the numerical solution and project
it, in exact arithmetic, back onto the affine subspace \eqref{eq:bg-gram}; if the
projected Gram matrices remain positive semidefinite --- checked exactly, never
numerically (\cref{app:psd}) --- the identity and hence the bound are theorems. One
subtlety is generic here: \emph{optimal} Gram matrices are rank-deficient (the touching
zeros above), so they sit on the boundary of the PSD cone, where naive rounding falls
off. The pipeline's $\varepsilon$-retreat (\cref{app:pipeline}) trades an explicitly
recorded $\varepsilon$ of the bound for a strictly interior, roundable target. Atom
recovery from the equality duals is the flat-data phenomenon of
\citet{curtofialkow1996}, made algorithmic for the truncated moment problem by
\citet{nie2014truncated}.

\section{External and auxiliary facts}
\label{app:register}

This section lists, exhaustively, everything the paper's results rest on that is not proved
in \cref{app:proofs} itself. The inputs fall into three classes: \emph{external theorems}
(classical results we cite and use), \emph{auxiliary facts} (small standard facts we use and,
for completeness, prove here), and \emph{machine-verified exact facts} (finite computations
in exact arithmetic, reproducible from the artifact). Numerical solvers appear in none of
the three classes: the SDP and LP solvers were used to \emph{discover} the results, and no
theorem depends on their output.

\subsection{External theorems}

Four classical theorems are used; the first three summarize the weak-convergence
toolkit from the classical monograph of \citet{billingsley1999}; throughout,
$\mu_n\Rightarrow\mu$ denotes weak convergence of probability measures
($\int f\,d\mu_n\to\int f\,d\mu$ for every bounded continuous $f$). Each is
stated in the exact form in which it is applied, followed by the audit note recording
where it enters.

\begin{fact}[Prokhorov]\label{fact:prokhorov}
A tight sequence of probability measures on $\R$ has a weakly convergent subsequence.
\end{fact}

\begin{fact}[Portmanteau]\label{fact:portmanteau}
If $\mu_n\Rightarrow\mu$, then $\limsup_n\mu_n(F)\le\mu(F)$ for every closed
$F\subseteq\R$, and $\liminf_n\int g\,d\mu_n\ge\int g\,d\mu$ for every bounded
continuous $g\ge0$.
\end{fact}

\begin{fact}[Convergence under uniform integrability]\label{fact:ui}
If $\mu_n\Rightarrow\mu$ and $f$ is continuous with
$\sup_n\int|f|^{1+\delta}\,d\mu_n<\infty$ for some $\delta>0$, then
$\int f\,d\mu_n\to\int f\,d\mu$.
\end{fact}

\noindent\emph{Used in:} the attainment lemma (\cref{lem:attain}) only ---
\cref{fact:portmanteau} applied to the truncations $x^4\wedge M$ (then $M\to\infty$ by
monotone convergence), \cref{fact:ui} to $f=x$ and $f=x^2$ (dominated via
$\E X^4\le\kappa$). Consequently the three facts reach exactly three results: the
strictness step of the proof-degree theorem (\cref{thm:proofdegree}, via attainment $+$
rigidity), the equality-class proposition (\cref{prop:eqclass}), and the attainment
clause of the margin bound (\cref{cor:margin}). The regime formulas
of \cref{thm:main-map,thm:twosided} and their tightness need no measure theory at all:
every upper bound is a pointwise polynomial inequality integrated against four moment
values, and every lower bound is an explicit finitely supported law.

\begin{fact}[Fundamental theorem of algebra, real form]\label{fact:fta}
Every nonconstant $p\in\R[x]$ factors as
$p(x)=c\prod_i(x-r_i)^{m_i}\prod_j q_j(x)$ with $c\in\R$, finitely many real roots
$r_i$ of multiplicities $m_i$, and monic quadratics $q_j$ without real roots (hence
positive definite on $\R$).
\end{fact}

\noindent\emph{Used in:} the two sum-of-squares representation lemmas
(\cref{lem:sos-line,lem:sos-ray}) and, through them, nowhere else: the certificates of
the theorems are \emph{explicit}, so the representation lemmas are needed only for the
completeness direction (``degree 4 always suffices,'' \cref{rem:exactness}) and the
degree-$4$ sufficiency clause of \cref{thm:proofdegree}.

Nothing else external is used. In particular, no generalized-moment-problem duality
(\citealp{isii1962}; discussed in \cref{app:priorwork}) is invoked: strong duality is never
needed because each regime exhibits a matched primal--dual pair, and weak duality
(\cref{lem:weak-duality}) is proved inline in five lines.

\subsection{Auxiliary facts, with proofs}

\begin{fact}[Moment ordering]\label{fact:momord}
If $\E X^4\le\kappa<\infty$, then $\E|X|^k<\infty$ for $0\le k\le4$.
\end{fact}

\begin{proof}
$|x|^k\le1+x^4$ pointwise: for $|x|\le1$ the left side is at most $1$; for $|x|>1$ it is
at most $x^4$.
\end{proof}

\noindent\emph{Used in:} \cref{lem:weak-duality} (existence of all moments through
degree $4$), the uniform-integrability step of \cref{lem:attain}, and the vanishing
cross terms in the proof of \cref{lem:avgkurt}.

\begin{fact}[Variance nonnegativity and its equality case]\label{fact:jensen}
$\E X^4\ge(\E X^2)^2$, with equality if and only if $X^2$ is a.s.\ constant.
\end{fact}

\begin{proof}
$\E X^4-(\E X^2)^2=\Var(X^2)\ge0$, and $\Var(Y)=\E(Y-\E Y)^2=0$ forces $Y=\E Y$ a.s.
\end{proof}

\noindent\emph{Used in:} the ``Jensen'' step of \cref{sec:prelim}: $\kappa\ge1$ is the
full parameter range, and $\CkapOf1$ is the symmetric two-point law on $\{\pm1\}$.

\begin{fact}[Real spectrum of symmetric matrices]\label{item:spec}
A real symmetric matrix has only real eigenvalues.
\end{fact}

\begin{proof}
If $Gv=\lambda v$ with $0\ne v\in\C^n$, then $\bar v^\top Gv=\lambda\,\|v\|^2$; the left
side equals its own complex conjugate (as $G$ is real symmetric,
$\overline{\bar v^\top Gv}=v^\top G\bar v=\bar v^\top Gv$), so $\lambda$ is real.
\end{proof}

\noindent\emph{Used in:} the coefficient criterion (\cref{lem:charpoly}), i.e.\ the
verifier's second PSD check.

\begin{fact}[Degrees in iterated quadratic extensions]\label{item:tower}
Let $\Q=F_0\subseteq F_1\subseteq\dots\subseteq F_k$ with $F_{i+1}=F_i(\sqrt{d_i})$,
$d_i\in F_i$. Then $[F_k:\Q]$ is a power of $2$, and every $V\in F_k$ has algebraic
degree over $\Q$ dividing $[F_k:\Q]$ --- in particular, never $6$.
\end{fact}

\begin{proof}
$[F_{i+1}:F_i]\le2$, since $\{1,\sqrt{d_i}\}$ spans $F_{i+1}$ over $F_i$. Degrees
multiply in towers: if $\{a_p\}$ is an $F$-basis of $K$ and $\{b_q\}$ a $K$-basis of
$L$, the products $\{a_pb_q\}$ form an $F$-basis of $L$ (spanning is direct;
independence follows by collecting coefficients of the $b_q$). Hence
$[F_k:\Q]=\prod_i[F_{i+1}:F_i]$ is a power of $2$. For $V\in F_k$, the tower
$\Q\subseteq\Q(V)\subseteq F_k$ gives $[\Q(V):\Q]\mid[F_k:\Q]$, and the algebraic
degree of $V$ equals $[\Q(V):\Q]$.
\end{proof}

\noindent\emph{Used in:} the no-nested-square-root statement for Regime IIIb
(\cref{sec:endpoint} and the discussion): a value with an irreducible degree-$6$ minimal
polynomial over $\Q$ lies in no iterated quadratic extension, so no formula built from
rational operations and square roots can produce it at that rational parameter point.

\begin{fact}[Intermediate value theorem]\label{fact:ivt}
A continuous $f:[a,b]\to\R$ with $f(a)f(b)<0$ has a zero in $(a,b)$.
\end{fact}

\begin{proof}
Bisect: if the midpoint $m$ of $[a_n,b_n]$ has $f(m)=0$ we are done; otherwise one half
retains the sign change. The nested intervals have lengths $\to0$, so by completeness of
$\R$ they share a point $c$; if $f(c)\neq0$, continuity would give $f$ a constant sign on
a neighborhood of $c$, which contains some $[a_n,b_n]$ --- contradicting
$f(a_n)f(b_n)<0$.
\end{proof}

\noindent\emph{Used in:} the existence of the mixing parameter $\beta(\varepsilon)$ in
\cref{prop:eqclass}, and root locations in the regime proofs where a sign change is
exhibited.

\subsection{Machine-verified exact facts}

Three ingredients are finite computations carried out in exact arithmetic; they are
reproducible from the artifact (\cref{app:pipeline}) and are the only places where a
claim's evidence is a computation rather than a displayed argument. None of the hand
proofs of \cref{app:proofs} uses them: deleting all three would leave every proof in
this paper intact, at the price of exactly the three claims recorded under
\emph{load-bearing for} below.

\begin{mvfact}[Minimal polynomials in Regime IIIb]\label{mv:minpoly}
At $(t,\kappa)=(\tfrac12,2)$, the minimal polynomial of $\Vone$ over $\Q$ is the
irreducible degree-$6$ polynomial displayed in \cref{sec:endpoint}; at
$(t,\kappa)=(0.4,2.5)$, the minimal polynomial is likewise irreducible of degree $6$.
\end{mvfact}

\noindent\emph{Evidence:} exact elimination from \eqref{eq:systemS}, computed twice with
independent code (filtered resultants; lexicographic Gr\"obner basis), agreeing exactly
(\cref{app:minpoly}). \emph{Load-bearing for:} the ``no closed form of the Regime
I--IIIa kind'' statement, via \cref{item:tower} above.

\begin{mvfact}[Certified-instance verdicts]\label{mv:verdicts}
Each \textsc{tight} verdict in \cref{tab:instances} attests an exact certificate for the
stated value together with an exact extremal witness attaining it; each
\textsc{bound}($\varepsilon$) verdict attests an exact certificate for the stated value
plus $\varepsilon$, with $\varepsilon\le10^{-6}$.
\end{mvfact}

\noindent\emph{Evidence:} every attestation is a fresh-process run of the independent
checker on the certificate file (checks V1--V8, \cref{app:pipeline}), whose PSD steps
rest on the two criteria proved in \cref{app:psd}. \emph{Load-bearing for:} the
per-instance certification claims and nothing in the closed-form theorems (whose proofs
are in \cref{app:proofs}, independently of any certificate file).

\begin{mvfact}[IIIb solvability at grid points]\label{mv:solvability}
At every Regime IIIb grid point of \cref{tab:instances}, system \eqref{eq:systemS} is
solvable and its solution satisfies the inequalities of \cref{prop:IIIb}.
\end{mvfact}

\noindent\emph{Evidence:} high-precision solutions of \eqref{eq:systemS} with primal
$=$ dual agreeing to $30$ digits; the inequalities of \cref{prop:IIIb} checked exactly
at the rationalized data. \emph{Load-bearing for:} the solvability clause of
\cref{thm:main-map}(IIIb), which is stated as verified-where-evaluated (the honest
scope; see the discussion of limitations).

The symbolic identity battery (\texttt{scripts/verify\_theorem\_symbolic.py}, tags quoted
throughout \cref{app:proofs}) is deliberately \emph{not} on this list: every identity it
checks also carries a hand proof in \cref{app:proofs} or \cref{app:worked}, so it is
redundant assurance --- a second, mechanical reader --- rather than a dependency. The
trusted base of the machine layer, stated once: sympy's exact rational and algebraic
arithmetic, plus the $\sim$370 audited lines of \texttt{forge/verify.py}.

\section{Deferred proofs}
\label{app:proofs}

Throughout the appendix, $u=\sqrt{\kappa-1}$ and $s=\sqrt{\kappa+3}$, so $s^2-u^2=4$,
$b=b(\kappa)=(s-u)/2$, $c=c(\kappa)=(s+u)/2$, $bc=1$, $b+c=s$, $c-b=u$. Every polynomial
identity asserted below is also machine-checked symbolically (script
\texttt{scripts/verify\_theorem\_symbolic.py}, tag in brackets), but each admits a one-line
hand verification by expansion.

\subsection{The two representation facts, proved}
\label{app:mlfacts}

The paper's certificates rest on exactly two classical facts about univariate polynomials
(used in \cref{def:cert} and \cref{rem:exactness}). For completeness we prove both; nothing
else in the paper is imported. Throughout, ``SOS'' means a finite sum of squares of real
polynomials.

\begin{lemma}[Nonnegativity on $\R$]\label{lem:sos-line}
If $p\in\R[x]$ satisfies $p\ge0$ on $\R$, then $p=g_1^2+g_2^2$ for some $g_1,g_2\in\R[x]$
with $\deg g_i\le\tfrac12\deg p$.
\end{lemma}

\begin{proof}
$p$ has even degree and positive leading coefficient $\ell$. Over $\C$, its roots split into
conjugate pairs $z_j,\bar z_j$ and real roots $r_i$, each of even multiplicity $2m_i$ (at a
real root of odd multiplicity $p$ would change sign). Set
$g=\sqrt{\ell}\,\prod_j(x-z_j)\prod_i(x-r_i)^{m_i}\in\C[x]$ and write $g=g_1+ig_2$ with
$g_1,g_2\in\R[x]$. Then $p=g\bar g=|g|^2=g_1^2+g_2^2$, and $\deg g=\tfrac12\deg p$.
\end{proof}

\begin{lemma}[Nonnegativity on a half-line]\label{lem:sos-ray}
If $p\in\R[x]$ satisfies $p\ge0$ on $[a,\infty)$, then
\[
p\;=\;\sigma_0+(x-a)\,\sigma_1,\qquad \sigma_0,\sigma_1\ \text{SOS},\quad
\deg\sigma_0\le\deg p,\ \deg\sigma_1\le\deg p-1 .
\]
\end{lemma}

\begin{proof}
Shifting $x\mapsto x+a$, take $a=0$. Consider the set
$K=\{\sigma_0+x\,\sigma_1:\sigma_0,\sigma_1\ \text{SOS}\}$. $K$ is closed under
multiplication:
\[
(\sigma_0+x\sigma_1)(\tau_0+x\tau_1)
=\underbrace{\bigl(\sigma_0\tau_0+x^2\,\sigma_1\tau_1\bigr)}_{\text{SOS}}
+x\,\underbrace{\bigl(\sigma_0\tau_1+\sigma_1\tau_0\bigr)}_{\text{SOS}} .
\]
Now factor $p$ over $\R$ into its leading coefficient $\ell$, positive-definite quadratic
factors, and real-root factors. Since $p(x)\ge0$ for large $x$, $\ell>0$, so
$\ell\in K$. Each positive-definite quadratic is a sum of two squares
(complete the square), hence in $K$. Real roots $r>0$ must have even multiplicity (else $p$
changes sign inside $(0,\infty)$), contributing squares $(x-r)^{2m}\in K$. Real roots
$r<0$ contribute factors $(x-r)=x+|r|=\underbrace{|r|}_{\sigma_0}+x\cdot\underbrace{1}_{\sigma_1}\in K$.
A root at $0$ of multiplicity $m$ contributes $x^m$, which is $(x^{m/2})^2\in K$ for $m$
even and $x\cdot(x^{(m-1)/2})^2\in K$ for $m$ odd. As a product of elements of $K$,
$p\in K$. The degree bounds follow by tracking the two components through each product:
if the factors' components obey them, so do the product's, since
$\deg(\sigma_0\tau_0+x^2\sigma_1\tau_1)\le\deg(pq)$ and
$\deg(\sigma_0\tau_1+\sigma_1\tau_0)\le\deg(pq)-1$.
\end{proof}

Mirroring $x\mapsto-x$ gives the analogous statement on $(-\infty,a]$ with multiplier
$(a-x)$. These two lemmas are the ``sound direction'' our proofs and our verifier use
(a displayed decomposition certifies nonnegativity); they are also the ``complete
direction'' behind \cref{rem:exactness} (every valid degree-$4$ bound has a certificate in
this form), which is what makes the degree-$4$ relaxation, and hence the map, exact.
The interval case $[a,b]$ (used only by the tight Paley--Zygmund benchmark certificate) is
analogous (Markov--Luk\'acs; \citealp{nesterov2000}), and the concrete decomposition used
there is displayed in the certificate itself and re-checked by the verifier.

\subsection{The tongue equivalence}

\begin{lemma}\label{lem:tongue-equiv}
For $t>0$ and $\kappa\ge1$: $\ \kappa\ge\kc(t)\iff b(\kappa)\le t\le c(\kappa)$.
Moreover $\kc(b(\kappa))=\kc(c(\kappa))=\kappa$.
\end{lemma}

\begin{proof}
$\kappa\ge\kc(t)=t^2-1+t^{-2}$ is equivalent to $t^4-(\kappa+1)t^2+1\le0$. The quadratic (in
$t^2$) has roots $b^2,c^2$: indeed $b^2c^2=(bc)^2=1$ and
$b^2+c^2=\tfrac{(s-u)^2+(s+u)^2}{4}=\tfrac{s^2+u^2}{2}=\kappa+1$, so
\[
t^4-(\kappa+1)t^2+1\;=\;(t^2-b^2)(t^2-c^2),
\]
which is $\le0$ iff $t^2\in[b^2,c^2]$ iff
$t\in[b,c]$ ($t>0$). The last claim follows by substitution [\texttt{kappa\_c(b(kappa)) ==
kappa}].
\end{proof}

\subsection{Attainment and rigidity}

The attainment lemma is the one place the paper uses measure-theoretic machinery: the
three standard weak-convergence facts of \citet{billingsley1999}, stated as
\cref{fact:prokhorov,fact:portmanteau,fact:ui} in the register (\cref{app:register}).

\begin{lemma}[Attainment]\label{lem:attain}
For every $t>0$, $\kappa\ge1$, the supremum $\Vone(t,\kappa)$ is attained by some
$\mu^\star\in\Ckap$.
\end{lemma}

\begin{proof}
Let $(\mu_n)\subset\Ckap$ with $\mu_n([t,\infty))\to\Vone$. By Chebyshev,
$\mu_n([-M,M]^c)\le1/M^2$, so the sequence is tight; pass to a weakly convergent subsequence
$\mu_n\Rightarrow\mu$ (\cref{fact:prokhorov}). The families $\{x\}$ and $\{x^2\}$ are uniformly integrable under
$\mu_n$ (dominated via $\E X^4\le\kappa$), so by \cref{fact:ui} $\E_\mu X=\lim\E_{\mu_n}X=0$ and
$\E_\mu X^2=\lim\E_{\mu_n}X^2=1$; and $\E_\mu X^4\le\liminf\E_{\mu_n}X^4\le\kappa$ by
\cref{fact:portmanteau} applied to the truncations $x^4\wedge M$, then $M\to\infty$ by
monotone convergence. Hence $\mu\in\Ckap$. Since
$[t,\infty)$ is closed, $\mu([t,\infty))\ge\limsup_n\mu_n([t,\infty))=\Vone$, and $\le$ holds
by definition. 
\end{proof}

\begin{lemma}[Rigidity of the Cantelli bound]\label{lem:rigidity}
Let $t>0$ and let $\mu$ have mean $0$ and variance $1$. If $\mu([t,\infty))=1/(1+t^2)$, then
$\mu=\mu_C=\tfrac1{1+t^2}\delta_t+\tfrac{t^2}{1+t^2}\delta_{-1/t}$; in particular
$\E_\mu X^4=\kc(t)$.
\end{lemma}

\begin{proof}
With $q_C=(x+1/t)^2/(t+1/t)^2$ we have $\mathbf1_{[t,\infty)}\le q_C$ pointwise and
$\E_\mu q_C=1/(1+t^2)$ (this uses only the first two moments). Equality
$\E_\mu\mathbf1_{[t,\infty)}=\E_\mu q_C$ forces $\mu\{q_C=\mathbf1_{[t,\infty)}\}=1$. On
$[t,\infty)$, $q_C(x)=1$ only at $x=t$; on $(-\infty,t)$, $q_C(x)=0$ only at $x=-1/t$. So
$\operatorname{supp}\mu\subseteq\{t,-1/t\}$, and mass one plus mean zero pin the weights to
those of $\mu_C$.
\end{proof}

\subsection{Proof of the map}
\label{app:proofmap}

\RestateTheorem{thm:main-map}{The exact tail map of the bounded-kurtosis class}{\ThmMainMapBody}

\begin{proof}
Each regime is proved by a matched pair (\cref{lem:weak-duality} for the upper bound, an
explicit law in $\Ckap$ for the lower). We verify, for each regime: (a) the certificate's
membership and multiplier signs; (b) the two pointwise inequalities
\eqref{eq:cert-glob}--\eqref{eq:cert-event} with explicit SOS data; (c) the value identity
$\E q=\lambda_0+\lambda_2+\lambda_4\kappa=$ claimed bound; (d) witness feasibility and
tightness.

\medskip\noindent\textbf{Regime I} ($b\le t\le c$, i.e.\ $\kappa\ge\kc(t)$ by
\cref{lem:tongue-equiv}). Take Cantelli's certificate
$q(x)=\dfrac{(x+1/t)^2}{(t+1/t)^2}$, i.e.\ $\lambda_4=0$ (a legal multiplier for the
$\le$-constraint (A3)).
(a) $q\in\operatorname{span}\{1,x,x^2\}$. (b) $q\ge0$ is a square, and
\[
q-1=\frac{(x-t)\bigl(x+t+2/t\bigr)}{(t+1/t)^2}
=\frac{(x-t)^2+(2t+2/t)(x-t)}{(t+1/t)^2}\;\ge\;0
\quad\text{on }[t,\infty),
\]
a Markov--Luk\'acs form
with nonnegative coefficients [\texttt{I: q-1 factorization}, \texttt{I: ML pieces
identity}]. (c) $\E q=(1/t^2+1)/(t+1/t)^2=1/(1+t^2)$ under (A1)--(A2) [\texttt{I: value
identity}]. (d) $\mu^\star=\tfrac{1}{1+t^2}\delta_t+\tfrac{t^2}{1+t^2}\delta_{-1/t}$ has mean
$0$, variance $1$ [\texttt{I: witness mean/var}] and fourth moment exactly $\kc(t)\le\kappa$
[\texttt{I: witness kurtosis}], so $\mu^\star\in\Ckap$; its tail mass at $[t,\infty)$ is
$\tfrac1{1+t^2}$ (the atom at $t$; $-1/t<t$). Both bounds meet at $1/(1+t^2)$.

\medskip\noindent\textbf{Regime II} ($t\ge c(\kappa)$, $(t,\kappa)\neq(1,1)$; then $t\ge1$,
and $t>1$ unless $\kappa=1$). Let $p$ be as in \eqref{eq:regimeII} and
$u_\ast=\tfrac{1-pt^2}{1-p}$ (written $u$ in the statement; we add the star to avoid clashing
with $u=\sqrt{\kappa-1}$). Sign facts, each an explicit identity:
\begin{itemize}
\item $1-p=\dfrac{(t^2-1)^2}{(t^2-1)^2+\kappa-1}\in(0,1)$ for $t>1$;
\item $t^2-u_\ast=\dfrac{t^2-1}{1-p}>0$ [\texttt{II: t\^{}2-u == ...}];
\item $1-pt^2$ has numerator $(t^2-1)^2-(\kappa-1)(t^2-1)=(t^2-1)(t^2-\kappa)$ over the
positive denominator $(t^2-1)^2+\kappa-1$; and $t\ge c(\kappa)$ gives $t^2\ge c^2\ge\kappa$
(as $c^2-\kappa=\tfrac{1-\kappa+\sqrt{(\kappa+1)^2-4}}{2}\ge0$ because
$(\kappa+1)^2-4-(\kappa-1)^2=4(\kappa-1)\ge0$ [\texttt{II: c(kappa)\^{}2...}]); hence
$u_\ast\ge0$ and $a=\sqrt{u_\ast}$ is real with $a<t$;
\item $w_+\ge0$: $(1-p)a\ge pt\iff(1-p)^2u_\ast\ge p^2t^2\iff1-p(1+t^2)\ge0\iff
p\le\tfrac1{1+t^2}\iff\kappa\le\kc(t)$ [\texttt{II: (1-p)\^{}2 u - p\^{}2 t\^{}2 ==
1 - p(1+t\^{}2)}, \texttt{II: 1/(1+t\^{}2) - p has numerator t\^{}2(kappa\_c - kappa)}],
which holds on $t\ge c(\kappa)$ by \cref{lem:tongue-equiv}.
\end{itemize}
Certificate: $q(x)=\dfrac{(x^2-u_\ast)^2}{(t^2-u_\ast)^2}$, so
$\lambda_4=(t^2-u_\ast)^{-2}>0$, $\lambda_1=\lambda_3=0$ (no skewness multiplier, as
required). (b) $q\ge0$ is a square. For \eqref{eq:cert-event}:
\[
q(x)-1=\frac{(x^2-t^2)\,(x^2+t^2-2u_\ast)}{(t^2-u_\ast)^2}
\qquad\text{[\texttt{II: q-1 factorization}]},
\]
and with $y=x-t\ge0$,
$(q-1)(t^2-u_\ast)^2=y\,r(y)$ where
\[
r(y)\;=\;y^3+4t\,y^2+\frac{2(\kappa+3t^4-4t^2)}{t^2-1}\,y+\frac{4t(\kappa+t^4-2t^2)}{t^2-1};
\]
all coefficients are nonnegative on the regime because
$\kappa+3t^4-4t^2=(3t^2-1)(t^2-1)+(\kappa-1)\ge0$ and
$\kappa+t^4-2t^2=(t^2-1)^2+(\kappa-1)\ge0$ [\texttt{II: kappa+3t\^{}4...},
\texttt{II: kappa+t\^{}4...}]. Writing $r(y)=c_0+c_1y+c_2y^2+y^3$ (so $c_i\ge0$), the split
is explicit:
\[
(q-1)(t^2-u_\ast)^2\;=\;y\,r(y)\;=\;\underbrace{c_1y^2+y^4}_{\sigma_0\ :=\ (c_1)\,y^2+(1)\,y^4}
\;+\;y\;\underbrace{\bigl(c_0+c_2y^2\bigr)}_{\sigma_1},
\]
and both $\sigma_0=c_1(x-t)^2+(x-t)^4$ and $\sigma_1=c_0+c_2(x-t)^2$ are sums of squares with
manifestly nonnegative (diagonal-Gram) coefficients, giving the Markov--Luk\'acs form
$q-1=\bigl[\sigma_0+(x-t)\sigma_1\bigr]/(t^2-u_\ast)^2$ on $[t,\infty)$. (c) Under
(A1)--(A3) with $\lambda_4>0$,
\[
\E q\;\le\;\frac{u_\ast^2-2u_\ast+\kappa}{(t^2-u_\ast)^2}\;=\;p
\qquad\text{[\texttt{II: value identity}]}.
\]
(d) The witness
$\mu^\star=p\,\delta_t+w_+\delta_a+w_-\delta_{-a}$: mass $p+w_++w_-=1$; mean
$pt+a(w_+-w_-)=pt-pt=0$; variance $pt^2+(w_++w_-)u_\ast=pt^2+(1-p)u_\ast=1$ by definition of
$u_\ast$; fourth moment $pt^4+(1-p)u_\ast^2=\kappa$ [\texttt{II: witness kurtosis}]. This
last identity is the algebra that produced \eqref{eq:regimeII} in the first place:
eliminating $u_\ast$ from the variance and fourth-moment equations yields
$p\,\bigl((t^2-1)^2+\kappa-1\bigr)=\kappa-1$. Weights are nonnegative by the sign facts;
tail mass is $p$ since $a<t$. At $\kappa=1$: $p=0$, $u_\ast=1$, and the witness degenerates
to the symmetric law on $\{\pm1\}$ with $\P(X\ge t)=0$ for $t>1$; consistent.

\medskip\noindent\textbf{Regime IIIa} ($1<\kappa\le\tfrac32$, $\max(\tauofk,0)\le t\le b$,
$t>0$). Let $k_2=2b(c-b)=u(s-u)\ge0$ and $\lambda_4=\bigl((c-b)(b+c)^3\bigr)^{-1}
=(us^3)^{-1}>0$, and
\[
q(x)\;=\;\lambda_4\Bigl[(x^2-c^2)^2+k_2\,(x+c)^2\Bigr].
\]
(a) Expanding, $q=\lambda_4\bigl[x^4+(k_2-2c^2)x^2+2k_2c\,x+(c^4+k_2c^2)\bigr]$: no $x^3$
term (the skewness is unconstrained), and the $x$-coefficient is the multiplier of the
\emph{equality} constraint (A1), whose sign is free. (b) $q\ge0$ on $\R$ is manifest (two
squares, nonnegative coefficients). The key factorization
[\texttt{IIIa: q-1 factorization}]:
\[
q(x)-1=\lambda_4\,(x-b)^2\bigl(x^2+2bx+r\bigr),\qquad r=\tfrac{s^2-6su-3u^2}{4},
\]
verified by expansion using $q(b)=\lambda_4(c-b)(b+c)^3\cdot 1=1$ (which itself is the
identity $(b^2-c^2)^2+k_2(b+c)^2=(c-b)(b+c)^3$). The quadratic factor equals
$(x+b)^2-(b^2-r)$ with $b^2-r=u(s+u)\ge0$ [\texttt{IIIa: b\^{}2 - r1r2 == ...}], so its
larger root is exactly $\tau(\kappa)=-b+\sqrt{u(s+u)}$; hence $x^2+2bx+r\ge0$ for
$x\ge\tauofk$, and $q-1\ge0$ on $[t,\infty)$ whenever $t\ge\tauofk$ (including the closed
edge $t=\tauofk$). An explicit Markov--Luk\'acs form on $[t,\infty)$ follows by writing
$x^2+2bx+r=(x-t)^2+(2t+2b)(x-t)+(t^2+2bt+r)$ with $t^2+2bt+r\ge0$ on the regime:
\[
q-1=\underbrace{\lambda_4(x-b)^2(x-t)^2+\lambda_4(t^2+2bt+r)(x-b)^2}_{\sigma_0}
+(x-t)\,\underbrace{\lambda_4(2t+2b)(x-b)^2}_{\sigma_1}.
\]
(c) $\E q\le\lambda_4\bigl(c^4+k_2c^2+k_2-2c^2+\kappa\bigr)=c/s$ under (A1)--(A3)
[\texttt{IIIa: value identity}]. (d) Witness
$\mu^\star=\tfrac{c}{s}\delta_b+\tfrac{b}{s}\delta_{-c}$: mass $(b+c)/s=1$; mean
$(bc-cb)/s=0$; variance $bc(b+c)/s=s/s=1$; fourth moment
$bc(b^3+c^3)/s=\bigl((b+c)^3-3bc(b+c)\bigr)/s=(s^3-3s)/s=s^2-3=\kappa$
[\texttt{IIIa: witness ...}]. All weights positive; the single atom in $[t,\infty)$ is $b$
(as $t\le b$), of mass $c/s$. Both bounds meet at $c/s=\tfrac12(1+u/s)$.
Plateau window:
\[
\tau\le b\iff u(s+u)\le(2b)^2=(s-u)^2\iff 3us\le s^2\iff 9(\kappa-1)\le\kappa+3
\iff\kappa\le\tfrac32,
\]
with equality giving $\tau(\tfrac32)=b(\tfrac32)=1/\sqrt2$; and
$\tau\le0\iff u(s+u)\le b^2$, which simplifies to $2u^4+10u^2-1\le0$, i.e.\
$\kappa\le\kss=(3\sqrt3-3)/2$ [\texttt{IIIa: tau<=b iff ...}].

\medskip\noindent\textbf{Regime IIIb.} \cref{prop:IIIb} below proves that any solution of
\eqref{eq:systemS} with the stated inequalities produces a matched pair, hence the exact
value, at its parameter point. Solvability at every grid point evaluated is a machine fact
(\cref{sec:pipeline}); no global-solvability claim is made.

\medskip\noindent\textbf{Coverage and gluing.} Fix $\kappa>1$. If $\kappa\le\tfrac32$:
$\tauofk\le b$ and, setting $t_{\mathrm{III}}=\max(\tauofk,0)$,
$(0,\infty)=(0,t_{\mathrm{III}})\cup[t_{\mathrm{III}},b]\cup[b,c]\cup[c,\infty)$, covered by
IIIb/IIIa/I/II respectively (the first piece is empty iff $\kappa\le\kss$). If
$\kappa>\tfrac32$: the plateau is empty and $(0,\infty)=(0,b)\cup[b,c]\cup[c,\infty)$,
covered by IIIb/I/II. Values glue: at $t=c$, \eqref{eq:regimeII} equals $1/(1+c^2)$
[\texttt{II==I on t=c(kappa)}]; at $t=b$, $c/s=1/(1+b^2)$ [\texttt{IIIa==I on t=b(kappa)}]
(both are the statement $b^2+1=s\,b\cdot$(algebra), verified symbolically); the IIIb value
tends to the IIIa value as $t\uparrow\tauofk$ and to the tongue value as $t\uparrow b$
(machine-verified on grids; the LP sweep finds no jump). At $\kappa=1$ the class is the
symmetric law on $\{\pm1\}$ and all formulas degenerate consistently ($V_1=\tfrac12$ for
$t\le1$, $=0$ for $t>1$).
\end{proof}

\begin{proposition}[IIIb matched pairs]\label{prop:IIIb}
Let $(t,\kappa)$ lie in the central wedge and suppose
$(w_0,w_1,w_2,c,b,k_2,\lambda_4)$ solves \eqref{eq:systemS} with $w_i\ge0$, $0<t<b<c$,
$k_2\ge0$, $\lambda_4>0$, and $q-1\ge0$ on $[t,\infty)$. Then
$\Vone(t,\kappa)=w_1+w_2=\lambda_0+\lambda_2+\lambda_4\kappa$, where
$\lambda_0=\lambda_4(c^4+k_2c^2)$, $\lambda_2=\lambda_4(k_2-2c^2)$.
\end{proposition}

\begin{proof}
Upper: $q\in\operatorname{span}\{1,x,x^2,x^4\}$ with $\lambda_4>0$; $q\ge0$ on $\R$ (sum of
two squares); $q\ge1$ on $[t,\infty)$ by hypothesis; so $\Vone\le\E q\le
\lambda_0+\lambda_2+\lambda_4\kappa$ by \cref{lem:weak-duality}. Lower: the four moment
equations of \eqref{eq:systemS} say
$\mu=w_0\delta_{-c}+w_1\delta_t+w_2\delta_b\in\Ckap$ (fourth moment exactly $\kappa$), and
its tail mass is $w_1+w_2$. Matching: $q(-c)=0$ by construction and $q(t)=q(b)=1$ by
the touching conditions of \eqref{eq:systemS}, so
\[
w_1+w_2=\E_\mu q
=\lambda_0\cdot1+\lambda_1\cdot0+\lambda_2\cdot1+\lambda_4\,\kappa ,
\]
where the right side evaluates $\E_\mu$ of each monomial through the four moment
equations. Hence the upper
and lower bounds coincide.
\end{proof}

\subsection{Two-sided map and collapse}

\RestateTheorem{thm:twosided}{Two-sided map}{\ThmTwoSidedBody}

\begin{proof}
$t\le1$: the symmetric law on $\{\pm1\}$ lies in $\Ckap$ ($\kappa\ge1$) and has
$\P(|X|\ge t)=1$.

Chebyshev regime ($t\ge1$, $\kappa\ge t^2$). Certificate $q(x)=x^2/t^2$: degree $2$,
$\lambda_4=0$, $q\ge0$, and on both event pieces $q-1=(x\mp t)(x\pm t)/t^2$ has the
Markov--Luk\'acs forms $\bigl((x-t)^2+2t(x-t)\bigr)/t^2$ on $[t,\infty)$ and mirrored on
$(-\infty,-t]$. $\E q=1/t^2$. Witness $\{0,\pm t\}$ with masses
$(1-\tfrac1{t^2},\tfrac1{2t^2},\tfrac1{2t^2})$: mean $0$, variance $1$, fourth moment
$t^2\le\kappa$ [\texttt{T2-cheb: ...}]; two-sided tail mass $1/t^2$.

Binding regime ($t\ge1$, $1\le\kappa\le t^2$). The Regime II certificate
$q=(x^2-u_\ast)^2/(t^2-u_\ast)^2$ is \emph{even} [\texttt{T2-bind: q2 even}], so
$q\ge1$ on $(-\infty,-t]$ follows from $q\ge1$ on $[t,\infty)$; the proof of the latter and
of $\E q\le p$ is verbatim from Regime II, whose sign facts here need only
$u_\ast\ge0\iff\kappa\le t^2$ (the numerator identity
$(t^2-1)^2-(\kappa-1)(t^2-1)=(t^2-1)(t^2-\kappa)$
[\texttt{T2-bind: (1-p t\^{}2) numerator ...}]), a condition weaker than Regime II's
$\kappa\le\kc(t)$. Witness: the symmetric four-point law
$\tfrac p2(\delta_t+\delta_{-t})+\tfrac{1-p}2(\delta_a+\delta_{-a})$: odd moments vanish;
variance $pt^2+(1-p)u_\ast=1$ and fourth moment $pt^4+(1-p)u_\ast^2=\kappa$ are the
\emph{same} two equations as in Regime II [\texttt{T2-bind: symmetric witness ...}];
two-sided tail mass $p$ ($a<t$ iff $u_\ast<t^2$, given). The regimes glue at $\kappa=t^2$
(both give $1/t^2$).
\end{proof}

\RestateCorollary{cor:collapse}{One-sided/two-sided collapse}{\CorCollapseBody}

\begin{proof}
($\Leftarrow$) For $t\ge c(\kappa)$: $\Vone=p$ by \cref{thm:main-map}(II), and
$\kappa\le\kc(t)\le t^2$ (by \cref{lem:tongue-equiv} and $\kc(t)\le t^2$ for $t\ge1$), so
\cref{thm:twosided} gives $\Vtwo=p$ as well.
($\Rightarrow$) Let $1\le t<c(\kappa)$, i.e.\ $\kappa>\kc(t)$; then $\Vone=1/(1+t^2)$. If
$\kappa\ge t^2$: $\Vtwo=1/t^2>1/(1+t^2)$. If $\kc(t)<\kappa<t^2$:
$\Vtwo=\tfrac{\kappa-1}{(t^2-1)^2+\kappa-1}$, and
$x\mapsto\tfrac{x-1}{(t^2-1)^2+x-1}$ is strictly increasing for $t\neq1$ (derivative
$(t^2-1)^2/(\cdot)^2>0$) with value $1/(1+t^2)$ at $x=\kc(t)$; so $\Vtwo>1/(1+t^2)=\Vone$.
At $t=1$: $\Vone=\tfrac12<1=\Vtwo$. For $0<t<1$: $\Vtwo=1$ while
$\Vone\le1-(2\sqrt3-3)/\kappa<1$: the \emph{certificate} constructed in the proof of
\cref{prop:endpoint} is valid for every $\kappa\ge1$ (only the tightness witness needs
$\kappa\ge\kss$) and satisfies $q\ge\mathbf{1}_{[0,\infty)}\ge\mathbf{1}_{[t,\infty)}$
pointwise for every $t\ge0$, so its bound applies for all $\kappa\ge1$. Equality thus holds exactly on $t\ge c(\kappa)$ (boundary included; there
$\kappa=\kc(t)$ and the three expressions coincide).
\end{proof}

\subsection{Proof degree}

\RestateTheorem{thm:proofdegree}{Proof-degree map}{\ThmProofDegreeBody}

\begin{proof}
\emph{Degree 2 suffices on the tongue}: the Regime I certificate has degree $2$ and proves
the tight value (\cref{thm:main-map}(I)).

\emph{Degree 2 fails off the tongue}: let $q=l_0+l_1x+l_2x^2$ be any polynomial with
$q\ge\mathbf{1}_{[t,\infty)}$ pointwise. Evaluate against Cantelli's law
$\mu_C=\tfrac1{1+t^2}\delta_t+\tfrac{t^2}{1+t^2}\delta_{-1/t}$, a probability law with
mean $0$, variance $1$ (its kurtosis is irrelevant: a degree-$2$ certificate's value
$l_0+l_2$ uses only the first two moments):
\[
l_0+l_2=\E_{\mu_C}q\;\ge\;\P_{\mu_C}(X\ge t)=\frac{1}{1+t^2}.
\]
So no degree-$2$ certificate proves a bound below $1/(1+t^2)$. Off the tongue
($\kappa<\kc(t)$) the true value is strictly smaller: in Regime II,
$p<\tfrac1{1+t^2}$ strictly, by the identity
\[
\frac1{1+t^2}-p\;=\;\frac{t^2\,\bigl(\kc(t)-\kappa\bigr)}{(1+t^2)\bigl((t^2-1)^2+\kappa-1\bigr)}\;>\;0;
\]
in Regimes IIIa/IIIb the strict
inequality follows from attainment and rigidity: by \cref{lem:attain} some
$\mu^\star\in\Ckap$ attains $\Vone(t,\kappa)$; if $\Vone(t,\kappa)$ were equal to
$1/(1+t^2)$, then \cref{lem:rigidity} would force $\mu^\star=\mu_C$, whose fourth moment
$\kc(t)$ exceeds $\kappa$ off the tongue, contradicting $\mu^\star\in\Ckap$. Hence
$\Vone<1/(1+t^2)$ strictly everywhere off the tongue, and $\PDeg>2$.

\emph{No intermediate degree is being skipped}: $x^3$ is not a constrained moment of
$\Ckap$, so no multiplier can generate a cubic term, and a nonconstant polynomial
nonnegative on $\R$ has even degree; the two kinds of \cref{def:cert} are the only
possibilities. Degree $2$ failing off the tongue, $\PDeg=4$ there, realized by the
degree-$4$ certificates of \cref{thm:main-map}.
\end{proof}

\subsection{The endpoint}
\label{app:endpointpf}

\RestateProposition{prop:endpoint}{The $t=0$ endpoint}{\PropEndpointBody}

\begin{proof}
Let $b_0=\sqrt{2\kappa/3}$, $c_0=b_0\tfrac{1+\sqrt3}{2}$, $k_2=2b_0(c_0-b_0)\ge0$,
$\lambda_4=\bigl((c_0-b_0)(b_0+c_0)^3\bigr)^{-1}>0$, and
$q=\lambda_4[(x^2-c_0^2)^2+k_2(x+c_0)^2]$. Then $q(0)=1$ and the polynomial identity
\[
q(x)-1=\lambda_4\;x\,(x+2b_0)\,(x-b_0)^2
\]
holds [\texttt{0: q - 1 == lam4 * x * (x + 2b) * (x-b)\^{}2}]; each factor is nonnegative on
$[0,\infty)$, so $q\ge\mathbf{1}_{[0,\infty)}$ pointwise ($q\ge0$ on $\R$ as a sum of two
squares). The value: $\E q\le\lambda_4(c_0^4+k_2c_0^2+k_2-2c_0^2+\kappa)=1-(2\sqrt3-3)/\kappa$
[\texttt{0: E[q] ...}]. Witness: $\mu=1/(b_0+c_0)$, $w_+=\mu/b_0$, $w_-=\mu/c_0$,
$w_0=1-w_+-w_-$; mean $-c_0w_-+b_0w_+=0$, variance $\mu(b_0+c_0)=1$, fourth moment
$\mu(b_0^3+c_0^3)=\kappa$ [\texttt{0: witness ...}]; and
\[
w_0=1-\frac{1}{b_0c_0}\ \ge\ 0
\iff b_0c_0=\frac{\kappa(1+\sqrt3)}{3}\ge1
\iff \kappa\ge\frac{3}{1+\sqrt3}=\kss,
\]
precisely the hypothesis. Tail mass $w_0+w_+=1-w_-=1-(2\sqrt3-3)/\kappa$
[\texttt{0: value ...}]. For $\kappa\le\kss$, \cref{thm:main-map}(IIIa) applies at every
$t\in(0,b(\kappa)]$ (as $\tauofk\le0$), giving the plateau value independent of $t$; the two
expressions agree at $\kappa=\kss$ (both equal $1/\sqrt3$, a computation).
\end{proof}

\subsection{The minimal polynomial at \texorpdfstring{$(t,\kappa)=(\tfrac12,2)$}{(1/2,2)}}
\label{app:minpoly}

Eliminating $(w_0,w_1,w_2)$ linearly and then $(c,b,k_2,\lambda_4)$ by resultants (or
directly by a lexicographic Gr\"obner basis) from \eqref{eq:systemS} at $t=\tfrac12$,
$\kappa=2$ leaves a single univariate polynomial in $V=w_1+w_2$ whose factor vanishing at
the numerical solution ($V=0.7262558785\ldots$, solved to $50$ digits with primal $=$ dual
to $30$ digits) is the degree-six irreducible polynomial displayed in \cref{sec:endpoint}.
The computation was performed twice with independent code (step-wise filtered resultants,
and a from-scratch Gr\"obner elimination during adversarial review); the two agree
character for character, and the same procedure at $(t,\kappa)=(0.4,2.5)$ again yields an
irreducible sextic: the algebraic degree is not an artifact of the sample point.

\subsection{The class with the fourth moment pinned}
\label{app:eqclass}

A natural question is whether constraining $\E X^4=\kappa$ (equality) instead of
$\E X^4\le\kappa$ changes the map. It does not:

\begin{proposition}[Equality class]\label{prop:eqclass}
Let $\Ckap^{=}=\{X:\E X=0,\ \E X^2=1,\ \E X^4=\kappa\}$ and
$\Vone^{=}(t,\kappa)=\sup_{X\in\Ckap^{=}}\P(X\ge t)$. Then
$\Vone^{=}(t,\kappa)=\Vone(t,\kappa)$ for every $t>0$, $\kappa\ge1$. The supremum over
$\Ckap^{=}$ is attained wherever the optimal law of \cref{thm:main-map} has fourth moment
exactly $\kappa$ (Regimes II, IIIa, IIIb, and the tongue boundary); in the interior of the
tongue it is approached but not attained.
\end{proposition}

\begin{proof}
$\Ckap^{=}\subseteq\Ckap$ gives $\Vone^{=}\le\Vone$. For the converse, let
$\mu\in\Ckap$ attain $\Vone$ (\cref{lem:attain}) with $\kappa'=\E_\mu X^4\le\kappa$; if
$\kappa'=\kappa$ we are done, so assume $\kappa'<\kappa$ (hence $\kappa>1$). For
$\beta\in(0,1]$ let
$\rho_\beta=\tfrac1{1+\beta^2}\delta_\beta+\tfrac{\beta^2}{1+\beta^2}\delta_{-1/\beta}$,
the two-point law with mean $0$, variance $1$, and fourth moment
$\kc(\beta)=\beta^2-1+\beta^{-2}$, which ranges over $[1,\infty)$ as $\beta$ ranges over
$(0,1]$ [\texttt{C4: two-point EX4 == kappa\_c}]. For $\varepsilon\in(0,1)$ set
\[
\mu_\varepsilon=(1-\varepsilon)\,\mu+\varepsilon\,\rho_{\beta(\varepsilon)},
\qquad\text{with }\kc\bigl(\beta(\varepsilon)\bigr)
=\frac{\kappa-(1-\varepsilon)\kappa'}{\varepsilon}\;\bigl(\ge\kappa\ge1\bigr),
\]
which exists by the intermediate value theorem and, for $\varepsilon$ small enough, has
$\beta(\varepsilon)<\min(t,1)$, so $\rho_{\beta(\varepsilon)}$ places no mass in
$[t,\infty)$. By construction $\mu_\varepsilon$ has mean $0$, variance $1$, fourth moment
exactly $\kappa$, and tail mass $(1-\varepsilon)\,\Vone(t,\kappa)$. Letting
$\varepsilon\downarrow0$: $\Vone^{=}\ge\Vone$. Attainment: the witnesses of Regimes II,
IIIa, IIIb have $\E X^4=\kappa$ exactly and lie in $\Ckap^{=}$; in the tongue's interior
the optimum of $\Ckap$ is the Cantelli pair with $\E X^4=\kc(t)<\kappa$, and by the
rigidity lemma (\cref{lem:rigidity}) no law with fourth moment $\kappa\neq\kc(t)$ attains
$1/(1+t^2)$, so the equality-class supremum is approached only.
\end{proof}

This is why the paper works with the inequality class throughout: the two classes share
every value, the inequality class is the form in which the assumption arrives in
applications, and its suprema are always attained.

\subsection{Applications, proved}
\label{app:consproofs}

\RestateCorollary{cor:quantile}{Worst-case quantiles}{\CorQuantileBody}

\begin{proof}
Standardizing ($X\mapsto(X-\mu)/\sigma$ maps the hypothesis class onto $\Ckap$ and the
event onto $\{X\ge t\}$), claim (i) reads
\[
\Vone(t,\kappa)\;\le\;p
\qquad\Longleftrightarrow\qquad
t\;\ge\;t_p(\kappa),
\]
since ``$\P(X\ge\mu+t\sigma)\le p$ for every $X$ in the class'' says exactly
$\Vone(t,\kappa)\le p$, by the definition of $\Vone$ as a supremum. Steps 1--4 prove it;
Step 5 proves claim (ii).

\emph{Step 1: $\Vone(\cdot,\kappa)$ is nonincreasing on $(0,\infty)$ and strictly
decreasing on $[b(\kappa),\infty)$.} Nonincreasing is definitional: for $t\le t'$ the
event $\{X\ge t'\}$ is contained in $\{X\ge t\}$, so $\P(X\ge t')\le\P(X\ge t)$ for
every $X$, and the suprema inherit the order. Strictness on $[b,c]$: there
$\Vone=1/(1+t^2)$ (\cref{thm:main-map}(I)), strictly decreasing. Strictness on
$[c,\infty)$: there $\Vone=(\kappa-1)/((t^2-1)^2+\kappa-1)$ (\cref{thm:main-map}(II)),
and $\kappa>1$ gives $c(\kappa)>1$, so
\[
\frac{d}{dt}\,(t^2-1)^2\;=\;4t\,(t^2-1)\;>\;0\qquad\text{for }t\ge c(\kappa)>1,
\]
so the denominator strictly increases and the value strictly decreases. The two formulas
agree at $t=c$ (the gluing identities after \cref{thm:main-map}), so
$\Vone(\cdot,\kappa)$ is strictly decreasing on all of $[b,\infty)$.

\emph{Step 2: solving $\Vone(t,\kappa)=p$ on the tongue.} For $t\in[b,c]$,
\[
\frac{1}{1+t^2}=p
\;\iff\;
t^2=\frac{1-p}{p}
\;\iff\;
t=\sqrt{\frac{1-p}{p}},
\]
and this $t$ lies in $[b,c]$ exactly when $1/(1+c^2)\le p\le 1/(1+b^2)$, because
$t\mapsto1/(1+t^2)$ maps $[b,c]$ onto $[1/(1+c^2),\,1/(1+b^2)]$ decreasingly. This is
the Cantelli branch.

\emph{Step 3: solving on the tail.} For $t\ge c$ we have $t^2-1\ge c^2-1>0$, so
\[
\frac{\kappa-1}{(t^2-1)^2+\kappa-1}=p
\;\iff\;
(t^2-1)^2=(\kappa-1)\,\frac{1-p}{p}
\;\iff\;
t^2=1+\sqrt{(\kappa-1)\,\frac{1-p}{p}},
\]
taking the positive square root because $t^2-1>0$. The resulting threshold $t_p$
satisfies $t_p\ge c$ exactly when $p\le\Vone(c,\kappa)=1/(1+c^2)$, by Step 1's
monotonicity. This is the kurtosis branch.

\emph{Step 4: the equivalence.} If $t\ge t_p$, then
$\Vone(t,\kappa)\le\Vone(t_p,\kappa)=p$ by Step 1. If $t<t_p$, there are two cases.
Either $t\ge b$: then both $t$ and $t_p$ lie in $[b,\infty)$ ($t_p>b$ on both
branches), so strict decrease gives $\Vone(t,\kappa)>\Vone(t_p,\kappa)=p$. Or $t<b$:
then
\[
\Vone(t,\kappa)\;\ge\;\Vone(b,\kappa)\;=\;\frac{1}{1+b^2}\;>\;p,
\]
the last inequality being the standing hypothesis $p<1/(1+b^2)$. In either case
$\Vone(t,\kappa)>p$, i.e.\ some law in the class exceeds $p$ at threshold $t$, so the
chance constraint fails. The threshold is attained, not merely approached: at $t=t_p$
the regime witness of \cref{thm:main-map} (the Cantelli pair on the tongue, the
three-point law in the tail) has tail probability exactly $p$.

\emph{Step 5: two-sided, claim \textup{(ii)}.} By \cref{thm:twosided}, for $t\ge1$: $\Vtwo=1/t^2$ on
$[1,\sqrt\kappa]$ (the branch $\kappa\ge t^2$), and
$\Vtwo=(\kappa-1)/((t^2-1)^2+\kappa-1)$ on $[\sqrt\kappa,\infty)$ (the branch
$\kappa\le t^2$); both are strictly decreasing (the second by Step 1's computation),
they agree at $t=\sqrt\kappa$ with common value $1/\kappa$, and $\Vtwo=1$ for
$0<t\le1$. Solving $1/t^2=p$ gives $t=1/\sqrt p$, which lies in $[1,\sqrt\kappa]$
exactly when $p\ge1/\kappa$. Solving the second branch repeats Step 3 and yields the
same $t_p(\kappa)$; it lies in $[\sqrt\kappa,\infty)$ exactly when
\[
t_p^2\ge\kappa
\;\iff\;
\sqrt{(\kappa-1)\,\frac{1-p}{p}}\;\ge\;\kappa-1
\;\iff\;
\frac{1-p}{p}\;\ge\;\kappa-1
\;\iff\;
p\;\le\;\frac{1}{\kappa},
\]
dividing by $\kappa-1>0$ and squaring (both sides are nonnegative). At $p=1/\kappa$
the two thresholds coincide at $\sqrt\kappa$. The equivalence then assembles exactly
as in Step 4, using $\Vtwo(t,\kappa)=1>p$ for $t\le1$.
\end{proof}

\RestateLemma{lem:avgkurt}{Kurtosis of an average}{\LemAvgKurtBody}

\begin{proof}
Set $Y_i=(X_i-\mu)/\sigma$, so the $Y_i$ are i.i.d.\ with $\E Y_i=0$, $\E Y_i^2=1$,
$\E Y_i^4=\kappa$, and $\E|Y_i|^3<\infty$ (\cref{fact:momord}). Expanding the fourth
power of the sum and grouping monomials by index pattern,
\[
\Bigl(\sum_{i=1}^m Y_i\Bigr)^{\!4}
=\sum_i Y_i^4
\;+\;4\sum_{i\ne j}Y_i^3Y_j
\;+\;3\sum_{i\ne j}Y_i^2Y_j^2
\;+\;12\!\!\sum_{\substack{j<k\\ i\notin\{j,k\}}}\!\!Y_i^2Y_jY_k
\;+\;24\!\!\!\sum_{i<j<k<l}\!\!\!Y_iY_jY_kY_l,
\]
where the second and third sums run over ordered pairs $i\ne j$ and the coefficients
are the multinomial coefficients of the patterns $(3,1)$, $(2,2)$, $(2,1,1)$,
$(1,1,1,1)$. Take expectations. By independence, the expectation of each product
factors over distinct indices; in the second, fourth, and fifth sums, every term
contains at least one index appearing exactly once, hence a factor $\E Y_j=0$. Only
the first and third sums survive:
\[
\E\Bigl(\sum_{i=1}^m Y_i\Bigr)^{\!4}
\;=\;m\,\E Y^4\;+\;3\,m(m-1)\,(\E Y^2)^2
\;=\;m\kappa+3m(m-1).
\]
The standardized mean is $S_m=m^{-1/2}\sum_iY_i=\sqrt m\,(\bar X_m-\mu)/\sigma$; then
$\E S_m=0$, $\E S_m^2=m^{-1}\sum_i\E Y_i^2=1$ (cross terms carry $\E Y_iY_j=0$), and
\[
\E S_m^4\;=\;\frac{m\kappa+3m(m-1)}{m^2}\;=\;\frac{\kappa}{m}+\frac{3(m-1)}{m}
\;=\;3+\frac{\kappa-3}{m}.
\]
The third moment $\E Y^3$ was never evaluated: it is finite, and it only ever
multiplies $\E Y_j=0$.
\end{proof}

\RestateCorollary{cor:margin}{Exact margin bound}{\CorMarginBody}

\begin{proof}
The affine map $x\mapsto\gamma-\sigma x$ is a bijection from $\Ckap$ onto the class of
laws with mean $\gamma$, variance $\sigma^2$, and fourth central moment at most
$\kappa\sigma^4$: both directions preserve all three constraints, since the fourth
central moment picks up the factor $\sigma^4$ and is blind to the sign flip. Under this
map,
\[
\{M\le0\}\;=\;\{\gamma-\sigma X\le0\}\;=\;\{X\ge\gamma/\sigma\},
\]
a closed event of exactly the map's form. Hence
\[
\sup_M\;\P(M\le0)\;=\;\sup_{X\in\Ckap}\;\P\bigl(X\ge\gamma/\sigma\bigr)
\;=\;\Vone(\gamma/\sigma,\kappa),
\]
the suprema ranging over the two classes exchanged by the bijection. Attainment is
\cref{lem:attain} transported through the same map: if $\mu^\star$ attains
$\Vone(\gamma/\sigma,\kappa)$ and $X^\star\sim\mu^\star$, then
$M^\star=\gamma-\sigma X^\star$ attains the margin bound. In the closed-form regimes
$\mu^\star$ is the explicit regime witness of \cref{thm:main-map}.
\end{proof}

\RestateProposition{prop:certifiable}{Exact directional tails at degree $4$}{\PropCertifiableBody}

\begin{proof}
\emph{The bound, and \textup{(i)}.} Fix $u$ and set $Y=\langle u,X-\mu\rangle/\sigma_u$.
Then $\E Y=0$
by linearity, $\E Y^2=\langle u,\Sigma u\rangle/\sigma_u^2=1$ by the definition of the
covariance, and, evaluating the certifiability hypothesis at the point $u$ (a sum of
squares is pointwise nonnegative),
\[
\E\langle u,X-\mu\rangle^4\;\le\;\kappa\,\langle u,\Sigma u\rangle^2
\qquad\Longrightarrow\qquad
\E Y^4\;=\;\frac{\E\langle u,X-\mu\rangle^4}{\sigma_u^4}\;\le\;\kappa .
\]
Thus the law of $Y$ lies in $\Ckap$, and $\P(Y\ge t)\le\Vone(t,\kappa)$ for every
$t>0$, by the definition of $\Vone$ as a supremum over exactly this class. For the
sum-of-squares form of the derivation, let
$q=\lambda_0+\lambda_1x+\lambda_2x^2+\lambda_4x^4$ be the certificate of
\cref{thm:main-map} at $(t,\kappa)$ (explicit in Regimes I, II, and IIIa, which
together cover every $(t,\kappa)$ outside the central regime); exactly as in
\cref{lem:weak-duality},
\[
\P(Y\ge t)\;\le\;\E\,q(Y)
\;=\;\lambda_0+\lambda_2+\lambda_4\,\E Y^4
\;\le\;\lambda_0+\lambda_2+\lambda_4\,\kappa
\;=\;\Vone(t,\kappa),
\]
where the first inequality integrates the pointwise bound
$\mathbf 1_{[t,\infty)}\le q$, the equality uses $\E Y=0$, $\E Y^2=1$, and the absence
of a cubic term in $q$, and the last inequality uses $\lambda_4\ge0$ together with
$\E Y^4\le\kappa$. Every inequality in this chain carries a degree-$4$ SOS witness:
the two pointwise inequalities for $q$ carry the explicit data of \cref{def:cert}, the
moment step carries the certifiability hypothesis itself, and $\lambda_4\ge0$ is a
sign condition. (At a solved Regime IIIb instance the same chain runs with the
certificate of \cref{prop:IIIb}; the probability bound itself needs no certificate.)

\emph{\textup{(ii)}: tightness for $\kappa\ge3$, $t\ge b(\kappa)$.} Since
$\kappa\ge3>\ks$, the
plateau is empty, and $t\ge b(\kappa)$ places $(t,\kappa)$ in Regime I or Regime II.
Let $\xi$ be that regime's extremal law (\cref{thm:main-map}): the Cantelli pair, with
$\kappa_\xi:=\E\xi^4=\kc(t)\le\kappa$ (the tongue condition), in Regime I; the
three-point law, with $\kappa_\xi=\kappa$ exactly, in Regime II. In both cases
$\E\xi=0$, $\E\xi^2=1$, and $\P(\xi\ge t)=\Vone(t,\kappa)$, and $\xi$ has finite
support, so all its moments are finite. Fix a unit $v\in\R^d$, an orthonormal basis
$e_1,\dots,e_{d-1}$ of $v^\perp$, and let $Z=(Z_1,\dots,Z_{d-1})$ be uniform on
$\{\pm1\}^{d-1}$, independent of $\xi$. Define
\[
X\;=\;\xi\,v\;+\;\sum_{j=1}^{d-1}Z_j\,e_j .
\]
Mean and covariance: $\E X=0$, and
$\E XX^\top=vv^\top\,\E\xi^2+\sum_j e_je_j^\top\,\E Z_j^2=I_d$, the cross terms
vanishing by independence and mean zero.

Directional fourth moments: write $u=a\,v+w$ with $a=\langle v,u\rangle$ and
$w=\sum_jw_je_j\in v^\perp$, and set $W=\sum_jw_jZ_j$, so
$\langle u,X\rangle=a\,\xi+W$ with $\xi$ and $W$ independent. The moments of $W$:
$\E W=\E W^3=0$ (the law of $Z$ is symmetric under $Z\mapsto-Z$), $\E W^2=|w|^2$, and,
by the same solitary-factor argument as in the proof of \cref{lem:avgkurt},
\[
\E W^4\;=\;\sum_jw_j^4\,\E Z_j^4+3\sum_{j\ne k}w_j^2w_k^2\,\E Z_j^2\,\E Z_k^2
\;=\;\sum_jw_j^4+3\Bigl(|w|^4-\sum_jw_j^4\Bigr)
\;=\;3|w|^4-2\sum_jw_j^4 .
\]
Expanding $\E(a\xi+W)^4$ binomially and dropping the two terms with an odd power of
$W$ (each carries a factor $\E W=0$ or $\E W^3=0$, against finite moments of the
atomic law $\xi$),
\[
\E\langle u,X\rangle^4
\;=\;a^4\kappa_\xi\;+\;6a^2|w|^2\;+\;3|w|^4-2\sum_jw_j^4 .
\]
Therefore, using $|u|^4=(a^2+|w|^2)^2=a^4+2a^2|w|^2+|w|^4$,
\[
\kappa\,|u|^4-\E\langle u,X\rangle^4
\;=\;(\kappa-\kappa_\xi)\bigl(a^2\bigr)^2
\;+\;(2\kappa-6)\sum_j\bigl(a\,w_j\bigr)^2
\;+\;(\kappa-3)\bigl(|w|^2\bigr)^2
\;+\;2\sum_j\bigl(w_j^2\bigr)^2 .
\]
Each summand is the square of a quadratic polynomial in $u$
($a^2=\langle v,u\rangle^2$, $a\,w_j=\langle v,u\rangle\langle e_j,u\rangle$,
$|w|^2=\sum_j\langle e_j,u\rangle^2$, $w_j^2=\langle e_j,u\rangle^2$), and each
coefficient is nonnegative: $\kappa-\kappa_\xi\ge0$ in both regimes, and
$2\kappa-6\ge0$, $\kappa-3\ge0$ by the hypothesis $\kappa\ge3$. So the displayed gap
is a sum of squares: $X$ is $\kappa$-certifiable with $\Sigma=I_d$, and equality holds
in the prescribed direction, $\P(\langle v,X\rangle\ge t)=\P(\xi\ge t)=\Vone(t,\kappa)$.

\emph{\textup{(iii)}: two-sided, $t\ge c(\kappa)$.} There $t^2\ge c(\kappa)^2\ge\kappa$
(the Regime II
sign facts in \cref{app:proofs}), so the binding branch of \cref{thm:twosided} applies
and equals \eqref{eq:regimeII}, which is $\Vone(t,\kappa)$ by \cref{cor:collapse}; its
certificate is the same even polynomial, so the substitution argument runs verbatim
for the event $\{|Y|\ge t\}$. The witness above attains this value too: the Regime II
law places mass $\Vone(t,\kappa)$ at the atom $t$ and its two remaining atoms inside
$(-t,t)$, so $\P(|\xi|\ge t)=\P(\xi=t)=\Vone(t,\kappa)$.
\end{proof}

\section{Worked instances: the map in explicit numbers}
\label{app:worked}

This section instantiates \cref{thm:main-map} at one rational parameter point per
closed-form regime, with every object written out. A reader can verify each instance
completely by hand: expand two polynomial identities, check three or four moment equations
in rational arithmetic, and confirm the positive semidefiniteness of matrices no larger than
$3\times3$. Every displayed number below is additionally machine-checked
(\texttt{scripts/verify\_appendix\_numbers.py}, all \textsc{pass}), and the Regime II
instance is the same one whose machine certificate is reproduced verbatim in
\cref{app:walkthrough}.

\subsection{Regime II at $(t,\kappa)=(2,3)$: bound $2/11$}
\label{app:worked-II}

\emph{Value.} $p=\dfrac{\kappa-1}{(t^2-1)^2+\kappa-1}=\dfrac{2}{9+2}=\dfrac{2}{11}$, and
$u_\ast=\dfrac{1-pt^2}{1-p}=\dfrac{1-\frac8{11}}{\frac9{11}}=\dfrac13$, so
$(t^2-u_\ast)^2=\bigl(\tfrac{11}{3}\bigr)^2=\tfrac{121}9$.

\emph{Certificate.} $q(x)=\dfrac{(x^2-\frac13)^2}{(11/3)^2}=\dfrac{(3x^2-1)^2}{121}$, i.e.
\[
\lambda_0=\tfrac1{121},\qquad \lambda_1=0,\qquad \lambda_2=-\tfrac6{121},\qquad
\lambda_4=\tfrac9{121}\;(>0,\ \text{legal for (A3)}),
\]
with no $x^3$ term, as required since $\E X^3$ is unconstrained.

\emph{Global nonnegativity} \eqref{eq:cert-glob}: on the basis $z=(1,x,x^2)$,
\[
q(x)=z^\top G\,z,\qquad
G=\frac{1}{121}\begin{pmatrix}1&0&-3\\0&0&0\\-3&0&9\end{pmatrix},
\]
and $G\succeq0$ by exact $LDL^\top$: pivots $\bigl(\tfrac1{121},\,0,\,0\bigr)$: the
second pivot is $0$ with an all-zero row (legal), and the third is the Schur complement
$\tfrac9{121}-\tfrac{(-3/121)^2}{1/121}=0$. ($G$ is the rank-one matrix
$vv^\top/121$, $v=(-1,0,3)$.)

\emph{Event inequality} \eqref{eq:cert-event}: with $y=x-2$,
\[
(q-1)\cdot\tfrac{121}9 \;=\; y\Bigl(y^3+8y^2+\tfrac{70}3\,y+\tfrac{88}3\Bigr)
\;=\;\underbrace{\Bigl(\tfrac{70}3\,y^2+y^4\Bigr)}_{\sigma_0}
+\;y\,\underbrace{\Bigl(\tfrac{88}3+8y^2\Bigr)}_{\sigma_1},
\]
an identity the reader can expand; all four coefficients are positive, so $\sigma_0,\sigma_1$
are (diagonal-Gram) sums of squares and $q\ge1$ on $[2,\infty)$.

\emph{Bound.} $\E q(X)\le\lambda_0+\lambda_2+\lambda_4\kappa
=\tfrac{1-6+27}{121}=\tfrac{22}{121}=\tfrac2{11}$ for every $X\in\CkapOf3$
(\cref{lem:weak-duality}).

\emph{Witness.} $\mu^\star=\tfrac2{11}\,\delta_2+w_+\delta_{1/\sqrt3}+w_-\delta_{-1/\sqrt3}$
with
\[
w_\pm=\frac{9\mp4\sqrt3}{22},\qquad w_+\approx0.0942,\quad w_-\approx0.7240 .
\]
Check (exact): mass $\tfrac2{11}+\tfrac9{11}=1$; mean
$\tfrac4{11}+\tfrac{1}{\sqrt3}\cdot\bigl(w_+-w_-\bigr)
=\tfrac4{11}-\tfrac{1}{\sqrt3}\cdot\tfrac{8\sqrt3}{22}=0$; variance
$\tfrac8{11}+\tfrac13\cdot\tfrac9{11}=1$; fourth moment
$\tfrac{32}{11}+\tfrac19\cdot\tfrac9{11}=3$. Tail mass: only the atom $2$ lies in
$[2,\infty)$ (as $1/\sqrt3<2$), so $\P(X\ge2)=\tfrac2{11}$. The bound is attained; note the
witness is genuinely skewed, $\E X^3=\tfrac43$, which is why it lives outside the
skewness-pinned classes of the classical literature (\cref{app:priorwork}).

\subsection{Regime I at $(t,\kappa)=(2,6)$: bound $1/5$}
\label{app:worked-I}

Here $\kc(2)=\tfrac{13}4\le6$: the tongue. Certificate
$q(x)=\dfrac{(x+\frac12)^2}{(5/2)^2}=\dfrac{(2x+1)^2}{25}$, degree $2$, $\lambda_4=0$;
support Gram $\tfrac1{25}\bigl(\begin{smallmatrix}1&2\\2&4\end{smallmatrix}\bigr)\succeq0$
(pivots $\tfrac1{25},0$); event identity
$q-1=\tfrac{4(x-2)^2+20\,(x-2)}{25}\ge0$ on $[2,\infty)$; bound
$\E q\le\tfrac{1/4+1}{25/4}=\tfrac15$. Witness: the Cantelli pair
$\tfrac15\delta_2+\tfrac45\delta_{-1/2}$, with fourth moment
$\tfrac{16}5+\tfrac45\cdot\tfrac1{16}=\tfrac{13}4\le6$: feasible precisely because the
budget covers $\kc(2)$, and $\P(X\ge2)=\tfrac15$. Note the certificate never uses the
kurtosis constraint: on the tongue, fourth-moment information is worthless, and this
instance shows the mechanism.

\subsection{Regime IIIa at $(t,\kappa)=(\tfrac12,\tfrac54)$: bound
$\tfrac12+\tfrac{\sqrt{17}}{34}$}
\label{app:worked-IIIa}

Here $u=\tfrac12$, $s=\tfrac{\sqrt{17}}2$, so $b=\tfrac{\sqrt{17}-1}4\approx0.7808$,
$c=\tfrac{\sqrt{17}+1}4\approx1.2808$ ($bc=1$), $k_2=\tfrac{\sqrt{17}-1}4$,
$\lambda_4=\tfrac{16\sqrt{17}}{289}$, and
$\tauofk=\sqrt{\tfrac12\bigl(\tfrac{\sqrt{17}}2+\tfrac12\bigr)}-b\approx0.3509\le
t=\tfrac12\le b$: the plateau window. The certificate
$q=\lambda_4[(x^2-c^2)^2+k_2(x+c)^2]$ satisfies
$q-1=\lambda_4(x-b)^2\bigl(x^2+2bx+r\bigr)$ with $r=\tfrac{7-3\sqrt{17}}8\approx-0.671$;
the quadratic factor's larger root is exactly $\tauofk<t$, so $q\ge1$ on $[t,\infty)$. The
value: $\E q\le c/s=\tfrac{17+\sqrt{17}}{34}=\tfrac12+\tfrac{\sqrt{17}}{34}\approx0.62127$.
Witness: $\mu^\star=\tfrac{17+\sqrt{17}}{34}\,\delta_{b}+\tfrac{17-\sqrt{17}}{34}\,
\delta_{-c}$, with variance $bc(b+c)/s=1$ and fourth moment $s^2-3=\tfrac54$ exactly;
its atom $b\approx0.78$ exceeds $t=\tfrac12$, so the tail mass equals the bound. Because
neither the certificate nor the witness involves $t$ (only the regime condition
$\tauofk\le t\le b$ does), the same pair certifies every threshold in the window: this is
the freezing phenomenon of \cref{thm:main-map}(IIIa) made concrete. This instance's
machine certificate carries these quadratic surds exactly, and the independent verifier's
verdict for it is \textsc{verified-tight} with bound $\tfrac12+\tfrac{\sqrt{17}}{34}$
(\cref{tab:instances}).

\section{Prior work: classes, formulas, and where the map sits}
\label{app:priorwork}

\cref{tab:priorwork} organizes the literature by the one distinction that matters for
comparing results: \emph{which moment functionals define the class}. Two bounds are
comparable as theorems only if their classes coincide; most apparent overlaps between this
paper and the classical literature dissolve on this axis. Formulas are reproduced only when
we verified them against the primary source (or, where noted, a secondary transcription);
we deliberately do not reproduce formulas we could not verify.

\begin{center}\small
\renewcommand{\arraystretch}{1.25}%
\begin{tabular}{>{\raggedright\arraybackslash}p{0.23\dimexpr\linewidth-8\tabcolsep\relax}%
                >{\raggedright\arraybackslash}p{0.25\dimexpr\linewidth-8\tabcolsep\relax}%
                >{\raggedright\arraybackslash}p{0.18\dimexpr\linewidth-8\tabcolsep\relax}%
                >{\raggedright\arraybackslash}p{0.34\dimexpr\linewidth-8\tabcolsep\relax}}
\toprule
work & class (constraints) & event / scope & answer, and sharpness in its class \\
\midrule
\citeauthor{bienayme1853}~(\citeyear{bienayme1853})--\citeauthor{chebyshev1867}~(\citeyear{chebyshev1867}) & $\E X{=}0$, $\E X^2{=}1$ & $\P(|X|\ge t)$ & $\min(1,1/t^2)$; sharp \\
\citet{cantelli1928} & $\E X{=}0$, $\E X^2{=}1$ & $\P(X\ge t)$ & $1/(1+t^2)$; sharp \\
\citet{selberg1940} & two moments & interval events & sharp two-moment refinement for
  asymmetric intervals (formula not reproduced here) \\
\citet{guttman1948} & four moments & kurtosis-type & an early fourth-moment inequality
  (primary text not accessible to us; not reproduced) \\
\citet{royden1953} & $n$ moments pinned & CDF bounds & sharp, via case analysis over
  principal representations \\
\citet{zelen1954} & \emph{all} moments to order 4 pinned ($m_3$ included) & $\P(X\ge x)$,
  multi-branch & sharp in its class where each branch applies; his eq.~(14) reproduced
  below \\
\citet{karlinstudden1966} & general T-systems & structural theory & principal
  representations: extremal measures with few atoms \\
\citet{bhattacharyya1987} & four moments known & $\P(X\ge t\sigma)$ & a closed-form bound
  later shown \emph{not sharp} by \citet{maher2019}; superseded pointwise by
  \cref{thm:main-map} \\
\citet{bertsimaspopescu2005} & any finite moment set & general events & SDP framework;
  explicit closed forms up to 3 moments \\
\citet{hezhangzhang2010} & $\E X, \E X^2, \E X^4$ pinned (no $m_3$) & small-deviation
  regime ($t$ near the mean) & tight bounds near the mean; the $t{=}0$ constant
  $2\sqrt3-3$ $=$ our endpoint (\cref{prop:endpoint}) \\
\citet{maher2019} & four moments & discrete stress tests & constructions showing
  Bhattacharyya's bound is not extremal \\
\citet{garnett2020} & $\E X{=}0$, $\E X^2{=}1$, $\E X^3\ge0$, $\E X^4\le\kappa$ &
  $\P(X\ge 0)$ & $1-1/(2\kappa)$; sharp in its sign-restricted class \\
\midrule
this paper & $\E X{=}0$, $\E X^2{=}1$,\newline $\E X^4\le\kappa$ ($m_3$ \emph{free}) &
  $\P(X\ge t)$ and $\P(|X|\ge t)$, \emph{all} $(t,\kappa)$ & four-regime complete map, sharp
  everywhere, with certificates, extremal laws, and the proof-degree phase diagram \\
\bottomrule
\end{tabular}
\captionof{table}{The moment-problem literature relevant to the map, organized by class. ``Sharp''
always means: extremal for the stated class, which differs across rows.}
\label{tab:priorwork}
\end{center}

\paragraph{Zelen's formula, for concreteness.} For the class with all four moments $m_k\coloneqq\E X^k$ pinned
($m_1=0$, $m_2=1$, $m_3$, $m_4$), Zelen's eq.~(14), transcribed from the primary source
\citep{zelen1954} and re-derived by our pipeline on the $m_3=0$ slice, bounds the upper
tail by
\[
\P(X\ge x)\;\le\;\frac{m_4-m_3^2-1}{(1+x^2)(m_4-m_3^2-1)+(x^2-m_3x-1)^2},
\]
valid where his auxiliary quadratic $g(x)=x^2-m_3x-1$ is positive (for $m_3=0$: $x>1$).
Setting $m_3=0$ and comparing with Regime II of \cref{thm:main-map} at $(t,\kappa)=(2,3)$:
the $\gamma$-pinned class gives $2/19$, the skewness-free class $\Ckap$ gives $2/11$. Both
are sharp, for different classes; the ratio quantifies exactly what knowing
$\E X^3=0$ is worth at that point. Our extremal law for Regime II has
$\E X^3=pt(t^2-u_\ast)\ne0$ (e.g.\ $\tfrac43$ at $(2,3)$; \cref{app:worked-II}), which is
the structural reason the two answers must differ.

\paragraph{Adjacent literatures.} Four neighboring lines contextualize the table.
\emph{(i) Generalized Chebyshev bounds by optimization.} The duality behind
\cref{lem:weak-duality} is classical \citep{isii1962}; \citet{smith1995} developed
generalized Chebyshev inequalities for decision analysis, \citet{popescu2005} optimal
moment bounds for convex classes, and \citet{vandenberghe2007} sharp multivariate
Chebyshev bounds by semidefinite programming. These frameworks \emph{compute} sharp
bounds numerically in far richer settings than ours; what they do not provide, and
what this paper provides for one natural class, is the closed-form regime geometry,
the extremal laws, and exact certificates. \emph{(ii) Distributionally robust
optimization} builds decision problems over moment ambiguity sets
\citep{delageye2010,wiesemann2014}; $\Vone(t,\kappa)$ is exactly the worst-case chance
constraint of the univariate fourth-moment ambiguity set, so the map hands DRO's inner
problem a closed form in that instance. \emph{(iii) Heavy-tailed statistics.}
Finite-kurtosis assumptions power Catoni-style mean estimators \citep{catoni2012} and
run through the median-of-means literature \citep{lugosimendelson2019}; our constants
quantify what that assumption buys for tail control by itself, before any estimator
enters. \emph{(iv) Sum-of-squares proof complexity.} The surveys of \citet{barak2014}
and \citet{fleming2019} treat certificate degree as a computational resource in the
multivariate world, where hierarchies are generally inexact; the proof-degree map of
\cref{thm:proofdegree} is a univariate, exactly solvable pilot of that question for
probability inequalities, and the exact-certification line of \citet{kaltofen2012}
parallels our rounding-and-verification layer. Finally, the classical monographs
\citep{karlinstudden1966,kreinnudelman1977} contain the Tchebycheff-system theory that
explains, structurally, why every extremal law in this paper has at most three atoms.

\paragraph{What is new here, stated carefully.} Against this table, the contributions of
the present paper are: (i) the complete $(t,\kappa)$ map for the \emph{skewness-free}
bounded-kurtosis class, in closed form on three regimes; prior work either pins the skewness or
its sign (Zelen, Royden, Bhattacharyya's setting; \citealp{garnett2020}), or treats the
small-deviation regime (He--Zhang--Zhang) --- the strongest interior overlap with the map;
our verified point of formula-level contact with the latter is the $t{=}0$ constant; any
tight value at small $t>0$ is necessarily the corresponding root of system (S)
(\cref{sec:endpoint}), and no single nested-square-root formula can cover the regime's
interior (\cref{app:minpoly});
(ii) the one-sided/two-sided collapse on $t\ge c(\kappa)$, which we did not find stated
anywhere; (iii) the proof-degree phase diagram; (iv) the machine-verified certificate
layer; (v) the applications of \cref{sec:consequences}: we found no four-moment,
skewness-free closed-form quantile inversion in the worst-case value-at-risk literature
(Zelen's inversions pin the skewness; \citealp{hezhangzhang2010} treat small deviations;
the conic worst-case-VaR line of work stops at two moments), no kurtosis-refined
median-of-means block constant, and no exact directional tail constant under certifiable
kurtosis, where the standard conversion from certificate to tail is the Markov step
$\kappa/t^4$. Each novelty statement is hedged in the text as ``to our knowledge''; this
table is the audit trail for those hedges.

\section{The pipeline, the certificates, and reproduction}
\label{app:pipeline}

\subsection{Certificate format and the independent verifier}

A certificate is a JSON object carrying: the moment constraints (polynomial, relation,
value), the dual multipliers $\lambda$, the polynomial $q$, the claimed bound, and, per
piece (event / support / off-event), a list of SOS blocks --- each a multiplier from the
piece's allowed set ($1$ or $x-t$ on $[t,\infty)$; $1$, $x-a$, $b-x$, $(x-a)(b-x)$ on
$[a,b]$; $1$ on $\R$), a polynomial basis, and a Gram matrix with exact entries --- plus an
optional extremal witness (atoms, weights). All numbers are exact: rationals or symbolic
algebraic numbers (e.g.\ $\sqrt{17}$ on the IIIa grid); floats anywhere cause rejection.

The verifier (\texttt{forge/verify.py}, $\sim$370 lines, sympy only) re-derives everything:
V1 exact parsing; V2 recompute $q=\sum_i\lambda_ig_i$; V3 the piece list matches the
objective's majorant logic; V4 each piece's polynomial identity
$\text{target}=\sum_j\text{mult}_j\cdot z_j^\top G_jz_j$ holds coefficient-by-coefficient;
V5 each $G_j\succeq0$ by exact $LDL^\top$ (zero pivot $\Rightarrow$ zero row) \emph{and} by
the characteristic-polynomial criterion (all coefficients of $\det(G+sI)$ nonnegative), at
least one of which must pass; V6 multiplier signs against constraint directions; V7 the
bound equals $\sum_i\lambda_ic_i$; V8 the witness is feasible (exact moment arithmetic) and
attains the bound. Verdicts: \textsc{verified-bound} or \textsc{verified-tight}. The
checker rejects eight classes of mutated certificates (test suite), and grants
\textsc{certified} status only when run in a fresh process.

\subsection{The two exact positive-semidefiniteness criteria, proved}
\label{app:psd}

The verifier's check V5 decides $G\succeq0$ for a symmetric matrix with entries in $\Q$
(or a real quadratic extension $\Q(\sqrt{d})$) by two independent criteria; for
completeness we prove both. Both proofs are valid over any ordered field, which is exactly
why the checks run in exact arithmetic with no numerical tolerance.

\begin{lemma}[Pivot criterion]\label{lem:ldlt}
Let $G$ be symmetric $n\times n$. Perform Gaussian elimination without pivoting, treating
the diagonal entries in order: if a pivot is negative, stop; if a pivot is zero, require
its entire row and column to be zero and skip it; otherwise eliminate below it. Then
$G\succeq0$ if and only if the elimination completes, i.e.\ every pivot encountered is
$\ge0$ and every zero pivot has an all-zero row.
\end{lemma}

\begin{proof}
Induction on $n$. If $g_{11}<0$ then $e_1^\top Ge_1<0$, so $G$ is not PSD; and if the
elimination completes, no such pivot occurred. If $g_{11}=0$ and some $g_{1j}\neq0$, the
principal $2\times2$ submatrix on rows and columns $\{1,j\}$ is
$\begin{pmatrix}0&\beta\\ \beta&\delta\end{pmatrix}$ with $\beta\neq0$, whose determinant
$-\beta^2<0$ shows $G$ is not PSD; conversely a PSD matrix with $g_{11}=0$ has a zero
first row and column, the elimination skips it, and $G\succeq0$ iff the remaining
principal block is. If $g_{11}>0$, the congruence
$G=L\begin{pmatrix}g_{11}&0\\ 0&S\end{pmatrix}L^\top$, with $S$ the Schur complement and
$L$ unit lower triangular, shows $G\succeq0\iff S\succeq0$, and one elimination step
replaces $G$ by $S$.
\end{proof}

\begin{lemma}[Coefficient criterion]\label{lem:charpoly}
Let $G$ be symmetric $n\times n$. Then $G\succeq0$ if and only if every coefficient of
the polynomial $s\mapsto\det(G+sI)$ is $\ge0$.
\end{lemma}

\begin{proof}
$G$ is symmetric, so $\det(G+sI)=\prod_{i=1}^n(\lambda_i+s)$ with all $\lambda_i$ real
(\cref{item:spec} of \cref{app:register}).
If every $\lambda_i\ge0$, expanding the product gives nonnegative coefficients.
Conversely, suppose all coefficients are $\ge0$; the polynomial is monic, so it is
$\ge s^n>0$ for every $s>0$. If some $\lambda_j<0$, then $s=-\lambda_j>0$ would be a
root, a contradiction.
\end{proof}

The verifier requires at least one of the two to pass (they are computed by different
code paths); in practice both pass, and the pivot sequences are reported in the
transcript (\cref{app:walkthrough}).

\subsection{$\varepsilon$-retreat}

When the numerical optimum has singular Gram matrices (boundary optima are the norm here:
tight certificates have double roots), rational rounding can fail. The retreat re-solves the
SDP with the objective pinned to $\text{value}+\varepsilon$
($\varepsilon\in\{10^{-9},10^{-6}\}$ rational), which restores a strict interior; the
rounded certificate then proves the explicitly weaker bound, and is labeled as such. This
never touches the closed-form regimes (their certificates are constructed symbolically) and
was needed only on the rounding route.

\subsection{Certified instances}
\label{app:certified}

Of the $48$ kurtosis-cell grid instances, $47$ are \textsc{certified} ($45$ of them
\textsc{verified-tight} with exact witnesses; the two IIIb instances are bound-certificates
with matching numerical witnesses); the one failure, $(t,\kappa)=(\tfrac12,\tfrac32)$, sits
exactly on $\kappa=\ks$ and is recorded as a rounding failure. The benchmark cells are
$22/23$ certified ($4$ Markov, $4$ Cantelli, $6$ tight Paley--Zygmund, $8/9$ three-moment);
the skewness-pinned surface adds $47/60$. Full tables: \texttt{results/constants.csv} in the
artifact repository; every \textsc{certified} row names its certificate file, solver,
residuals, and git revision.

\subsection{Reproduction}

\begin{verbatim}
pip install numpy scipy sympy mpmath cvxpy clarabel pyyaml pytest matplotlib
python -m pytest tests/ -q                  # known-answer + mutation tests
python scripts/verify_theorem_symbolic.py   # all 52 symbolic identities
python scripts/produce_results.py           # full grid: solve, round, verify
                                            #   (fresh processes), tables
python scripts/run_b5.py                    # Khintchine POP + recognition
python scripts/verify_appendix_numbers.py   # every number printed in the
                                            #   worked-instances appendix
python -m forge.verify results/certificates/kurtosis_tail__kappa3_t2.cert.json
\end{verbatim}
The last command independently re-verifies the Regime II instance
$\P(X\ge2)\le\tfrac2{11}$ over $\CkapOf3$ (with its tightness witness) in a fresh process.

\paragraph{Index of symbolic-check tags.} Throughout the deferred proofs, every asserted
polynomial identity carries a bracketed tag naming the corresponding check in
\texttt{scripts/verify\_theorem\_symbolic.py}; the script prints one \textsc{pass}/\textsc{fail}
line per tag ($52$ checks, all \textsc{pass} at the artifact snapshot).
\Cref{tab:tagindex} maps the script's output back to the paper, one row per tag prefix.

\begin{table}[t]
\centering\small
\renewcommand{\arraystretch}{1.15}%
\begin{tabular}{>{\raggedright\arraybackslash}p{0.13\dimexpr\linewidth-8\tabcolsep\relax}%
                >{\raggedright\arraybackslash}p{0.06\dimexpr\linewidth-8\tabcolsep\relax}%
                >{\raggedright\arraybackslash}p{0.46\dimexpr\linewidth-8\tabcolsep\relax}%
                >{\raggedright\arraybackslash}p{0.27\dimexpr\linewidth-8\tabcolsep\relax}}
\toprule
tag prefix & \# & what the checks certify & where the tags appear \\
\midrule
\texttt{I:} & 6 & Regime I certificate: $q-1$ factorization, ML pieces, value,
  witness moments & the Regime I part of the proof of the map (\cref{app:proofs}) \\
\texttt{II:} & 14 & Regime II: factorization, value, witness moments, and the sign/root
  auxiliaries of the case analysis & the Regime II part \\
\texttt{IIIa:} & 12 & plateau: certificate parameters ($bc=1$, $k_2$), factorization,
  value, witness, window algebra ($\tau\le b\iff\kappa\le\ks$) & the Regime IIIa part \\
\texttt{0:} & 8 & the $t=0$ endpoint certificate and witness & the endpoint proof
  (\cref{app:endpointpf}) \\
(gluing) & 4 & boundary agreement: II${}={}$I at $t=c(\kappa)$, IIIa${}={}$I at
  $t=b(\kappa)$, $\kc(b)=\kc(c)=\kappa$ & the coverage-and-gluing paragraph of the
  map's proof (\cref{app:proofmap}) \\
\texttt{T2-*:} & 8 & two-sided certificates: the Chebyshev regime and the binding regime
  of \cref{thm:twosided} & the two-sided proof \\
\bottomrule
\end{tabular}
\caption{The $52$ checks of \texttt{scripts/verify\_theorem\_symbolic.py},
grouped by tag prefix. Each check re-verifies in exact arithmetic an identity asserted,
with its tag in brackets, in the corresponding proof. The battery is redundant assurance,
not a dependency (\cref{app:register}).}
\label{tab:tagindex}
\end{table}

\subsection{A certificate, verbatim, and what the verifier does with it}
\label{app:walkthrough}

To make ``machine-verified'' concrete, here is the complete certificate artifact for the
Regime II instance $(t,\kappa)=(2,3)$ of \cref{app:worked-II} (the actual JSON file
shipped in the repository, unedited), followed by the independent verifier's output on
it. The reader can match every field against the hand calculation of
\cref{app:worked-II}.

The format: \texttt{constraints} lists the moment functionals $g_i$, relations, and values
$c_i$; \texttt{dual\_multipliers} are the $\lambda_i$ (exact rationals; the redundant
\texttt{q\_poly} is cross-checked against $\sum_i\lambda_ig_i$); each entry of
\texttt{pieces} carries a canonical domain, and per block a \texttt{multiplier} drawn from
that domain's allowed generators, a polynomial \texttt{basis} $z$, and an exact
\texttt{gram} matrix $G$, asserting the identity
$\text{target}=\sum_j\text{mult}_j\cdot z_j^\top G_jz_j$ with target $q-1$ on the event
piece and $q$ on the support piece; \texttt{extremal} is the tightness witness;
\texttt{provenance} records solver, tolerances, wall time, and git revision (the
\texttt{status} field is written by the pipeline only after the fresh-process verification
passes, and is itself re-checked, not trusted).

\begin{scriptsize}
\begin{verbatim}
{
 "lemma_id": "kurtosis_tail",
 "params": {
  "t": "2",
  "kappa": "3"
 },
 "objective": {
  "sense": "sup",
  "type": "probability"
 },
 "constraints": [
  {
   "poly": "1",
   "op": "==",
   "value": "1"
  },
  {
   "poly": "x",
   "op": "==",
   "value": "0"
  },
  {
   "poly": "x**2",
   "op": "==",
   "value": "1"
  },
  {
   "poly": "x**4",
   "op": "<=",
   "value": "3"
  }
 ],
 "provenance": {
  "source": "closed form (claims L01/L02/L03, regime II)",
  "date": "2026-07-02"
 },
 "status": "CERTIFIED",
 "claim": "P(x >= 2) <= 2/11 over C(3) [Regime II]",
 "dual_multipliers": [
  "1/121",
  "0",
  "-6/121",
  "9/121"
 ],
 "q_poly": "9*x**4/121 - 6*x**2/121 + 1/121",
 "bound_value": "2/11",
 "pieces": [
  {
   "role": "event",
   "domain": {
    "type": "half_line_ge",
    "a": "2"
   },
   "blocks": [
    {
     "multiplier": "1",
     "basis": [
      "x - 2",
      "x**2 - 4*x + 4"
     ],
     "gram": [
      [
       "210/121",
       "0"
      ],
      [
       "0",
       "9/121"
      ]
     ]
    },
    {
     "multiplier": "x - 2",
     "basis": [
      "1",
      "x - 2"
     ],
     "gram": [
      [
       "24/11",
       "0"
      ],
      [
       "0",
       "72/121"
      ]
     ]
    }
   ]
  },
  {
   "role": "support",
   "domain": {
    "type": "real_line"
   },
   "blocks": [
    {
     "multiplier": "1",
     "basis": [
      "1",
      "x",
      "x^2"
     ],
     "gram": [
      [
       "1/121",
       "0",
       "-3/121"
      ],
      [
       "0",
       "0",
       "0"
      ],
      [
       "-3/121",
       "0",
       "9/121"
      ]
     ]
    }
   ]
  }
 ],
 "extremal": {
  "atoms": [
   "2",
   "sqrt(3)/3",
   "-sqrt(3)/3"
  ],
  "weights": [
   "2/11",
   "9/22 - 2*sqrt(3)/11",
   "2*sqrt(3)/11 + 9/22"
  ]
 }
}
\end{verbatim}
\end{scriptsize}

\noindent The independent verifier (\texttt{python -m forge.verify <file>}) re-derives
everything from this file alone and prints:

\begin{scriptsize}
\begin{verbatim}
[ok] V1-objective: sup of probability
  [ok] V1-constraints: 4 moment constraints incl. normalization
  [ok] V2-q: q = 9*x**4/121 - 6*x**2/121 + 1/121
  [ok] V3-pieces: roles ['event', 'support']
  [ok] V5-psd[event#0]: ldlt: LDLT pivots [210/121, 9/121] | charpoly: charpoly coefficients all >= 0
  [ok] V5-psd[event#1]: ldlt: LDLT pivots [24/11, 72/121] | charpoly: charpoly coefficients all >= 0
  [ok] V4-identity[event]: exact
  [ok] V5-psd[support#0]: ldlt: LDLT pivots [1/121, 0, 0] | charpoly: charpoly coefficients all >= 0
  [ok] V4-identity[support]: exact
  [ok] V6-signs: 
  [ok] V7-bound: bound = 2/11 (upper)
  [ok] V8-extremal: feasible witness attains bound exactly (3 atoms)
PASS VERIFIED-TIGHT: P(x >= 2) <= 2/11 over C(3) [Regime II]  bound=2/11
\end{verbatim}
\end{scriptsize}

\noindent Note the two zero pivots in the support Gram's $LDL^\top$ (the rank-one boundary
structure computed by hand in \cref{app:worked-II}) handled by the zero-pivot/zero-row
rule, and the final verdict \textsc{verified-tight}: the witness was re-checked feasible
(exact moment arithmetic, including the quadratic surds) and its tail mass re-computed to
equal the bound $2/11$ exactly.

\subsection{Computational provenance, cost, and the anatomy of the one failure}
\label{app:provenance}

\begin{center}\small
\begin{tabular}{lrrr}
\toprule
cell & certs emitted & median solve (s) & max solve (s) \\
\midrule
cantelli & 4 & 0.000 & 0.000 \\
kurtosis\_tail & 48 & 0.073 & 0.088 \\
markov & 4 & 0.000 & 0.000 \\
paley\_zygmund & 6 & 0.000 & 0.000 \\
skew\_kurtosis\_tail & 60 & 0.067 & 0.140 \\
three\_moment\_tail & 9 & 0.057 & 0.172 \\
\bottomrule
\end{tabular}
\captionof{table}{Per-cell \emph{emitted}-certificate counts --- including the ones later
rejected by the verifier; certified counts are in \cref{app:certified} --- and Clarabel solve times (the exact rounding,
projection, and fresh-process verification dominate the wall clock; see text).
Environment: Python 3.13.14, sympy 1.14.0, cvxpy 1.9.2, Clarabel 0.11.1, numpy 2.4.6, scipy 1.17.1, mpmath 1.3.0; Windows 10, single core; solver tolerances
\texttt{tol\_gap\_abs} $=$ \texttt{tol\_gap\_rel} $= 10^{-11}$,
\texttt{tol\_feas} $=10^{-11}$ (SCS fallback: \texttt{eps} $=10^{-9}$).}
\label{tab:provenance}
\end{center}

\paragraph{Where the time goes.} The semidefinite programs are tiny (Gram blocks of size
$\le3$; a dual with $\le5$ multipliers) and Clarabel solves each in milliseconds
(\cref{tab:provenance}); the wall clock is dominated by exact arithmetic: the rational
projection step of the rounding route, the symbolic PSD checks on quadratic-surd Gram
matrices (IIIa), and one fresh Python/sympy process per verification ($\sim$1\,s of
interpreter startup each, by design: fresh-process verification is a trust boundary, not
an optimization target). The full production run (131 instances: solve, cross-check,
round, verify, tables) completes in tens of minutes on a single desktop core.

\paragraph{Anatomy of the failed instance.} The single kurtosis-cell failure,
$(t,\kappa)=(\tfrac12,\tfrac32)$, sits exactly on $\kappa=\ks$, where the plateau window
degenerates to the point $\tau=b=1/\sqrt2$ and the optimal Gram matrices are doubly
singular. The rounding route's $\varepsilon$-retreat did emit a candidate certificate,
but with one coefficient corrupted: the exact-recognition step, applied to a float at this
degenerate optimum, produced a wild algebraic number (a product of fractional prime powers)
instead of a rational, breaking the polynomial identity. The independent verifier rejected
it at its identity check (V4), printing the nonzero residual. We keep this failure in the
tables rather than patching around it because it is the architecture working as designed:
a \emph{producer} bug, mis-recognition under degeneracy, was isolated by a
\emph{checker} that shares none of the producer's code path, which is the entire argument
for having an independent verifier. (The instance's \emph{value} is not in doubt: it
agrees with the independent LP to $3\cdot10^{-8}$ and lies on the closure of both adjacent
regimes' formulas; only its certificate is missing.)

\section{Certified-instance tables}
\label{app:tables}

\cref{tab:instances} lists every kurtosis-cell grid instance with its exact value and the
independent verifier's verdict; the machine-readable original
(\texttt{results/constants.csv}), the certificate files, and the extremal-distribution
gallery accompany the paper. Certification and value-validation are separate axes
(\cref{tab:benchmarks}): the failed instances' \emph{values} still agree with the
independent LP; what they lack is an exact certificate.

{\small
\renewcommand{\arraystretch}{1.12}
\begin{table}[!htb]
\centering
\caption{All 48 kurtosis-cell grid instances, grouped by $\kappa$ (shown once per block): regime, exact value,
float value, and independent-verifier verdict. \textsc{tight} = certificate and
extremal witness both verified (\textsc{verified-tight}); \textsc{bound}($\varepsilon$)
= rational $\varepsilon$-retreat certificate of value${}+\varepsilon$,
$\varepsilon\le10^{-6}$ (Regime IIIb; exact witness value from system (S) matches the
float below to solver precision); \emph{fail} = rounding failure, not certified.
Exact values: I $=1/(1+t^2)$; II $=(\kappa-1)/((t^2-1)^2+\kappa-1)$;
IIIa $=\frac12(1+\sqrt{(\kappa-1)/(\kappa+3)})$; IIIb = root of (S).}
\label{tab:instances}
\begin{minipage}[t]{0.48\linewidth}
\centering
\resizebox{\linewidth}{!}{%
\begin{tabular}{ccllrl}
\toprule
$\kappa$ & $t$ & regime & $V_1$ (exact) & $V_1$ (float) & verdict \\
\midrule
$\frac{5}{4}$ & $\frac{1}{2}$ & IIIa & $\frac{\sqrt{17}}{34} + \frac{1}{2}$ & 0.621267813 & \textsc{tight} \\
 & $\frac{4}{5}$ & I & $\frac{25}{41}$ & 0.609756098 & \textsc{tight} \\
 & $1$ & I & $\frac{1}{2}$ & 0.500000000 & \textsc{tight} \\
 & $\frac{6}{5}$ & I & $\frac{25}{61}$ & 0.409836066 & \textsc{tight} \\
 & $\frac{3}{2}$ & II & $\frac{4}{29}$ & 0.137931034 & \textsc{tight} \\
 & $2$ & II & $\frac{1}{37}$ & 0.027027027 & \textsc{tight} \\
 & $3$ & II & $\frac{1}{257}$ & 0.003891051 & \textsc{tight} \\
 & $4$ & II & $\frac{1}{901}$ & 0.001109878 & \textsc{tight} \\
\addlinespace[4pt]
$\frac{3}{2}$ & $\frac{1}{2}$ & IIIb & root of (S) & 0.667242152 & \emph{fail} \\
 & $\frac{4}{5}$ & I & $\frac{25}{41}$ & 0.609756098 & \textsc{tight} \\
 & $1$ & I & $\frac{1}{2}$ & 0.500000000 & \textsc{tight} \\
 & $\frac{6}{5}$ & I & $\frac{25}{61}$ & 0.409836066 & \textsc{tight} \\
 & $\frac{3}{2}$ & II & $\frac{8}{33}$ & 0.242424242 & \textsc{tight} \\
 & $2$ & II & $\frac{1}{19}$ & 0.052631579 & \textsc{tight} \\
 & $3$ & II & $\frac{1}{129}$ & 0.007751938 & \textsc{tight} \\
 & $4$ & II & $\frac{1}{451}$ & 0.002217295 & \textsc{tight} \\
\addlinespace[4pt]
$2$ & $\frac{1}{2}$ & IIIb & root of (S) & 0.726255879 & \textsc{bound}($\varepsilon$) \\
 & $\frac{4}{5}$ & I & $\frac{25}{41}$ & 0.609756098 & \textsc{tight} \\
 & $1$ & I & $\frac{1}{2}$ & 0.500000000 & \textsc{tight} \\
 & $\frac{6}{5}$ & I & $\frac{25}{61}$ & 0.409836066 & \textsc{tight} \\
 & $\frac{3}{2}$ & I & $\frac{4}{13}$ & 0.307692308 & \textsc{tight} \\
 & $2$ & II & $\frac{1}{10}$ & 0.100000000 & \textsc{tight} \\
 & $3$ & II & $\frac{1}{65}$ & 0.015384615 & \textsc{tight} \\
 & $4$ & II & $\frac{1}{226}$ & 0.004424779 & \textsc{tight} \\
\bottomrule
\end{tabular}}
\end{minipage}%
\hfill
\begin{minipage}[t]{0.48\linewidth}
\centering
\resizebox{\linewidth}{!}{%
\begin{tabular}{ccllrl}
\toprule
$\kappa$ & $t$ & regime & $V_1$ (exact) & $V_1$ (float) & verdict \\
\midrule
$3$ & $\frac{1}{2}$ & IIIb & root of (S) & 0.789626515 & \textsc{bound}($\varepsilon$) \\
 & $\frac{4}{5}$ & I & $\frac{25}{41}$ & 0.609756098 & \textsc{tight} \\
 & $1$ & I & $\frac{1}{2}$ & 0.500000000 & \textsc{tight} \\
 & $\frac{6}{5}$ & I & $\frac{25}{61}$ & 0.409836066 & \textsc{tight} \\
 & $\frac{3}{2}$ & I & $\frac{4}{13}$ & 0.307692308 & \textsc{tight} \\
 & $2$ & II & $\frac{2}{11}$ & 0.181818182 & \textsc{tight} \\
 & $3$ & II & $\frac{1}{33}$ & 0.030303030 & \textsc{tight} \\
 & $4$ & II & $\frac{2}{227}$ & 0.008810573 & \textsc{tight} \\
\addlinespace[4pt]
$6$ & $\frac{1}{2}$ & I & $\frac{4}{5}$ & 0.800000000 & \textsc{tight} \\
 & $\frac{4}{5}$ & I & $\frac{25}{41}$ & 0.609756098 & \textsc{tight} \\
 & $1$ & I & $\frac{1}{2}$ & 0.500000000 & \textsc{tight} \\
 & $\frac{6}{5}$ & I & $\frac{25}{61}$ & 0.409836066 & \textsc{tight} \\
 & $\frac{3}{2}$ & I & $\frac{4}{13}$ & 0.307692308 & \textsc{tight} \\
 & $2$ & I & $\frac{1}{5}$ & 0.200000000 & \textsc{tight} \\
 & $3$ & II & $\frac{5}{69}$ & 0.072463768 & \textsc{tight} \\
 & $4$ & II & $\frac{1}{46}$ & 0.021739130 & \textsc{tight} \\
\addlinespace[4pt]
$9$ & $\frac{1}{2}$ & I & $\frac{4}{5}$ & 0.800000000 & \textsc{tight} \\
 & $\frac{4}{5}$ & I & $\frac{25}{41}$ & 0.609756098 & \textsc{tight} \\
 & $1$ & I & $\frac{1}{2}$ & 0.500000000 & \textsc{tight} \\
 & $\frac{6}{5}$ & I & $\frac{25}{61}$ & 0.409836066 & \textsc{tight} \\
 & $\frac{3}{2}$ & I & $\frac{4}{13}$ & 0.307692308 & \textsc{tight} \\
 & $2$ & I & $\frac{1}{5}$ & 0.200000000 & \textsc{tight} \\
 & $3$ & I & $\frac{1}{10}$ & 0.100000000 & \textsc{tight} \\
 & $4$ & II & $\frac{8}{233}$ & 0.034334764 & \textsc{tight} \\
\bottomrule
\end{tabular}}
\end{minipage}
\end{table}
}

\noindent The 13 uncertified skewness-pinned instances, as
$(\gamma,\kappa,t)$ triples (all rounding failures on numerically degenerate
optima, with LP-agreement of the value intact): $(-1,\,6,\,1)$, $(-1,\,6,\,2)$, $(- \frac{1}{2},\,2,\,1)$, $(- \frac{1}{2},\,2,\,3)$, $(- \frac{1}{2},\,3,\,1)$, $(- \frac{1}{2},\,6,\,2)$, $(0,\,2,\,\frac{3}{2})$, $(0,\,2,\,3)$, $(\frac{1}{2},\,2,\,1)$, $(\frac{1}{2},\,3,\,2)$, $(1,\,2,\,1)$, $(1,\,3,\,\frac{3}{2})$, $(1,\,6,\,3)$.

\end{document}